%% file: corepaper.tex
\documentclass{thielstyle-ams}

\newcommand{\hypalg}{\mathfrak{u}}

\DeclareMathOperator{\Pim}{Pim}
\DeclareMathOperator{\uPim}{\underline{Pim}}
\DeclareMathOperator{\uIrr}{\underline{Irr}}

\input{gwyns_definitions}

\usepackage{calc} 

\begin{document}  

\title[Cores of graded algebras with triangular decomposition]{Cores of graded algebras with triangular decomposition}

\author{Gwyn Bellamy}

\address{School of Mathematics and Statistics, University Gardens, University of Glasgow, Glasgow, G12 8QW, UK}
\email{gwyn.bellamy@glasgow.ac.uk}

\author{Ulrich Thiel}
\address{School of mathematics and Statistics, University of Sydney, NSW 2006, Australia}
\email{ulrich.thiel@sydney.edu.au}

\begin{abstract}
We consider self-injective finite-dimensional graded algebras admitting a triangular decomposition. In the preceding paper \cite{hwtpaper} we have shown that the graded module category of such an algebra is a highest weight category and has tilting objects in the sense of Ringel. In this paper we focus on the degree zero part of the algebra, the \textit{core} of the algebra. We show that the core captures essentially all relevant information about the graded representation theory. Using  tilting theory, we show that the core is cellular. We then describe a canonical construction of a highest weight cover, in the sense of Rouquier, of this cellular algebra using a finite subquotient of the highest weight category. Thus, beginning with a self-injective graded algebra admitting a triangular decomposition, we canonically construct a quasi-hereditary algebra which encodes key information, such as the graded multiplicities, of the original algebra. Our results are general and apply to a wide variety of examples, including restricted enveloping algebras, Lusztig’s small quantum groups, hyperalgebras, finite quantum groups, and restricted rational Cherednik algebras. We expect that the cell modules and quasi-hereditary algebras introduced here will provide a new way of understanding these important examples.
\end{abstract}

\blfootnote{Version from \today}

\maketitle
\tableofcontents

\section{Introduction}

The goal of this paper, and the preceding paper \cite{hwtpaper}, is to develop new structures in the representation theory of a class of algebras commonly encountered in algebraic Lie theory: finite-dimensional $\mathbb{Z}$-graded algebras $A$ which admit a \textit{triangular decomposition}, i.e., a 
vector space decomposition 
$$
A^- \otimes T \otimes A^+ \overset{\sim}{\longrightarrow} A
$$
into graded subalgebras given by the multiplication map, where we assume that $A^-$ is concentrated in negative degree, $T$ in degree zero, and $A^+$ in positive degree.

There are a variety of examples of such algebras:
\begin{enumerate}\label{eq:VIPlist}
\item Restricted enveloping algebras $\overline{U}(\mf{g}_K)$;
\item Lusztig's small quantum groups $\mathbf{u}_{\zeta}(\mf{g})$, at a root of unity $\zeta$; \item Hyperalgebras $\hypalg_r(\mf{g}) \dopgleich \mathrm{Dist}(G_r)$ on the Frobenius kernel $G_r$; 
\item Finite quantum groups $\mc{D}$ associated to a finite group $G$;
\item Restricted rational Cherednik algebras $\overline{\H}_{\bc}(W)$ at $t = 0$;
\item RRCAs $\overline{\H}_{1,\bc}(W)$ at $t = 1$ in positive characteristic.
\end{enumerate}
We refer to the above list as the \textit{VIP examples}. For more details, see \cite{hwtpaper} and the references therein. The representation theory of these algebras has important applications to other areas of mathematics. For instance, to symplectic algebraic geometry \cite{Baby}, \cite{Singular}, \cite{BellamyCounting}, \cite{BellSchedThiel}, to algebraic combinatorics \cite{GriffMac}, \cite{TomaszP}, and to algebraic groups in positive characteristic \cite{Jantzen}, \cite{AJS}. The applications mostly derive in one way or another from computing the graded character of irreducible modules.

In this article, we propose a fundamentally new approach to the problem of computing these characters. Our approach focuses on understanding the subalgebra $A_0$ of $A$ formed by the elements of degree zero, which we call the \textit{core} of $A$. In Section \ref{core_rep_general} we show that $A_0$ captures, in a precise sense, essentially all information about the graded representation theory of $A$. The advantage of considering $A_0$ is that it possess additional structure that $A$ itself need not have. In Sections \ref{cellular_structure} and \ref{sec:cover}, under mild assumptions on $A$, we construct two such structures: 
\begin{itemize}
	\item a \textit{cellular structure} on $A_0$, in the sense of Graham--Lehrer \cite{Graham-Lehrer-Cellular}, 
	\item a \textit{quasi-hereditary} algebra covering $A_0$, in the sense of Rouquier \cite{RouquierQSchur}.
\end{itemize}
The proof of these results build on the foundational material of~\cite{hwtpaper}. We explain these results in more detail below.   

\subsection{Cellularity}

In the seminal paper \cite{Graham-Lehrer-Cellular}, Graham and Lehrer introduced \textit{cellular algebras} to representation theory. Providing a means to combinatorially describe the representation theory of algebras, cellularity has become an important and influential concept in representation theory. This can be see in the concentrated effort, ever since the appearance of \cite{Graham-Lehrer-Cellular}, to find new examples of cellular algebras--our work provides \textit{several} infinite families of cellular algebras that have never been considered in the literature before. 

More generally, by dropping the anti-involution from the definition, Du and Rui \cite{DuRui} introduced a broader class of algebras called \textit{standardly based} algebras. Abstractly, it can be shown that every finite-dimensional algebra is actually standardly based, see Koenig--Xi \cite[Section 5]{KonigXiCellular}. However, exhibiting an explicit standard basis is computationally very difficult and provides, through the associated cells and cell modules, meaningful combinatorial information about the representation theory of the algebra.  

We obtain cellularity of $A_0$ from a very general result, proved in Appendix \ref{cellular_appendix}, concerning the endomorphism algebra of a tilting object in a highest weight category:

\begin{thm} \label{hwc_cellular}
Let $\mathcal{C}$ be a highest weight category (possibly with infinitely many simple objects), equipped with enough tilting objects\footnote{See Appendix \ref{cellular_appendix} for the precise definition.}. Let $T \in \mathcal{C}$ be a tilting object.
\begin{enumerate}
\item The decomposition of $T$ into indecomposable tilting objects defines a standard datum for $\End_{\mathcal{C}}(T)$. 
\item If $\mathcal{C}$ is equipped with a duality $\bbD$ such that $\bbD(T) \simeq T$, then $\bbD$ induces an anti-involution on $\End_{\mathcal{C}}(T)$, making it into a cellular algebra.
\end{enumerate}
\end{thm}

The proof of this theorem is a modification of very recent work of Andersen--Stroppel--Tubbenhauer \cite{AST}, who prove this in the case of quantum groups. Unfortunately, one essential ingredient in \textit{loc. cit.} is the notion of \textit{weight spaces}, which is not available in an arbitrary highest weight category (in particular, such a structure does not exist for $\mathcal{G}(A)$). We explain the necessary modifications to their arguments in Appendix \ref{cellular_appendix}. By results of \cite{hwtpaper}, the categories $\mathcal{G}(A)$ satisfy the assumptions of Theorem \ref{hwc_cellular}. Applying Theorem \ref{hwc_cellular} to the tilting object $A$ in $\mathcal{G}(A)$, we obtain (see Theorem \ref{core_cellular}):

\begin{thm}\label{thm:standardcellularintro} 
Assume that $T$ is semisimple and $A$ is self-injective.
\begin{enumerate}
\item The decomposition of $A$ into indecomposable tilting modules provides a \textit{canonical} standard datum for $A_0$.
\item If $A$ is graded Frobenius and admits a triangular anti-involution, then $A_0$ is a cellular algebra. 
\end{enumerate}

\end{thm}

The notion of a \textit{triangular anti-involution} was introduced in \cite{hwtpaper}. Such an involution induces a duality $\bbD$ on $\mathcal{G}(A)$, and in case $A$ is graded Frobenius this fixes the tilting object $A$, see Corollary 6.4 of \textit{loc. cit.} We have shown in \textit{loc. cit.} that all VIP examples listed above satisfy the assumptions of Theorem \ref{thm:standardcellularintro}, giving a canonical standard datum for their cores. The only obstruction to being cellular is the existence of a triangular anti-involution. Again as noted in \textit{loc. cit.}, this exists for all VIP examples, except for restricted rational Cherednik algebras associated to \textit{non-real} reflection groups. Regardless, one can now define cell modules and cells for all VIP examples. Multiplicities of simple modules in cell modules leads to the \textit{decomposition matrix} for $A_0$. This raises the obvious problem:
\begin{itemize}
	\item Compute the decomposition matrix for $A_0$. 
\end{itemize}
We show that this problem is equivalent to computing the character of the simple graded $A$-modules.

\subsection{Highest weight covers}
We show, in Proposition \ref{A0_quasi_hereditary}, that the core $A_0$ is quasi-hereditary if and only if $A$ is semi-simple. Thus, the core is (essentially) never quasi-hereditary. The main goal of this article is to ``resolve'' this ``singular'' algebra by finding a highest weight cover, in the sense of Rouquier \cite{RouquierQSchur}. We recall that a highest weight cover of $A_0$ is a highest weight category $\mc{B}$ (necessarily with finitely many simple objects), together with a projective object $P \in \mc{B}$ such that $A_0 \simeq \End_{\mc{B}}(P)^{\mathrm{op}}$ and the corresponding exact functor 
$$
F = \Hom_{\mc{B}}(P,-) \from \mc{B} \to \mc{M}(A_0)
$$
is fully faithful on projectives. It is shown in \cite{RouquierQSchur} that a highest weight cover, if it exists, is essentially unique. 

In Section \ref{sec:cover} we show that $A_0$, under some mild assumptions satisfied by all VIP examples, admits a highest weight cover. This cover is explicitly constructed by taking a subquotient $\mc{G}_{\lbrack N \rbrack}(A)$ of the highest weight category $\mc{G}(A)$. First, by bounding from above the degrees of the simple composition factors that are allowed to occur in a module $M$, we get for each integer $d$ a Serre subcategory $\mc{G}_{\le d}(A)$ of $\mc{G}(A)$. Since these subcategories come from an ideal in the poset $\Irr \mc{G}(T)$, they are also highest weight categories. More importantly, for each positive integer $d$, the subfactors 
$$
\mc{G}_{[d]}(A) := \mc{G}_{\le d}(A) / \mc{G}_{< 0}(A)
$$
are also highest weight categories. These categories have finitely many simple objects and enough projectives. Thus, there exists a basic quasi-hereditary algebra $C_d$ such that $\mc{G}_{[d]}(A)$ is equivalent to the category $\mc{M}(C_d)$ of finite dimensional $C_d$-modules. In this way, $A$ gives rise to an infinite family $\{ C_d \}$ of quasi-hereditary algebras. 

In general, the algebras $C_d$ are very difficult to describe. A simple motivating example is given by $A = \C[x,y] / (x^2, y^2)$. Here $\mc{G}(A)$ is equivalent to the category of finite-dimensional modules for the quiver 
$$
\begin{tikzcd}
\cdots & e_{i-2} \arrow[bend left]{r}{y_{i-2}}  & e_{i-1} \arrow[bend left]{l}{x_{i-1}} \arrow[bend left]{r}{y_{i-1}} & e_{i} \arrow[bend left]{l}{x_{i}} \arrow[bend left]{r}{y_{i}} & e_{i+1} \arrow[bend left]{l}{x_{i+1}} & \cdots
\end{tikzcd}
$$
with relations $x_{i+1} \circ y_i = y_{i-1} \circ x_i$ and $y_{i+1} \circ y_i = x_{i+1} \circ x_i = 0$. In this case, the quotient category $\mc{G}_{[d]}(A)$ is equivalent to the highest weight category $\mathrm{Perv}_{\C}(\mathbb{P}^d)$ of perverse sheaves on projective space, stratified with respect to the usual Schubert stratification. 
 
Assume that $A$ is graded symmetric. Then the socle of $A^+$ equals the top non-zero graded piece $A_N^+$. Choosing $d \ge N$, we set $\ell \dopgleich N-d$ and let
$$
A \langle \ell \rangle \dopgleich A \lbrack 0 \rbrack \oplus A \lbrack 1 \rbrack \oplus \ldots \oplus A \lbrack \ell \rbrack.  
$$
This is an object of $\mc{G}_{\lbrack d \rbrack}(A)$. In Theorem \ref{thm:A0elling}, we show that it is projective-injective and that
\[
\End_{\mc{G}_{\lbrack d \rbrack}(A)}(A \langle \ell \rangle ) = \End_{\mc{G}(A)}(A \langle \ell \rangle ) \ (=: B_\ell) \;.
\]
Since $A\lbrack 0,\ell \rbrack$ is projective in $\mc{G}_{\lbrack d \rbrack}(A)$, it defines an exact functor
\[
F_d \dopgleich \Hom_{\mc{G}_{\lbrack d \rbrack}(A)}(A\langle \ell \rangle , -) \from \mc{G}_{\lbrack d \rbrack}(A) \to \mc{M}(B_\ell) \;,
\]
where $\mc{M}(B_\ell)$ denotes the category of finite-dimensional $B_\ell$-modules. Our main result, Theorem \ref{thm:highestweightcover}, states:

\begin{thm}\label{thm:introcoverresult}
	If $A$ is graded symmetric, well-generated and ambidextrous, then $F_d$ is a highest weight cover of $B_\ell$.
\end{thm}

The assumptions of Theorem \ref{thm:introcoverresult} are not particularly restrictive. By \textit{well-generated} we mean that $A^\pm$ is generated as a $K$-algebra in degree $\pm 1$. The notion of \textit{ambidextrous} was introduced in \cite{hwtpaper}, and says that after swapping $A^-$ with $A^+$, the multiplication map $A^+ \otimes_K T \otimes_K A^- \to A$ is still an isomorphism. The assumptions may be relaxed somewhat, see Section \ref{sec:coversmain}, but \textit{all} the VIP examples mentioned at the beginning of the introduction satisfy them; see \textit{loc. cit.} for details. The proof of Theorem \ref{thm:highestweightcover} relies on a technical result, Lemma \ref{lem:soclAlhdsupport}, about the socle of certain graded modules over the algebras $A^+$ and $A^-$. Remarkably, though the statement is the same for both algebras, they play opposite roles. Namely, the result applied to $A^-$ implies that $F_d$ is faithful on projectives. In fact, it is faithful on all standardly filtered modules. On the other hand, when applied to $A^+$ it implies that $F_d$ is full on projectives. Taking $d=N$, we have $B_0 = A_0^\op$ and:

\begin{thm}
	If $A$ is graded symmetric, well-generated and ambidextrous, then 
	$$
	F_N \from \mc{G}_{\lbrack N \rbrack}(A) \to \mc{M}(A_0^{\mathrm{op}})
	$$
	is a highest weight cover.
\end{thm}

Though $F_N$ is $(-1)$-faithful, in the sense of \cite{RSVV}, it is not in general $0$-faithful. When $d < N$, there is no natural functor $\mc{G}_{[d]}(A) \rightarrow \mc{M}(A_0)$, and when $d > N$, one can show that $F_d$ is \textit{not} a highest weight cover. See Section \ref{0_faithful} for more details on $0$-faithfulness. Our results raise a number of important questions: 
\begin{itemize}
	\item For which algebras $A$ are the highest weight categories $\mc{G}_{\lbrack d \rbrack}(A)$ Koszul? 
\end{itemize}
The categories $\mc{G}(A)$, for $A$ self-injective, are Ringel self-dual since the tilting modules are precisely the projective modules. This no longer holds for the subquotients $\mc{G}_{\lbrack d \rbrack}(A)$. 
\begin{itemize}
	\item What is the Ringel dual of $\mc{G}_{\lbrack d \rbrack}(A)$? 
\end{itemize}
We hope to answer these questions in future work. Identifying the infinite family of quasi-hereditary algebras $C_d$ seems to be a difficult problem. We note that in the important case of the VIP examples, we do not as yet have any explicit description of these quasi-hereditary algebras. Toy examples are given in section \ref{cover_examples}.

\subsection*{Cellularity of $A$}

Based on Theorem \ref{thm:standardcellularintro}, it is natural to expect that the algebra $A$ itself is cellular, at least when equipped with a triangular anti-involution. In Section \ref{A_not_cellular} we explain that a key obstruction to cellularity is given by the blockwise \word{rank one property} of the decomposition matrix. This implies that most VIP examples are in fact \textit{not} cellular. More precisely (see Proposition \ref{vip_not_cellular}):

\begin{prop}
	Let $A$ be one of the following VIP examples:
	\begin{enumerate}
		\item Restricted enveloping algebras $\overline{U}(\mf{g}_K)$;
		\item Lusztig's small quantum groups $\mathbf{u}_{\zeta}(\mf{g})$, at a root of unity $\zeta$ of odd order; 
		\item Restricted rational Cherednik algebras $\overline{\H}_{\bc}(W)$ at $t = 0$, with $\bc$ non-generic in case all factors of $W$ are among the groups $G(m,1,n)$ and $G_4$; 
	\end{enumerate}
	Then $A$ is \textnormal{not} cellular.
\end{prop}

For the remaining VIP examples, see Remark \ref{hyperalg_rank1}.

\begin{rem}
	We wish to note one other interesting application of Theorem \ref{hwc_cellular}. Let $\mc{H}_q(W)$ be the cyclotomic Hecke algebra, defined over $\C$, associated to the complex reflection group $W$ at parameter $q$. By \cite[Theorem 5.15]{GGOR}, there exists a parameter $\bc$ and a projective object $P_{\mathrm{KZ}}$ in category $\mc{O}$ for the corresponding rational Cherednik algebra $\H_{1,\bc}(W)$, at $t = 1$, such that
	$$
	\mc{H}_q(W) = \End_{\H_{1,\bc}(W)}(P_{\mathrm{KZ}})^{\op}.
	$$
	Moreover, $P_{\mathrm{KZ}}$ is tilting in $\mc{O}$ by \cite[Proposition 5.21]{GGOR}. Therefore, by Theorem \ref{hwc_cellular}, the decomposition of $P_{\mathrm{KZ}}$ into indecomposables defines an explicit standard datum for $\mc{H}_q(W)$ . In particular, when $W$ is a Coxeter group there is an anti-involution on $\H_{1,\bc}(W)$ under which $P_{\mathrm{KZ}}$ is self-dual. In this case, Theorem \ref{hwc_cellular} shows that $\mc{H}_q(W)$ is a cellular algebra. This gives a short case-free proof of a special case of an important theorem by Geck \cite{Geck:2007ki}. It may be an interesting problem to investigate the cell modules and cells we obtain for the Hecke algebra $\mc{H}_q(W)$ for $W$ any complex reflection group.
\end{rem}

\begin{rem}
	In case of restricted rational Cherednik algebras $\ol{\H}_\bc(W)$ attached to a complex reflection group $W$, we have $\bbC W \subs \ol{\H}_\bc(W)_0$. We expect that the cells for $\ol{\H}_\bc(W)_0$ obtained from the canonical standard datum should give rise to a notion of cells in $W$. The Gaudin operators used by Bonnafé–Rouquier \cite{BonnafeRouquier} to construct cellular characters for an arbitrary complex reflection group $W$ are contained in the core of the (non-restricted) rational Cherednik algebra $\H_\bc(W)$. This suggests a close connection, when $W$ is a Coxeter group, between our notion of cells and Kazhdan–Lusztig cells.
\end{rem}

\subsection*{Outline}
In Section \ref{notations_2} we recall the relevant notation and conventions from \cite{hwtpaper}. This includes a precise statement of the three key results from \textit{loc. cit.} mentioned at the beginning of the introduction. 

In Section \ref{core_rep_general}, we describe precisely in what sense the core captures the graded representation theory of $A$. We do not assume in this section that $A$ admits a triangular decomposition. 

Using the general theory of endomorphism algebras of tilting objects developed in Appendix \ref{cellular_appendix}, we prove in Section \ref{cellular_structure} that the core $A_0$ is cellular. Theorem \ref{core_cellular} provides further details about the cellular structure. In Section \ref{sec_cell_modules}, we discuss cell modules and in Proposition \ref{cell_param_match} we show how the parametrization of simple $A_0$-modules via the cellular structure is related to the natural parametrization obtained from $\mathcal{G}(A)$. 

In Section \ref{sec:cover} we construct the highest weight cover of $A_0$. This is broken up into several steps. The precise statement of the main results are summarized in Section \ref{sec:coversmain}. Then, in Section \ref{sec_filtered_pieces}, we take a closer look at filtered pieces of the highest weight category $\mathcal{G}(A)$. In Section \ref{sec:quasihereditary} we consider certain subquotients of $\mathcal{G}(A)$. Section \ref{cover_proof_section1} considers the endomorphism algebra of $A\langle \ell\rangle$. In Section \ref{cover_proof_section}, we prove our main result: that our construction gives a highest weight cover. This is done is slightly greater generality than is stated in the introduction. We show in Section \ref{sec_socle_assumption} that the setup considered in the introduction fits this more general framework. In Section \ref{cover_examples} we compute some explicit examples of highest weight covers, and in Section \ref{0_faithful} we address the problem of $0$-faithfulness.

\subsection*{Acknowledgements}

The first author was partially supported by EPSRC grant EP/N005058/1. The second author was partially supported by the DFG SPP 1489, by a Research Support Fund from the Edinburgh Mathematical Society, and by the Australian Research Council Discovery Projects grant no. DP160103897. We would especially like to thank S. Koenig for patiently answering several questions. We would also like to thank I. Losev and R. Rouquier for fruitful discussions. Moreover, we thank G. Malle, A. Henderson and O. Yacobi for comments on a preliminary version of this article.

\section{Notation} \label{notations_2}

We will use the same notation and conventions as in \cite{hwtpaper}. We recall the relevant material. Unless otherwise stated, all modules are \textit{left} modules and \textit{graded} always means $\bbZ$-graded. For a graded vector space $M$, we denote by $M_i$ the homogeneous component of degree $i$. We denote by $M \lbrack n \rbrack$ the \textit{right} shift of $M$ by $n \in \bbZ$, i.e., $M[n]_i = M_{i-n}$. So, if $M$ is concentrated in degree zero, then $M \lbrack n \rbrack$ is concentrated in degree $n$. The \word{support} of~$M$ is defined to be $\Supp M \dopgleich \lbrace i \in \bbZ \mid M_i \neq 0 \rbrace$.
 
 Let $A$ be a finite-dimensional graded algebra over a field~$K$. We denote by $\mathcal{M}(A)$ the category of finitely generated $A$-modules and by $\mathcal{G}(A)$ the category of graded finitely generated $A$-modules with morphisms preserving the grading. We use the symbol $\mathcal{C}$ to denote either of the two categories, i.e., $\mathcal{C} \in \lbrace \mathcal{M},\mathcal{G} \rbrace$. By $\Irr \mc{C}(A)$ we denote the set of isomorphism classes of simple objects of $\mc{C}(A)$. 
 
If $M$ is a graded vector space, we write $M^\op$ for the same vector space, but with grading reversed, i.e., 
$$
M^\op_i \dopgleich M_{-i} \;.
$$
We thus have $\Supp M^\op = - \Supp M$. With the reversed grading the opposite ring $A^\op$ of $A$ is again graded, see \cite[1.2.4]{Methods-of-graded-rings}. If $M$ is a (graded) \textit{right} $A$-module, then $M^\op$ is naturally a (graded) \textit{left} $A^\op$-module and vice versa. The assignment $M \mapsto M^\op$ with the identity on morphisms thus yields a natural identification between $\mathcal{C}(A^\op)$ and the category of finitely generated (graded) \textit{right} $A$-modules. For a $K$-vector space $M$ we denote by $M^* \dopgleich \Hom_K(M,K)$ its \word{dual}. If $M \in {\mathcal{C}}(A)$, then $M^*$ is naturally an object in $\mathcal{C}(A^{\op})$ with grading defined by 
$$
(M^*)_i \dopgleich \lbrace f \in M^* \mid f(M_j) = 0 \tn{ for all } j \neq i \rbrace \simeq M_i^*
$$
and $A^{\op}$-action on $M^*$ defined by $(a^\op f)(m) \dopgleich f(am)$, for $f \in M^*$, $a^\op \in A^\op$, and $m \in M$. With $\Hom_K(-,K)$ on morphisms, this defines a contravariant functor
$$
(-)^*:{ \mathcal{C}}(A) \rarr \mathcal{C}(A^{\op}) \;,
$$
called \word{standard duality}. Since $A$ is finite-dimensional, this functor is indeed a duality, i.e., $(-)^* \circ (-)^* \simeq \id_{\mathcal{C}(A)}$. It induces a bijection $\Irr \mathcal{C}(A) \simeq \Irr \mathcal{C}(A^{\op})$. We have $M^*\lbrack n \rbrack = M \lbrack n \rbrack^*$ and $\Supp M^* = \Supp M$. \\
 
We recall the definition of a triangular decomposition from \cite[Section 3]{hwtpaper} and the three main results proven in \textit{loc. cit.} Throughout, $A$ is a finite-dimensional graded algebra over a field $K$.

\begin{defn}\label{defn:triangle}
	A \word{triangular decomposition} of $A$ is a triple $\mathcal{T} = (A^-,T,A^+)$ of graded subalgebras of $A$ such that:
	\begin{enum_thm}
		\item the multiplication map $A^- \otimes_K T \otimes_K A^+ \rarr A$ is an isomorphism of vector spaces,
		\item $\Supp A^+ \subset \Z_{\ge 0}$, $\Supp A^- \subset \Z_{\le 0}$, and $T$ is concentrated in degree zero,
		\item \label{defn:connected} $A_0^- = K = A_0^+$,
		\item $A^+ T = T A^+$ and $A^- T = T A^-$ as subspaces of $A$,
		\item $T$ is a split $K$-algebra, i.e., $\End_K(\lambda) = K$ for all simple $T$-modules $\lambda$.
	\end{enum_thm}
\end{defn}

As discussed in \textit{loc. cit.}, all VIP examples listed in the introduction have such a decomposition. In order for $\mc{G}(A)$ to be highest weight, we assume here (as is done in the latter half of \textit{loc. cit.}) that $T$ is \textit{semisimple}. Once again, this assumption is satisfied by all VIP examples. For each $\lambda \in \Irr \mc{G}(T)$, there is a standard module $\Delta(\lambda)$ and a costandard module $\nabla(\lambda)$ in $\mc{G}(A)$, such that the head $L(\lambda)$ of $\Delta(\lambda)$ is a simple object of $\mc{G}(A)$ and the set $\{ L(\lambda) \  | \ \lambda \in \Irr \mc{G}(T) \}$ is a complete, irredundant, set of simple objects of $\mc{G}(A)$. Thus, $\lambda \mapsto L(\lambda)$ is a bijection between the simple objects of $\mc{G}(T)$ and the simple objects of $\mc{G}(A)$. Since $T$ is concentrated in degree zero, each simple graded $T$-module $\lambda$ is concentrated in a single degree, which we denote by $\deg \lambda$. We define a partial ordering $<$ on the indexing set $\Irr \mc{G}(T)$ by
$$
\lambda < \mu \Longleftrightarrow \deg \lambda < \deg \mu \;.
$$

The following theorem is \cite[Corollary 4.13]{hwtpaper}:

\begin{thm}
$\mathcal{G}(A)$ is a highest weight category with respect to the ordering $\leq$.
\end{thm} 

We recall that an object $S$ in $\mc{G}(A)$ is said to be \textit{tilting} if it has both a filtration by standard modules, and a filtration by costandard modules. In \cite[Theorem 5.1]{hwtpaper} we have shown:

\begin{thm}
The tilting objects in $\mc{G}(A)$ are precisely the objects which are both projective and injective. 
\end{thm}

Hence, if we assume that $A$ is self-injective, the indecomposable tilting objects are precisely the indecomposable projective objects. In \cite[Corollary 5.9]{hwtpaper} we have shown:

\begin{thm}
The category $\mathcal{G}(A)$ has indecomposable tilting objects in the sense of Ringel \cite{Ringel-Filtrations}: for each $\lambda$ there is an indecomposable tilting object $T(\lambda)$, uniquely characterized by the property that it has highest weight $\lambda$, an injection $\Delta(\lambda) \hookrightarrow T(\lambda)$, and a surjection $T(\lambda) \twoheadrightarrow \nabla(\lambda)$. All indecomposable tilting modules are obtained this way.
\end{thm}

As noted in \textit{loc. cit.} it is generally not true that $P(\lambda) = T(\lambda)$. Rather, there is a permutation $h$ on $\Irr \mathcal{G}(T)$ such that $P(\lambda) = T(\lambda^h)$; see \cite[Corollary 5.9]{hwtpaper}. As noted in \cite{hwtpaper}, all VIP examples are self-injective.

\section{Representation theory of the core} \label{core_rep_general}

In this section, $A$ can be an arbitrary finite-dimensional graded $K$-algebra. The degree zero part $A_0$ of $A$ is a subalgebra of $A$ which we call the \word{core} of $A$ (in \cite{GG-Graded-Artin} this is called the \word{initial subring}). The graded representation theory of $A$ and the ungraded representation theory of $A_0$ are related via the functor 
$$
(-)_0 \from \mc{G}(A) \to \mc{M}(A_0)
$$
sending $M$ to its subspace of degree zero, and the induction functor
$$
\Ind_{A_0}^A = A \otimes_{A_0} - \ \from \mc{M}(A_0) \to \mc{G}(A) \;.
$$
Notice that since $A = \bigoplus_{i \in \bbZ} A_i$ as $(A_0,A_0)$-bimodules, for each $M \in \mc{M}(A_0)$ we have a canonical direct sum decomposition 
\[
\Ind_{A_0}^A M = \bigoplus_{i \in \bbZ} A_i \otimes_{A_0} M
\]
as $A_0$-modules, allowing us to view $\Ind_{A_0}^A M$ as a \textit{graded} $A$-module.

\begin{lem} \label{ind_proj_adjunction}
For any $d \in \bbZ$ the shifted induction functor $(\Ind_{A_0}^A -)\lbrack d \rbrack: \mathcal{M}(A_0) \rarr \mathcal{G}(A)$ is left adjoint to the functor $(-)_d:\mathcal{G}(A) \rarr \mathcal{M}(A_0)$ projecting onto the $d$-th homogeneous component. The unit of the adjunction is an isomorphism.
\end{lem}

\begin{proof}
Let $M \in \mc{M}(A_0)$. Since $A \cdot A_0 = A$, it follows that the $A$-module $\Ind_{A_0}^A M$ is generated by its degree zero component $(\Ind_{A_0}^A M)_0 = A_0 \otimes_{A_0} M = M$. After shifting, we see in general that $(\Ind_{A_0}^A M)\lbrack d \rbrack$ is generated by its degree-$d$ component $M$. In particular, a graded $A$-module morphism $f \from (\Ind_{A_0}^A M)\lbrack d \rbrack \to N$ is uniquely determined by its degree-$d$ component, which is an $A_0$-module morphism $f_d:M \rarr N_d$. It is now straightforward to see that this gives the adjunction and that the unit of this adjunction is an isomorphism.
\end{proof}

\begin{prop}\label{prop:Azerosimple}
The functor $(-)_0$ induces a bijection between 
\[
\uIrr \mc{G}(A) \dopgleich \{  L \in \Irr \mathcal{G}(A) \mid L_0 \neq 0 \} 
\]
and $\Irr \mc{M}(A_0)$.
\end{prop}

\begin{proof}
Let us first show that this map is well-defined, so assume that $L_0 \neq 0$. Let $S \subs L_0$ be a simple $A_0$-submodule. Because of the adjunction in Lemma \ref{ind_proj_adjunction} the embedding $S \hookrightarrow L_0$ induces a graded $A$-module morphism $f \from \Ind_{A^0}^A S \to L$, which is non-zero since $S \hookrightarrow L$ is non-zero. Since $L$ is simple, $f$ must be surjective. Then  the component map $f_0 \from S \to L_0$ is a surjective $A_0$-module morphism. Hence $S \simeq L_0$ is a simple $A_0$-module.

To show that the given map is injective, we first argue that if $S \in \Irr \mathcal{M}(A_0)$, then among all graded $A$-submodules $N'$ of $M \dopgleich \Ind_{A^0}^A S$ with the property that $N'_0=0$ there is a unique maximal one. Let $\Sigma$ be the set of all graded submodules $N'$ of $M$ with $N_0' = 0$. Let $N$ be the submodule generated by the submodules in $\Sigma$. If we can show that $N_0 = 0$, then existence and uniqueness are clear. So, let $n \in N_0$. Then $n = \sum_{k} n_k$ for some homogeneous $n_k \in N^{(k)}$ with $N^{(k)} \in \Sigma$. Since $N_0$ is graded, we can assume that $n_k \in N_0$ for all $k$. Hence, $n_k \in N^{(k)} \cap N_0 = N^{(k)}_0 = 0$, so $n=0$. 

Now, assume that $L,L' \in \ul{\Irr} \ \mc{G}(A)$ are such that $L_0$ and $L'_0$ are isomorphic to the same simple $A_0$-module $S$. By the adjunction in Lemma \ref{ind_proj_adjunction} we get surjective graded $A$-module morphisms $f \from M  \twoheadrightarrow L$ and $f' \from M \twoheadrightarrow L'$, where $M \dopgleich \Ind_{A^0}^A S$. From the exact sequence
\[
\begin{tikzcd}
0 \arrow{r} & \Ker(f) \arrow{r} & M \arrow{r}{f} & L  \arrow{r} & 0 
\end{tikzcd}
\]
we get the exact sequence
\[
\begin{tikzcd}
0 \arrow{r} & \Ker(f)_0 \arrow{r} & M_0 \arrow{r}{f_0} & L_0 \arrow{r} & 0 \;.
\end{tikzcd}
\]
But $f_0$ is the isomorphism $S = M_0 \simeq L_0$ we started with, so $\Ker(f)_0 = 0$, implying that $\Ker(f) \subs N$, where $N$ is the submodule of $M$ from above. Similarly, we see that $\Ker(f') \subs N$. Hence, $M/N$ is both a (non-zero) quotient of $L$ and of $L'$. As the latter modules are simple, it follows that $L \simeq M/N \simeq L'$.

To show that the map is surjective, let $S \in \Irr \mathcal{M}(A_0)$. Let $L$ be a graded simple quotient of $M \dopgleich \Ind_{A^0}^A S$. Since $M$ is generated by $M_0$, it follows that the $A_0$-module $L_0$ is a quotient of $M_0$ which generates $L$. Hence, $L_0 \neq 0$ and since $L_0 = S$, it follows that $L_0 = S$.
\end{proof}

By considering the number of simple modules, we immediately deduce:

\begin{cor} \label{core_splits}
If the $K$-algebra $A$ is split then $A_0$ is split.  \qed
\end{cor}

Let $\Pim \mc{M}(A)$ be the set of isomorphism classes of indecomposable projective objects in $\mc{M}(A)$. For these objects we have a similar classification: 

\begin{prop} \label{Azeroproj} \hfill

\begin{enum_thm}
\item Every primitive idempotent of $A_0$ is a primitive idempotent of $A$.
\item The maps $(-)_0$ and $\Ind_{A_0}^A$ induce pairwise inverse bijections between the sets $\Pim \mc{M}(A_0)$ and 
\[
\uPim \mc{G}(A) \dopgleich \lbrace P(S) \mid S \in \uIrr \mc{G}(A) \rbrace \;.
\]
This bijection is compatible with taking projective covers, i.e., the diagram
\[
\begin{tikzcd}
\uIrr \mc{G}(A) \arrow{d}[swap]{(-)_0} \arrow{r}{P} & \uPim \mc{G}(A) \arrow{d}{(-)_0} \\
\Irr \mc{M}(A_0) \arrow{r}[swap]{P} & \Pim \mc{M}(A_0)
\end{tikzcd}
\]
commutes.
\item The decomposition of $A$ into indecomposable projective objects in $\mc{G}(A)$ is given by 
\[
A \simeq \bigoplus_{P \in \uPim \mc{G}(A)} P^{\oplus n_P} 
\]
for certain $n_P \in \bbN_{>0}$.
\end{enum_thm}
\end{prop}

\begin{proof}
The first statement is due to Green–Gordon \cite[Proposition 5.8]{GG-Graded-Artin}. Let $P \in \Pim \mc{M}(A_0)$. Then $P=A_0e$ for a primitive idempotent $e \in A_0$. Considering $A$ as an $A_0$-module, there is a canonical isomorphism $\Ind_{A_0}^A A_0 = A \otimes_{A_0} A_0 \simeq A$ of $A_0$-modules given by multiplication. This map is certainly $A$-linear and graded, thus an isomorphism of graded $A$-modules. This isomorphism induces a graded $A$-module isomorphism $\Ind_{A_0}^A P \simeq Ae$. Hence, by the first statement, $\Ind_{A_0}^A P \in \Pim \mc{G}(A)$. Since $(-)_0 \circ \Ind_{A_0}^A \simeq \id$, it follows that the map $\Pim \mc{M}(A_0) \rightarrow \Pim \mc{G}(A)$  is injective. By Proposition \ref{prop:Azerosimple}, we have $P \simeq P(S_0)$ for some $S \in \uIrr \mc{G}(A)$. By the adjunction of Lemma \ref{ind_proj_adjunction}, we have 
\[
0 \neq \Hom_{\mc{M}(A_0)}(P(S_0),S_0) \simeq \Hom_{\mc{G}(A)}(\Ind_{A_0}^A P(S_0), S) \;.
\]
Hence, there is a non-zero morphism $\Ind_{A_0}^A P(S_0) \to S$, which is surjective since $S$ is simple. Hence, $\Ind_{A_0}^A P(S_0) \simeq P(S)$. This shows that $\Ind_{A_0}^A P \in \uPim \mc{G}(A)$ and that the map $\Ind_{A_0}^A \from \Pim \mc{M}(A_0) \to \uPim\mc{G}(A)$ is surjective. Taking degree zero components we deduce that $P(S_0) \simeq P(S)_0$. This shows the commutativity of the diagram. Finally, we have a direct sum decomposition $A \simeq \bigoplus_{P \in \Pim \mc{M}(A_0)} P^{\oplus n_P}$ for certain $n_P \in \bbN_{>0}$. Inducing this to $A$ gives the claimed direct sum decomposition of $A$ in $\mc{G}(A)$.
\end{proof}

\begin{cor} \label{F_uIrr_eq_Irr}
$F(\uIrr \mc{G}(A)) = \Irr \mc{M}(A)$. \qed
\end{cor}

Since $A_0$ is concentrated in degree zero, the decomposition of graded modules into homogeneous components yields a decomposition $\mc{G}(A_0) = \bigoplus_{d \in \bbZ} \mc{G}_d(A_0)$ with $\mc{G}_d(A_0) \simeq \mc{M}(A_0)$. Using this, we can relate graded multiplicities in $\mc{G}(A)$ to ungraded multiplicities in $\mc{M}(A_0)$:

\begin{prop} \label{graded_mults_from_core}
Let $M \in \mc{G}(A)$ and $L \in \Irr \mc{G}(A)$. Then for any $d \in \bbZ$ such that $L_d \neq 0$ we have
$$
\lbrack M:L \rbrack_{\mc{G}(A)} =  \lbrack M_d:L_d \rbrack_{\mc{M}(A_0)} \;.
$$
\end{prop}

\begin{proof}
We may assume without loss of generality that $M$ is semi-simple. Thus, $M \simeq \bigoplus_{L \in \Irr \mc{G}(A)} L^{\oplus n_{L}}$ for some $n_L$. Now the equation is an immediate consequence of Proposition \ref{prop:Azerosimple}, since $L_d$ is a simple $A_0$-module for each $L \in \Irr \mc{G}(A)$ with $L_d \neq 0$. 
\end{proof}

\section{Cellularity of the core} \label{cellular_structure}
 
 \begin{tcolorbox}
 For the remainder of the paper we assume that $A$ is a graded with triangular decomposition $A \simeq A^- \otimes_K T \otimes_K A^+$ as in Definition \ref{defn:triangle}. Moreover, we assume that $T$ is semisimple. 
 \end{tcolorbox}
 
Since
$$
A = A^- \otimes_K T \otimes_K A^+ = \bigoplus_{i,j \in \bbZ} A_i^- \otimes_K T \otimes_K A_j^+ \;,
$$
we have
$$
A_0 = \bigoplus_{i \in \bbZ} A_i^- \otimes_K T \otimes_K A_{-i}^+ \;,
$$
and
$$
\dim A_0 = \dim T \cdot \sum_{i \in \bbZ} \dim A_i^- \cdot \dim A_i^+ \;.
$$

In \S\ref{core_rep_general} we have already seen some general connections between the graded representation theory of $A$ and the ungraded representation theory of $A_0$. First of all, we deduce from \cite[Proposition 3.15]{hwtpaper} and Corollary \ref{core_splits} that:

\begin{cor} \label{core_splits_2}
The core $A_0$ is a split $K$-algebra. \qed
\end{cor}

Moreover, it follows from Proposition \ref{prop:Azerosimple} that 
\begin{equation} \label{core_simples_by_projection}
\Irr \mc{M}(A_0) = \lbrace L(\lambda)_0 \mid \lambda \in \uIrr \mc{G}(T) \rbrace\;,
\end{equation}
where
$$
\uIrr \mc{G}(T) \dopgleich \lbrace \lambda \in \Irr \mc{G}(T) \mid L(\lambda)_0 \neq 0 \rbrace \;.
$$

If $\lambda \in \Irr \mc{G}_0(T)$, then we know from \cite[Lemma 3.17]{hwtpaper} that $L(\lambda)_0 \simeq \lambda$, so in particular $L(\lambda)_0 \neq 0$. We thus obtain:

\begin{lem}
$\Irr \mc{M}(T) \subs \Irr \mc{M}(A_0)$. \qed
\end{lem}

An application of Proposition \ref{graded_mults_from_core} shows that we can in principle compute the graded decomposition matrices of standard modules using the core: for any $\lambda,\mu \in \Irr \mc{G}(T)$ and $d \in \bbZ$ with $L(\mu)_d \neq 0$ we have
$$
\lbrack \Delta(\lambda) : L(\mu) \rbrack_{\mc{G}(A)} = \lbrack \Delta(\lambda)_d : L(\mu)_d \rbrack_{\mc{M}(A_0)} \;.
$$

\subsection{Standard datum and cellularity}

We will now derive another classification of the simple $A_0$-modules by showing that it comes from a richer structure on $A_0$, namely from a \textit{standard datum} in the sense of Du–Rui \cite{DuRui}. The key ingredient here is that we have a canonical algebra isomorphism
\begin{equation*}
E_A \dopgleich \End_{\mathcal{G}(A)}(A) \simeq A_0^\op.
\end{equation*}
Hence, if $A$ is self-injective, then the core $A_0$ is the (opposite of the) endomorphism algebra of the tilting object $A$ of $\mathcal{G}(A)$. We can now employ our highest weight and tilting theory for $\mathcal{G}(A)$, together with the general theory of Appendix \ref{cellular_appendix}, developed along the lines of Anderson–Stroppel–Tubbenhauer \cite{AST}, to construct a natural standard datum on $A_0$. In the case where $A$ is graded Frobenius and admits a triangular anti-involution, this datum is in fact a cellular datum in the sense of Graham–Lehrer \cite{Graham-Lehrer-Cellular}, so $A_0$ is a cellular algebra. This applies to most of our VIP examples, providing new tools to study their representation theory. We first recall some relevant definitions, for more details we refer to Appendix \ref{cellular_appendix}.

\begin{defn}[Du–Rui] \label{standard_datum_def}
A \word{standard datum} on a finite-dimensional $K$-algebra $E$ consists of the following data:
\begin{enum_thm}
\item A $K$-basis $\mathcal{B}$ of $E$ which is fibered over a poset $\Lambda$, i.e., $\mc{B} = \bigsqcup_{\lambda \in \Lambda} \mc{B}^{\lambda}$, where $\mc{B}^{\lambda} \neq \emptyset$ for all $\lambda \in \Lambda$.
\item Indexing sets $F^{\lambda}$ and $G^{\lambda}$ for any $\lambda \in \Lambda$ such that $\mc{B}^\lambda = \{ b_{i,j}^{\lambda} \  | \ (i,j) \in F^{\lambda} \times G^{\lambda} \}$,
\end{enum_thm} 
subject to the conditions:
\begin{enum_thm}
\item[(c)] For any $b \in E$ and $b_{i,j}^\lambda \in \mc{B}^{\lambda}$ we have 
$$
b \cdot b^{\lambda}_{i,j} \equiv \sum_{i' \in F^{\lambda}} f_{i',\lambda}(b,i) b_{i',j}^{\lambda} \ \mod \ E^{< \lambda} \;, 
$$
$$
b^{\lambda}_{i,j} \cdot b \equiv \sum_{j' \in G^{\lambda}} f_{\lambda,j'}(j,b) b_{i,j'}^{\lambda} \ \mod \ E^{< \lambda} \;, 
$$
where $f_{i',\lambda}(b,i), f_{\lambda,j'}(j,b) \in K$ are independent of $j$ and $i$, respectively. Here, $E^{<\lambda}$ is the subspace of $E$ spanned by the set $\bigcup_{\mu < \lambda} \mathcal{B}^\mu$. 
\end{enum_thm}

\end{defn}

\begin{defn}[Graham–Lehrer] \label{cellular_datum_def}
A \word{cell datum} for a finite-dimensional $K$-algebra $E$ is a standard datum, as above, such that additionally $F^\lambda = G^\lambda$ for all $\lambda \in \Lambda$ and there is a $K$-linear anti-involution $^*$ on $A$ such that $(b_{i,j}^\lambda)^* = b_{j,i}^\lambda$ for all $\lambda \in \Lambda$ and $i,j \in F^\lambda$. An algebra admitting a cell datum is called a \word{cellular algebra}.
\end{defn}

\begin{rem}
In fact, it is not difficult to see that any split finite-dimensional $K$-algebra admits a standard datum (this was pointed out to us by S. Koenig, see \cite[\S5]{KonigXiCellular}). What is thus relevant is not that a particular algebra admits a standard datum (hence, that it is \textit{standardly based} in the terminology of Du–Rui) but the datum itself. The situation for cellular algebras is different since not every algebra is cellular. The fact that an algebra admits a cellular datum implies certain restrictions on its representation theory. 
\end{rem}

\begin{thm} \label{core_cellular}
Assume that $A$ is self-injective. The algebra $E_A$ has a standard datum with poset
\begin{equation*}
\mathcal{P}_A \dopgleich \lbrace \lambda \in \Irr \mathcal{G}(T) \mid \lbrack A:\Delta(\lambda) \rbrack \neq 0 \tn{ and } \lbrack A:\nabla(\lambda) \rbrack \neq 0 \rbrace \subs \Irr \mathcal{G}(T)
\end{equation*}
and standard basis $\mathcal{B} = \coprod_{\lambda \in \mathcal{P}} \mathcal{B}^{\lambda}$ constructed from an arbitrary basis $F^{\lambda}$ of $\Hom_{\mathcal{G}(A)}(A,\nabla(\lambda))$ and an arbitrary basis $G^{\lambda}$ of $\Hom_{\mathcal{G}(A)}(\Delta(\lambda),A)$ as follows: for each $f^\lambda \in F^{\lambda}$ and $g^\lambda \in G^{\lambda}$ choose lifts $\hat{f}^\lambda$ and $\hat{g}^\lambda$ such that the diagram
\begin{equation*}
\begin{tikzcd}
 & \Delta(\lambda) \arrow[hookrightarrow]{d} \arrow{dr}{g} \\
A \arrow{dr}[swap]{f^\lambda} \arrow{r}{\hat{f}^\lambda} & T(\lambda) = P(\lambda^{h^{-1}}) \arrow[twoheadrightarrow]{d} \arrow{r}[swap]{\hat{g}^\lambda} & A  \\
& \nabla(\lambda)
\end{tikzcd}
\end{equation*}
commutes and set  
\begin{equation*}
\mathcal{B}^{\lambda} \dopgleich \lbrace \hat{g}^\lambda \circ \hat{f}^\lambda \mid f^\lambda \in F^{\lambda}, g^\lambda \in G^{\lambda} \rbrace \;.
\end{equation*}
This basis is independent of the choice of lifts. We have
\begin{equation} \label{cellular_structure_P}
\mathcal{P}_A = \lbrace \lambda \in \Irr \mathcal{G}(T) \mid \deg \lambda \in \Supp A^+ \cap - \Supp A^- \rbrace 
\end{equation}
and
\begin{equation} \label{cellular_structure_P_size}
|\mathcal{P}_A| = |\Irr \mathcal{M}(T)| \cdot | \Supp A^+ \cap - \Supp A^-| \;.
\end{equation}
Furthermore, for any $\lambda \in \mathcal{P}$ we have
\begin{equation} \label{cellular_structure_F_size}
\ \ \ \ \ \ |F^\lambda| = \lbrack A:\Delta(\lambda)\rbrack = \dim_K \lambda \cdot \dim_K A^{+}_{\deg \lambda} \;,
\end{equation}
\begin{equation} \label{cellular_structure_G_size}
 \quad |G^\lambda| = \lbrack A:\nabla(\lambda) \rbrack= \dim_K \lambda \cdot \dim_K A^-_{-\deg \lambda}
\end{equation}
If $A$ is graded Frobenius and admits a triangular anti-involution, then this standard datum is in fact a cellular datum, so $E_A$ is a cellular algebra.
\end{thm}

\begin{proof}
The claim about the standard datum is an application of Theorem \ref{end_tilting_is_cellular} once we know that $\mathcal{G}(A)$ satisfies the assumptions listed in Appendix \ref{cellular_appendix}. Assumption \ref{tilting_hwc_assumption} holds by \cite[Corollary 4.13]{hwtpaper}. The Ext-vanishing in Assumption \ref{tilting_ext_assumption} holds by \cite[Lemma 4.3 (b)]{hwtpaper}. Finally, Assumption \ref{tilting_indec_assumption} about tilting objects holds by \cite[Corollary 5.9]{hwtpaper}. If $A$ is graded Frobenius and admits a triangular anti-involution, then we know from \cite[Corollary 6.4]{hwtpaper} that $A$ is fixed by the induced duality $\bbD$ on $\mathcal{G}(A)$, hence the claim about cellularity of $E_A$ follows again from Theorem \ref{end_tilting_is_cellular}.

It remains to show equations (\ref{cellular_structure_P}) to (\ref{cellular_structure_G_size}). By the discussion in Appendix \ref{cellular_appendix} we have $|F^\lambda| = \dim_K \Hom_{\mathcal{G}(A)}(A,\nabla(\lambda)) = \lbrack A:\Delta(\lambda) \rbrack$. Let $N^i$ be the $B^+$-submodule of $B^+$ generated by $\lbrace b \in B^+ \mid \deg b \geq i \rbrace$. We get a descending filtration $B^+ = N^0 \supseteq N^1 \supseteq \cdots$ of $B^+$. Let $M^i \dopgleich \Ind_{B^+}^A N^i$. Since $A$ is a free $B^+$-module, the functor $\Ind_{B^+}^A$ is exact, so we get a filtration $A = M^0 \supseteq M^1 \sups \cdots$ of $A$ as an $A$-module with quotients
\[
M^i/M^{i+1} \simeq \Ind_{B^+}^A (N^i/N^{i+1}) \;.
\]  
By construction, the augmentation ideal $J^+$ of $B^+$ acts trivially on the quotient $N^i/N^{i+1}$, so $N^i/N^{i+1} = \mrm{Inf}_T^{B^+} Q^i$ for a $T$-module $Q^i$, which is semisimple since $T$ is semisimple. Hence,
\[
M^i/M^{i+1} \simeq \Ind_{B^+}^A \circ \Inf_T^{B^+}( Q^i) = \Delta(Q^i) \;,
\]
showing that the filtration $A = M^0 \sups M^1 \cdots$ is a standard filtration of $A$. In particular
\[
\lbrack A:\Delta(\lambda) \rbrack = \sum_{i \in \Supp(B^+)} \lbrack B_i^+ :\lambda \rbrack = \lbrack B_{\deg \lambda}^+ :\lambda \rbrack \;.
\]
We know that $B^+_i$ is a free $T$-module of rank $\dim_K A_i^+$. Hence,
\[
\lbrack B_{\deg \lambda}^+:\lambda \rbrack = \lbrack T:\lambda \rbrack \cdot \dim_K A_{\deg \lambda}^+ = \dim_K \lambda \cdot \dim_K A_{\deg \lambda}^+ \;.
\]
This shows (\ref{cellular_structure_F_size}). By Lemma \ref{lem:aaop} below, we have $\lbrack A:\nabla(\lambda) \rbrack = \lbrack A^\op:\Delta(\lambda^*) \rbrack$. Applying the same argument to $A^{\op}$ and $\lambda^* \in \Irr \mathcal{G}(T^{\op})$ and dualizing yields
\[
\lbrack A:\nabla(\lambda) \rbrack = \lbrack A^\op:\Delta(\lambda^*) \rbrack = \dim_K \lambda^* \cdot \dim_K (A^\op)^+_{\deg \lambda^*} = \dim_K \lambda \cdot \dim_K A^-_{- \deg \lambda} \;,
\] 
showing (\ref{cellular_structure_G_size}). It follows from the above equations that $\lambda \in \mathcal{P}$ if and only if $\deg \lambda \in \Supp A^+$ and $-\deg \lambda \in \Supp A^-$, proving (\ref{cellular_structure_P}). Since $\Irr \mathcal{G}(T) \simeq \Irr \mathcal{M}(T) \times \bbZ$, this immediately implies  (\ref{cellular_structure_P_size}).
\end{proof}

The following is an application of \cite[Proposition 3.19]{hwtpaper}, and was used in Theorem \ref{core_cellular} above.  

\begin{lem}\label{lem:aaop}
For all $\lambda \in \Irr \mc{G}(T)$, $[A : \nabla(\lambda)] = [A^{\mathrm{op}} : \Delta(\lambda^*)]$. 
\end{lem}

\begin{proof}
Under the conventions of \cite[\S 2.2]{hwtpaper}, the following diagram commutes
\begin{equation}\label{eq:commuteK0s}
\begin{tikzcd}
K_0(\mc{G}(A)) \arrow{d}{\chi} \arrow{r}{( - )^*} & K_0(\mc{G}(A^{\mathrm{op}})) \arrow{d}{\chi^{\mathrm{op}}} \\
K_0(\mc{G}(T)) \arrow{r}{( - )^*} & K_0(\mc{G}(T^{\mathrm{op}})).
\end{tikzcd}
\end{equation}
Since $\mc{G}^{\nabla}(A)$ is an exact category, we can consider its Grothendieck group $K_0(\mc{G}^{\nabla}(A))$. It follows from \cite[Proposition 3.19]{hwtpaper} that $K_0(\mc{G}^{\nabla}(A))$ is a $\Z$-submodule of $K_0(\mc{G}(A))$, with $\Z$-basis the classes $[\nabla(\lambda)]$. Similarly, the $[\Delta(\lambda^*)]$ are a $\Z$-basis of $K_0(\mc{G}^{\Delta}(A^{\mathrm{op}})) \subset K_0(\mc{G}(A^{\mathrm{op}}))$. We think of $K_0(\mc{G}(T^{\mathrm{op}}))$ as a free $\Z[t,t^{-1}]$-module with basis $[\lambda^*]$ for $\lambda \in \Irr \mc{M}(T)$. Then we define the $\Z$-linear involution $\mathrm{inv}$ on this $\Z$-module, which fixes all $[\lambda^*]$ and sends $t$ to $t^{-1}$. Then, it follows from \cite[Lemma 3.10]{hwtpaper} that 
$$
\chi([\nabla(\lambda)]) = \sum_{i \ge 0} t^{\deg \lambda - i} [((B^+)^{\mathrm{op}}_i \otimes_{T^{\mathrm{op}}} \lambda^*)^*],
$$
and
$$
\chi^{\mathrm{op}}([\Delta(\lambda^*)]) = \sum_{i \ge 0} t^{i - \deg \lambda} [(B^+)^{\mathrm{op}}_i \otimes_{T^{\mathrm{op}}} \lambda^*].
$$
Thus, $((\chi([\nabla(\lambda)]))^* = \mathrm{inv}(\chi^{\mathrm{op}}([\Delta(\lambda^*)]))$, which implies that diagram (\ref{eq:commuteK0s}) restricts to 
$$
\begin{tikzcd}
K_0(\mc{G}^{\nabla}(A)) \arrow{r}{\chi} \arrow[bend right=15]{rrrr}{\sim} & K_0(\mc{G}(T))  \arrow{r}{(-)^*} & K_0(\mc{G}(T^{\mathrm{op}})) & & \arrow[swap]{ll}{\mathrm{inv} \circ \chi^{\mathrm{op}}} K_0(\mc{G}^{\Delta}(A^{\mathrm{op}})).
\end{tikzcd}
$$
Since $[A] \in K_0(\mc{G}^{\nabla}(A))$ and $[A^{\mathrm{op}}] \in K_0(\mc{G}^{\Delta}(A^{\mathrm{op}}))$, the above isomorphism implies that it suffices to show that 
$$
(\chi([A]))^* = \mathrm{inv} (\chi^{\mathrm{op}}([A^{\mathrm{op}}]))
$$
in $K_0(\mc{G}(T^{\mathrm{op}}))$. We have
$$
\chi([A]) = \sum_{\lambda \in \Irr \mc{M}(T)} (\underline{\dim} A^- \cdot \underline{\dim} A^+ \cdot \dim \lambda ) [\lambda],
$$
and hence 
$$
(\chi([A]))^* = \sum_{\lambda \in \Irr \mc{M}(T)} (\underline{\dim} A^- \cdot \underline{\dim} A^+ \cdot \dim \lambda) [\lambda^*].
$$
Similarly, $\chi^{\mathrm{op}}([A^{\mathrm{op}}])) = \sum_{\lambda \in \Irr \mc{M}(T)} \mathrm{inv} ( \underline{\dim} A^- \cdot \underline{\dim} A^+ \cdot \dim \lambda) [\lambda^*]$.
\end{proof}

\subsection{Cell modules} \label{sec_cell_modules}
As explained in Appendix \ref{cellular_appendix}, see also \cite[1.2.1]{DuRui}, the standard datum on $A_0^\op$ naturally induces a standard datum on the opposite algebra $A_0$, with the same indexing poset $\mc{P}_A$. Attached to this standard datum are cellular standard and cellular costandard modules, for each $\lambda  \in \mathcal{P}_A$. In the terminology of Graham–Lehrer \cite[Definition 2.1]{Graham-Lehrer-Cellular}, these are the left cell representations and the duals of the right cell representations. By (\ref{opposite_cellular_standard_modules}) these modules for the standard datum on $A_0$ are given by
\begin{equation*}
\Delta_0(\lambda) \dopgleich \Delta_A^\op(\lambda) = \Hom_{\mathcal{G}(A)}(A,\nabla(\lambda))
\end{equation*}
and
\begin{equation*}
\nabla_0(\lambda) \dopgleich \nabla_A^\op(\lambda) = \Hom_{\mathcal{G}(A)}(\Delta(\lambda),A)^* \;,
\end{equation*}
respectively. Note that the cellular standard module is given by
\begin{equation*}
\Delta_0(\lambda) \simeq \nabla(\lambda)_0 \;,
\end{equation*}
so it is the degree zero part of the costandard module. For the cellular costandard modules we have a similar result:

\begin{lem}
If $A$ is graded Frobenius, then 
\begin{equation*}
\nabla_0(\lambda) \simeq \Delta(\lambda)_0 \;.
\end{equation*}
\end{lem}

\begin{proof}
By \cite[Lemma 5.5]{hwtpaper} we have $A^* \simeq A^\op$ in $\mc{G}(A^\op)$, so 
\begin{align*}
\nabla_0(\lambda) & = \Hom_{\mathcal{G}(A)}(\Delta(\lambda),A)^* \simeq \Hom_{\mc{G}(A^\op)}(A^*, \Delta(\lambda)^*)^* \\ & \simeq \Hom_{\mc{G}(A^\op)}( A^\op, \nabla(\lambda^*))^* \simeq (\nabla(\lambda^*)_0)^* \simeq \Delta(\lambda)_0 \;.
\end{align*}
Taking duals proves the claim.
\end{proof}

As explained in Appendix \ref{cellular_appendix}, there is a subset $\ul{\mathcal{P}}_{A}$ of $\mathcal{P}_A$ such that the cellular standard module $\Delta_0(\lambda)$ has a simple head, which we denote by $L_0(\lambda)$, and the map $L_0 \from \ul{\mathcal{P}}_{A} \to \Irr \mathcal{M}(A_0)$ is a bijection. This classification of the simple $A_0$-modules is linked to the one in (\ref{core_simples_by_projection}) given by projection as follows:

\begin{prop}  \label{cell_param_match}
If $A$ is self-injective, then the permutation $h$ on $\Irr \mc{G}(T)$ defined in \cite[Theorem 5.8]{hwtpaper} induces a bijection $\uIrr \mc{G}(T) \simeq \ul{\mathcal{P}_{A}}$ and there is an isomorphism
\begin{equation*}
L_0(\lambda) \simeq L(\lambda^{h})_0 \;.
\end{equation*}
\end{prop}

\begin{proof}
Let $\lambda \in \ul{\mc{P}}_A$, so $\Delta_0(\lambda) = \nabla(\lambda)_0$ has simple head $L_0(\lambda)$. It follows from Theorem \ref{cell_appendix_dim_simples} that 
\[
0 \neq \dim L_0(\lambda) = \lbrack A:T(\lambda) \rbrack = \lbrack A:P(\lambda^{h^{-1}}) \rbrack
\]
as multiplicities in $\mc{G}(A)$. So, $P(\lambda^{h^{-1}})$ is a direct summand of $A$ in $\mc{G}(A)$. Hence, $\lambda^{h^{-1}} \in \uIrr \mc{G}(T)$ by Proposition \ref{Azeroproj}. Since $h^{-1}$ is a permutation and both sets $\uIrr \mc{G}(T)$ and $\ul{\mc{P}}_A$ classify the simple $A_0$-modules, it follows that $h$ induces a bijection between these sets. From \cite[Theorem 5.8]{hwtpaper} we know that $\nabla(\lambda)$ has simple head $L(\lambda^{h^{-1}})$. Since $\lambda^{h^{-1}} \in \uIrr \mc{G}(T)$, we know that $L(\lambda^{h^{-1}})_0$ is a simple $A_0$-module, in particular non-zero, so $L(\lambda^{h^{-1}})_0$ is a constituent of the head of $\nabla(\lambda)_0$. But as $\nabla(\lambda)_0$ has simple head $L_0(\lambda)$ and $L(\lambda^{h^{-1}})_0$ is simple, it follows that $L(\lambda^{h^{-1}})_0 \simeq L_0(\lambda)$.
\end{proof}

As an application of Theorem \ref{cell_appendix_semisimple} we obtain:

\begin{prop}
Suppose that $A$ is self-injective. Then $A_0$ is semisimple if and only if $A$ is semisimple. \qed
\end{prop}

Under a certain condition on the triangular decomposition we can extend the above proposition.

\begin{prop}  \label{A0_quasi_hereditary}
Suppose that $A$ is self-injective and $\Supp A^+ = - \Supp A^-$ (e.g., if $A$ is BGG). Then $A_0$ is quasi-hereditary if and only if $A$ is semisimple.
\end{prop}

\begin{proof}
The classification of the simple $A_0$-modules in (\ref{core_simples_by_projection}) shows that 
\begin{equation} \label{core_number_of_simples_equ}
|\Irr \mathcal{M}(A_0)| = \sum_{\lambda \in \Irr \mathcal{G}_0(T)} |\Supp L(\lambda)|  \;.
\end{equation}
We know from Corollary \ref{core_splits_2} that $A_0$ splits and from Theorem \ref{core_cellular} that $A_0$ admits a standard datum. Results by Du–Rui \cite[Corollary 2.4.2, Theorem 4.2.3, and Theorem 4.2.7]{DuRui} now show that $A_0$ is quasi-hereditary if and only if $|\Irr \mathcal{M}(A_0)| = |\mathcal{P}_A|$. For $\lambda \in \Irr \mathcal{G}_0(T)$ we have $\Supp \Delta(\lambda) = \Supp A^-$ and $\Supp \nabla(\lambda) = -\Supp A^+$ by \cite[Lemma 3.9]{hwtpaper} 
and \cite[Lemma 3.10]{hwtpaper}, respectively. The assumption thus implies that $\Supp \Delta(\lambda) = \Supp \nabla(\lambda)$. Equation (\ref{cellular_structure_P_size}) now gives 
\[
|\mathcal{P}_A| = |\Irr \mathcal{M}(T)| \cdot |\Supp \Delta(\lambda)| = |\Irr \mathcal{M}(T)| \cdot |\Supp \nabla(\lambda)|
\]
for all $\lambda \in \Irr \mathcal{G}_0(T)$. We can rewrite this as
\begin{equation} \label{P_size_equ}
|\mathcal{P}_A| = \sum_{\lambda \in \Irr \mathcal{G}_0(T)} |\Supp \Delta(\lambda)| = \sum_{\lambda \in \Irr \mathcal{G}_0(T)} |\Supp \nabla(\lambda)| \;.
\end{equation}
Comparing equations (\ref{core_number_of_simples_equ}) and (\ref{P_size_equ}), and using the fact that $\Supp L(\lambda) \subs \Supp \Delta(\lambda)$ and $\Supp L(\lambda) \subs \Supp \nabla(\lambda)$, we conclude that $|\Irr \mathcal{M}(A_0) | = |\mathcal{P}_A|$ if and only if $\Supp L(\lambda) = \Supp \Delta(\lambda)$ and $\Supp L(\lambda) = \Supp \nabla(\lambda)$ for all $\lambda \in \Irr \mathcal{G}_0(T)$. This is equivalent to $L(\lambda) = \Delta(\lambda)$ and $L(\lambda) = \nabla(\lambda)$ for all $\lambda \in \Irr \mathcal{G}_0(T)$, which by \cite[Corollary 3.16]{hwtpaper} is equivalent to the semisimplicity of $A$.
\end{proof}

\begin{ex}
The following example shows that even if $A$ is self-injective, the core need not be self-injective. This shows also that there are cases where $A$ is self-injective but not graded Frobenius since otherwise the core would also be graded Frobenius. Consider the eight-dimensional algebra $A = K[x,y,z] / (x^2,y^2,z^2)$ with grading $\deg (x) = -1$ and $\deg(y) = \deg(z) = +1$. Then we get a triangular decomposition with $A^- = K \lbrack x \rbrack/(x^2)$, $T=K$, and $A^+ = K\lbrack y,z \rbrack/(y^2,z^2)$. The algebra $A$ is local, having only one simple module, so $P \dopgleich {_AA}$ is the unique indecomposable projective left $A$-module. The subspace $\langle x,y,z,xy,xz,yz,xyz \rangle_K$ of $A$ is clearly a nilpotent ideal and thus already equal to the Jacobson radical for dimension reasons. From this we obtain that 
\[
\mrm{Soc}(P) = \lbrace a \in A \mid \Rad(A)a = 0 \rbrace = \langle xyz \rangle_K \;.
\]
In particular, $\mrm{Soc}(P)$ is simple, so $\mrm{Soc}(P) \simeq \mrm{Hd}(P)$ and therefore $A$ is weakly symmetric. Now, it follows from \cite[Corollary IV.6.3]{Skowronski:2011bx} that $A$ is Frobenius, thus in particular self-injective. The core of $A$ is $A_0 = K \lbrack a,b \rbrack$ with $a \dopgleich xy$ and $b \dopgleich xz$. Again, $A_0$ is indecomposable, so $P_0 \dopgleich {_{A_0}}{A_0}$ is the unique indecomposable left $A_0$-module. We have $\Rad(A_0) = \langle xy,xz \rangle_K$, so $\Soc(P_0) = \langle xy,xz \rangle_K$ is two-dimensional. Since $A$ only has a one-dimensional simple module, our classification of simple $A_0$-modules, Proposition \ref{prop:Azerosimple}, shows that simple $A_0$-modules are also one-dimensional. Hence, $\Soc(P_0)$ is not simple. But then $A_0$ does not admit a Nakayama permutation, so is not self-injective, see \cite[Theorem IV.6.1]{Skowronski:2011bx}.
\end{ex}

\subsection{Cellularity of $A$} \label{A_not_cellular}
In this section we consider instead the question of whether $A$ itself is cellular. We will show that under certain conditions the algebra $A$ \textit{cannot} be cellular. This applies to many VIP examples as we will see.

Recall that the \textit{ungraded} decomposition matrix $\mathbf{D}_{\Delta}$ of $A$ is the square matrix of size $| \Irr \mc{M}(T) |$ with $(\lambda,\mu)$-th entry $[\Delta(\lambda) : L(\mu)]$. The ungraded Cartan matrix $\mathbf{C}$ of $A$ is the square matrix of size $| \Irr \mc{M}(T) |$ with $(\lambda,\mu)$-th entry $[P(\lambda) : L(\mu)]$. 

In \cite{hwtpaper} we introduced the following property: $A$ is \textit{BGG} if $\lbrack \Delta(\lambda) \rbrack = \lbrack \nabla(\lambda) \rbrack$ in the \textit{graded} Grothendieck group of $A$. We gave an explicit condition, see \cite[Proposition 4.22]{hwtpaper}, implying the BGG property, and using this, we showed that all VIP examples are BGG. The following factorization property of the Cartan matrix in the BGG case is well-known:

\begin{lem} \label{eq:CartanDecomp}
If $A$ is BGG, then $\mathbf{C} = \mathbf{D}_{{\Delta}}^T \cdot \mathbf{D}_{{\Delta}}$.
\end{lem}

If $A = B_1 \oplus \cdots \oplus B_k$ is the block decomposition of $A$ then $\mathbf{C}$ and $\mathbf{D}_{{\Delta}}$ have block decompositions $\mathbf{C} = C_1 \oplus \cdots \oplus C_k$ and $\mathbf{D}_{{\Delta}} = D_1 \oplus \cdots \oplus D_k$, where $C_i$ and $D_i$ are the Cartan and decomposition matrices for $B_i$. For each $i$ we then have $C_i = D_i^T D_i$. The key obstruction we found for the cellularity of the VIP algebras is the following property:

\begin{defn}
We say that the algebra $A$ has the \word{rank one property} if each $D_i$ is of rank one over $\bbQ$.
\end{defn}

\begin{lem}\label{lem:rank1VIP}
Suppose that $A$ is BGG and has the rank one property. If $A$ contains a block with at least two simple modules, then $A$ is \textit{not} cellular.
\end{lem}

\begin{proof}
The assumption implies that there is some $i$ with $B_i$ a rank one non-invertible  matrix. Hence, $\det(C_i) = \det(D_i^T D_i) = 0$, so $C_i$, and thus $\mathbf{C}$, is not invertible, too. On the other hand, it has been shown by Koenig–Xi \cite[Proposition 3.2]{KonigXiCellularQH} that if $A$ is cellular then the determinant of the Cartan matrix is a (strictly) positive integer. Hence, $A$ is \textit{not} cellular.
\end{proof}

\begin{prop} \label{vip_rank1}
All of the following VIP examples satisfy the rank one property:
\begin{enumerate}
\item Restricted enveloping algebras $\overline{U}(\mf{g}_K)$;
\item Lusztig's small quantum groups $\mathbf{u}_{\zeta}(\mf{g})$, at a root of unity $\zeta$; 
\item Restricted rational Cherednik algebras $\overline{\H}_{\bc}(W)$ at $t = 0$;
\end{enumerate}
\end{prop}

\begin{proof} We use the notations from \cite[Section 8]{hwtpaper}.

(1) Recall that $\Irr \mc{M}(\overline{U}(\mf{h}_K))$ is naturally in bijection with $X_p$. For $\lambda, \mu \in X_p$, we write $\lambda \sim \mu$ if there exists $w \in W$ such that $\lambda + \rho = w(\mu + \rho)$ in $X$. By \cite[Theorem 2.2]{Hummodular}, $[\Delta(\lambda)] = [\Delta(\mu)]$ in $K_0(\mc{M}(\overline{U}(\mf{g}_K)))$ if $\lambda \sim \mu$. On the other hand, \cite[Theorem 3.18]{Ramifications} says that $L(\lambda)$ and $L(\mu)$ are in the same block if and only if $\lambda \sim \mu$; the case where $p$ is greater than the Coxeter number of $\mf{g}$ was already done in \cite[Theorem 3.1]{Hummodular}.

(2) First, we note that we can restrict to ``representations of type $\mathbf{1}$''. Namely, there is a direct sum decomposition 
$$
\mathbf{u}_{\zeta} = \bigoplus_{\nu} \mathbf{u}_{\zeta,\nu}
$$
where the sum is over all maps $\nu : \{ 1, \ds, \ell \} \rightarrow \{ 1 , -1 \}$ and $\mathbf{u}_{\zeta,\nu}$ is the quotient of $\mathbf{u}_{\zeta}$ by the central ideal generated by all $K_i^{\ell} - \nu(i)$. As explained in \cite[\S 4.6]{Lusztigmodular}, the algebras $\mathbf{u}_{\zeta,\nu}$ are pairwise isomorphic. Therefore, we set $\mathbf{u}_{\zeta}' = \mathbf{u}_{\zeta,\nu_0}$, where $\nu_0(i) := 1$ for all $i$. The explicit isomorphism given in \text{loc. cit.} makes it clear that standard modules are identified under the isomorphism. Therefore it suffices to prove the statement for $\mathbf{u}_{\zeta}'$. This algebra is often called the restricted quantum group. The simple modules of this algebra are in bijection with $X_{\ell}$. The algebra is a quotient of the De Concini-Kac quantized enveloping algebra.  Write $w \idot  \lambda := w(\lambda + \rho) - \rho$ for the $W_{\idot}$-action on $X$.  Then \cite[Theorem 4.8]{Ramifications} implies that:
\begin{equation}\label{lem:quantumblocks}
\textnormal{The blocks of } \mathbf{u}_{\zeta}' \textnormal{ are the } W_{\idot}\textnormal{-orbits intersecting } X_{\ell } \;. 
\end{equation}
We note that $\mathbf{u}_{\zeta}'$ also admits a triangular decomposition. The grading on $\mathbf{u}_{\zeta}$ given in \cite[(91)]{hwtpaper} makes $\mathbf{u}_{\zeta}'$ into a $X$-graded algebra. As in \cite[Section 4.4]{hwtpaper} , let $\mc{G}_X(\mathbf{u}_{\zeta}')$ denote the category of $X$-graded $\mathbf{u}_{\zeta}'$-modules. In order to apply arguments from the theory of $G_rT$-modules, we define $\mc{G}'$ to be the full subcategory of $\mc{G}_X(\mathbf{u}_{\zeta}')$ consisting of all $M$ such that 
$$
K_i \cdot m= \zeta^{\langle \lambda, \varphi_i \rangle} m, \quad \textrm{for all} \ m \in M_{\lambda}. 
$$
The simple modules in $\mc{G}'$ are in bijection with $X$; whereas the simple modules in $\mc{G}_X(\mathbf{u}_{\zeta}')$ are in bijection with $X \times X_{\ell}$. If $M \in \mc{G}'$, then $M$ has a well-defined character $\chi(M) \in \Z[X]$. The Weyl group $W$ acts on $X$. The modules $\Delta(\nu)$ and $\nabla(\nu)$, for $\nu \in X_{\ell }$, have natural lifts to $\mc{G}'$. If $\lambda \in X$, the the corresponding simple module in $\mc{G}'$ is again denoted $L(\lambda)$. If $\lambda$ is chosen to lie in the set $X_{\ell}$, then by \cite[Proposition 7.1]{Lusztigmodular}, $L(\lambda)$ is the restriction of a simple $U_{\zeta}(\mf{g})$-module to $\mathbf{u}_{\zeta}'$. In particular, the character of $L(\lambda)$ in $\Z[X]$ is $W$-invariant. This is a consequence of the fact that the character of $L(\lambda)$ can be written as a $\Z$-linear combination of the characters of the induced modules $\overline{Y}_{\lambda}(z)$ of \cite[\S 6.4]{Lusztigmodular}, and the character of the latter modules is given by Weyl's character formula.

The statement in \eqref{lem:quantumblocks} implies that we must show the equality $[\Delta(\nu)]  = [\Delta(w \idot \nu)]$ in the Grothendieck group $K_0(\mc{M}(\mathbf{u}_{\zeta}'))$, where $\nu \in X_{\ell}$ and $w \in W$. We follow the argument given in \cite[page 306]{Jantzen}, where the corresponding situation for hyperalgebras is considered. For $\lambda \in X$, there is an isomorphism of $X$-graded spaces $\Delta (\lambda) \simeq  \mathbf{u}_{\zeta}^{-} \otimes_{\C} \C_{\lambda}$. This implies that 
$$
\chi(\Delta(\lambda)) = e^{\lambda} \prod_{\alpha \in R^+} \frac{1 - e^{-\ell \alpha}}{1 - e^{-\alpha}} = e^{\lambda - (\ell - 1) \rho} \chi(\Delta((\ell-1) \rho)).
$$
Notice that $\chi(\Delta((\ell-1) \rho)) \in \Z[X]^W$. As in \cite[II, 9.16 (2)]{Jantzen}, this implies that  
$$
\chi(\Delta({\ell} \rho + w \idot \lambda)) = w \chi(\Delta(\ell \rho + \lambda)) .
$$ 
We repeat, for the reader's benefit, verbatim the argument of \cite[II, 9.16]{Jantzen}. For each module $M$ in $\mc{G}'$, we have 
\begin{equation}\label{eq:charsum}
\chi(M) = \sum_{\mu_0 \in X_{\ell}} \sum_{\mu_1 \in X} [M : L(\mu_0 + \ell \mu_1)] e^{\ell \mu_1} \chi(L(\mu_0)).
\end{equation}
Using the fact that $\chi(L(\mu_0)) \in \Z[X]^W$ because $\mu_0 \in X_{\ell}$, taking $M = \Delta(\ell \rho  + \lambda)$ in (\ref{eq:charsum}) and applying $w$, we get
$$
\chi(\Delta(\ell \rho  + w \idot \lambda)) = \sum_{\mu_0 \in X_{\ell}} \sum_{\mu_1 \in X} [\Delta(\ell \rho  + \lambda) : L(\mu_0 + \ell \mu_1)] e^{\ell w \mu_1} \chi(L(\mu_0)).
$$
If we compare this with the expression for $\chi(\Delta(\ell \rho  + w \idot \lambda))$ given by (\ref{eq:charsum}) with $M = \Delta(\ell \rho  + w \idot \lambda)$, we conclude that for $\lambda, \mu_1 \in X$ and $\mu_0 \in X_{\ell}$, 
$$
[\Delta(\lambda) : L(\mu_0 + \ell \mu_1)] = [\Delta(w \idot \lambda) : L(\mu_0 + \ell w \idot \mu_1)]. 
$$
If we take $\lambda$ to be an arbitrary lift of $\nu$ to $X$, then we deduce from the above equation that $[\Delta(\nu)]  = [\Delta(w \idot \nu)]$ since the image of $L(\mu_0 + \ell \mu_1)$ and $L(\mu_0 + \ell w \idot \ \mu_1)$ in $K_0(\mc{M}(\mathbf{u}_{\zeta}'))$ are clearly equal. 

(3) The rank one property for restricted rational Cherednik algebras at $t=0$ was proven recently by Bonnafé and Rouquier \cite[Proposition 12.4.2]{BonnafeRouquier-New}.
\end{proof}

\begin{rem} \label{hyperalg_rank1}
We expect that restricted rational Cherednik algebras at $t=1$ in positive characteristic also have the rank one property, for the same reason as the $t=0$ case does. The computations in \cite[Section 4]{Pogo-Vay} imply that finite quantum groups do \textit{not} satisfy the rank one property in general. We do not know whether the hyperalgebras $\hypalg_r(\mf{g})$ for $r > 1$ satisfy the rank one property (we expect not). For $\lambda \in X$, let $r(\lambda)$ be the non-negative integer such that 
$$
\lambda \in (p^{r(\lambda)} \Z R + p^r X) \smallsetminus (p^{r(\lambda) + 1} \Z R + p^r X). 
$$
Then, by \cite[II, 9.16 (4)]{Jantzen} and \cite[II, 9.22]{Jantzen}, the hyperalgebras $\hypalg_r(\mf{g})$ satisfy the rank one property if and only if 
$$
\left[\Delta \left(\lambda + p^{r(\lambda) +1} \alpha \right)\right] = [\Delta(\lambda)] \ \textrm{ in } \ K_0(\mc{M}(\hypalg_r(\mf{g}))), \ \textrm{ for all } \ \alpha \in \Delta. 
$$
\end{rem}

\begin{prop} \label{vip_not_cellular}
Let $A$ be one of the following VIP examples:
\begin{enumerate}
\item Restricted enveloping algebras $\overline{U}(\mf{g}_K)$;
\item Lusztig's small quantum groups $\mathbf{u}_{\zeta}(\mf{g})$, at a root of unity $\zeta$ of odd order;
\item Restricted rational Cherednik algebras $\overline{\H}_{\bc}(W)$ at $t = 0$, with a non-trivial Calogero-Moser family.   
\end{enumerate}
Then $A$ is \textnormal{not} cellular.
\end{prop}

\begin{proof} All the examples satisfy the rank one property by Proposition \ref{vip_rank1}. Hence, if we can show that in each case $A$ has a block with more than one simple module, then Lemma \ref{lem:rank1VIP} implies that $A$ is not cellular.

(1) As shown in the proof of Proposition \ref{vip_rank1}, the blocks of $\overline{U}(\mf{g}_K)$  are in bijection with the $W_{\idot}$-orbits intersecting $X_{p}$.  Taking the subgroup $\s_2$ generated by any simple reflection in $W$, it suffices to show that $X_p$ contains at least one free $\s_2$-orbit. But this follows from Example \ref{ex:sl2example} in the introduction, which shows that when $p$ is odd, the algebra $\overline{U}((\mf{sl}_{2})_K)$ contains at least one block with more that one simple object (and hence $X_{p}$ has at least one free $\s_2$-orbit). 

(2) The proof is completely analogous to the above. Using the analogue of Example \ref{ex:sl2example} from the introduction, the fact that we have assumed that $\ell$ is odd implies that the algebra $\mathbf{u}_{\zeta}$ contains at least one block with more that one simple object (this time applying statement \eqref{lem:quantumblocks}). 

(3) The Calogero-Moser families of $\overline{\H}_{\bc}(W)$ is the partition of $\Irr W$ induced by the block decomposition of $\overline{\H}_{\bc}(W)$. Thus, the algebra has a non-trivial Calogero-Moser family if and only if it has a block with more than one simple. 
\end{proof}

\begin{rem}
	It is known by the results in \cite{Martino-blocks-gmpn} and \cite{Singular} that $\overline{\H}_{\bc}(W)$ at $t = 0$ will have a non-trivial Calogero-Moser family if $W$ has a factor not isomorphic to $G(m,1,n)$ or $G_4$. If each factor of $W$ is isomorphic to some $G(m,1,n)$ or $G_4$, then $\overline{\H}_{\bc}(W)$ at $t = 0$ will have a non-trivial Calogero-Moser family if and only if $\bc$ non-generic.
\end{rem}	

\begin{rem}
	If $A$ is a restricted rational Cherednik algebras $\overline{\H}_{\bc}(W)$ at $t = 0$, or a RRCA $\overline{\H}_{1,\bc}(W)$ at $t = 1$ in positive characteristic $\neq 2$, with the property that each block of $A$ contains only one simple module (i.e. the Calogero-Moser families are all trivial), then $A$ is cellular. To see this, we note first that a direct sum of cellular algebras is cellular. Therefore it suffices to show that each block $B$ of $A$ is cellular. It has been shown in \cite{Ramifications} that in all of the above examples, the fact that $B$ only has one simple module implies that there is a local commutative ring $\mc{O}$ such that $B \simeq \mathrm{Mat}_n(\mc{O})$, for some $n$. It is known by \cite[Proposition 3.5]{KonigXiCellular} that $\mc{O}$ is cellular if one takes the anti-involution $i$ to be the identity. Therefore we take $i'$ to be the transpose on $B = \mathrm{Mat}_n(\mc{O})$. Then it follows from \cite[Corollary 6.16]{KonigXiCellularMort} that $B$ is cellular, provided $\mathrm{char} \ K \neq 2$.
\end{rem}

\section{Highest weight cover of the core} \label{sec:cover}

\newcommand{\proj}[1]{#1\text{-}{\mathsf{proj}}}

As shown in Proposition \ref{A0_quasi_hereditary}, the algebra $A_0$ is very rarely quasi-hereditary. In general, it has infinite global dimension. In order to ``resolve'' this ``singular'' algebra, we show that a certain subquotient of the highest weight category $\mc{G}(A)$ defines a highest weight cover of $\mc{M}(A_0)$. 

\subsection{Highest weight cover—statements}\label{sec:coversmain}

In this section we state our main results. The proofs, which are rather involved, are then given in the following sections. Recall from \cite[\S 4.3]{hwtpaper} that the category $\mc{G}(A)$ is filtered by the Serre subcategories $\mathcal{G}_{\leq d}(A)$. Here an object $M$ belongs to $\mathcal{G}_{\leq d}(A)$ if and only if $[M : L(\lambda) ] \neq 0$ implies that $\deg \lambda \le d$. As we note in section \ref{sec_filtered_pieces}, these are highest weight subcategories. For $d \ge 0$, the subfactors 
$$
\mathcal{G}_{[d]}(A) := \mathcal{G}_{\leq d}(A) / \mathcal{G}_{< 0}(A)
$$
are also highest weight categories. Let $N \ge 0$ be the smallest positive integer such that 
$\Supp A^{+} \subs [0,N]$. 

\begin{tcolorbox}
	Throughout, we fix $d \geq N$ and set $\ell \dopgleich d - N$.
\end{tcolorbox}

If $i$ is positive integer and $M$ an object in $\mc{G}(A)$, we set  
$$
M \langle i \rangle := M[0] \oplus M[1] \oplus \cdots \oplus M[i] \;.
$$
Define 
\begin{equation*}
	B_{\ell} := \End_{\mc{G}(A)}(A \langle \ell \rangle ) \;.
\end{equation*}
Then 
$$
B_{\ell}^{\op} = \left( \begin{array}{cccc}
A_0 & A_1 & \cdots & A_{\ell} \\
A_{-1} & A_0 & & \vdots \\
\vdots & & \ddots & A_{1} \\
A_{-\ell} & \cdots & A_{-1} & A_0
\end{array} \right) \;.
$$
In particular, $B_0 = A_0^{\op}$. Our first main result, to be proven in Section \ref{cover_proof_section1}, is:

\begin{thm}\label{thm:A0elling} \hfill
	
	\begin{enum_thm}
		\item $A\langle \ell \rangle$ is a projective object in $\mc{G}_{[d]}(A)$. If $A$ is graded symmetric, it is also injective.  
		\item $\End_{\mc{G}_{[d]}(A)}(A\langle \ell \rangle) = B_{\ell}$.
	\end{enum_thm}
\end{thm}

As a corollary we obtain:

\begin{cor}
	Suppose that $A$ is graded symmetric. Then $B_\ell$ is standardly based. If $A$ is moreover equipped with a triangular anti-involution, then $B_\ell$ is a cellular algebra.
\end{cor}

Rouquier \cite{RouquierQSchur} introduced the key notion of highest weight covers. We recall that if $B$ is a finite dimensional $K$-algebra, then a \textit{highest weight cover} of $B$ is a highest weight category $\mc{B}$ with finitely many simple objects, together with a projective object $P \in \mc{B}$ such that $B \simeq \End_{\mc{B}}(P)^{\mathrm{op}}$ and the corresponding exact functor $F = \Hom_{\mc{B}}(P,-) \from \mc{B} \to \mc{M}(B)$ is fully faithful on projectives. Since $A[0,\ell]$ is a projective object in $\mc{G}_{[d]}(A)$ by Theorem \ref{thm:A0elling}, we have an exact functor
\begin{equation*}
	F_d \dopgleich \Hom_{\mc{G}_{[d]}(A)}(A\langle \ell \rangle, -) \from \mc{G}_{[d]}(A) \to \mc{M}(B_{\ell}) \;.
\end{equation*}
For our second main result, recall the notion of ambidexterity from \cite[Definition 3.4]{hwtpaper}. We also introduce the following terminology:

\begin{defn}
	We say that $A$ is \word{well-generated} if $A^\pm$ is generated, as a $K$-algebra, by elements in degree $\pm 1$.
\end{defn}

\begin{thm}\label{thm:highestweightcover}
	Assume that $A$ is graded symmetric, well-generated, and ambidextrous. 
	Then the functor $F_d : \mc{G}_{[d]}(A) \rightarrow \mc{M}(B_{\ell})$ is a highest weight cover of~$B_{\ell}$. 
\end{thm}

At the heart of the proof of Theorem \ref{thm:highestweightcover} is a technical result, Lemma \ref{lem:soclAlhdsupport}, on the socle of certain graded modules over the algebras $A^+$ and $A^-$. Remarkably, though the statement of Lemma \ref{lem:soclAlhdsupport} is the same for both algebras, they play opposite roles. Namely, Lemma \ref{lem:soclAlhdsupport} for $A^-$ implies that $F_d$ is faithful on projectives - in fact it is faithful on all standardly filtered modules. On the other hand, Lemma \ref{lem:soclAlhdsupport} for $A^+$ implies that $F_d$ is full on projectives, since it implies the $\Ext^1$-vanishing required in Proposition \ref{prop:covercriterion} (b). Taking $d = N$,  we deduce:

\begin{cor}
	Suppose that $A$ is graded symmetric, well-generated, and ambidextrous. Then the functor $F_N: \mc{G}_{[N]}(A) \rightarrow \mc{M}(A_0^{\mathrm{op}})$ is a highest weight cover. 
\end{cor}

The functor $F_d$ induces an algebra morphism 
\begin{equation*}
	\Phi_d : \mrm{Z}(\mc{G}_{[d]}(A)) \rightarrow \mrm{Z}(B_{\ell})
\end{equation*}
between the centers. Theorem \ref{thm:highestweightcover} implies:  

\begin{cor}\label{cor:Fcentre} Assume that $A$ is graded symmetric, well-generated, and ambidextrous. Then $\Phi_d$ is an isomorphism and the functor $F_d$ induces a bijection between the blocks of $\mc{G}_{[d]}(A)$ and $\mc{M}(B_{\ell})$.  
\end{cor}

The above statements are proven in the following sections. We will actually prove them under weaker hypotheses to clarify which assumptions are needed for the theory.

\begin{rem}
	The reader might wonder why we consider the algebras $B_{\ell}$ instead of just $A_0^{\op}$. The motivation comes from the fact that it is an important problem to try and give an explicit description of the quasi-hereditary algebras $C_d$. We note that, firstly, it is much easier to describe $B_{\ell}$ than it is to try and explicitly describe $C_d$. Secondly, the parametrization of simple objects in $\mc{M}(C_d)$, together with the partial ordering on these objects is very easy to describe. These facts combined imply that it seems likely one can use Rouquier's theorem on the uniqueness of highest weight covers to try and identify $C_d$ with some ``known'' highest weight category. 
\end{rem}

\subsection{Filtered pieces} \label{sec_filtered_pieces}

In this section, we consider in more detail the Serre subcategories $\mathcal{G}_{\leq d}(A)$. Let $M \in \mc{G}(A)$. There is a unique largest quotient ${}^{\perp} H_d(M)$ of $M$ that belongs to $\mc{G}_{\le d}(A)$, namely 
\begin{equation*}
{}^{\perp} H_d(M) = M / A \cdot M_{>d}.
\end{equation*}
 This defines a right exact functor 
\begin{equation*}
{}^{\perp} H_d : \mc{G}(A) \rightarrow \mc{G}_{\le d}(A) \;.
\end{equation*}
 Similarly, let $H^{\perp}_d(M)$ be the largest submodule of $M$ that belongs to $\mc{G}_{\le d}(A)$ and let 
 \begin{equation*}
H^{\perp}_d : \mc{G}(A) \rightarrow \mc{G}_{\le d}(A)
\end{equation*}
be the corresponding functor, which is left exact.  Note that
\begin{equation*}
(-)^* \circ H^{\perp}_d = {}^{\perp} H_d \circ (-)^* \;.
\end{equation*}
The functor ${}^{\perp} H_d$, resp. $H^{\perp}_d$, is left, resp. right, adjoint to the inclusion functor $\mc{G}_{\le d}(A) \hookrightarrow \mc{G}(A)$. Recall that a full subcategory $\mc{B}$ of a category $\mc{A}$ is said to be \textit{reflective}, resp. \textit{coreflective}, if the inclusion functor $\mc{B} \hookrightarrow \mc{A}$ admits a left, resp. right, adjoint $\mc{A} \rightarrow \mc{B}$.  

\begin{lem} \label{prop_filtered_piece_lemma}
$\mathcal{G}_{\leq d}(A)$ is both a reflective and coreflective subcategory of $\mathcal{G}(A)$. In particular:
\begin{enum_thm}
\item \label{prop_filtered_piece_lemma:closed}  $\mathcal{G}_{\leq d}(A)$ is closed under those limits and colimits which exist in $\mathcal{G}(A)$.
\item A morphism in $\mathcal{G}_{\leq d}(A)$ is a monomorphism (epimorphism) if and only if it is a monomorphism (epimorphism) in $\mathcal{G}_{\leq d}(A)$.
\item ${}^{\perp} H_d$ preserves projectives and $H^{\perp}_d$ preserves injectives.
\item \label{prop_filtered_piece_lemma:enough} $\mathcal{G}_{\leq d}(A)$ has enough projectives and enough injectives.
\item If $M \in \mathcal{G}_{\leq d}(A)$ and $P \twoheadrightarrow M$ is the projective cover of $M$ in $\mathcal{G}(A)$, then ${}^{\perp} H_d(P) \rarr {}^{\perp} H_d(M)$ is the projective cover of ${}^{\perp} H_d(M)$ in $\mathcal{G}_{\leq d}(A)$. The analogous statement holds for injective hulls.
\item Brauer reciprocity \cite[(55)]{hwtpaper} and its dual version \cite[(58)]{hwtpaper} hold in $\mathcal{G}_{\leq d}(A)$.
\end{enum_thm}
\end{lem}

\begin{proof}
All statements follow directly from the fact that the inclusion $\mc{G}_{\le d}(A) \hookrightarrow \mc{G}(A)$ is exact, has both a left adjoint ${}^{\perp} H_d$ and right adjoint $H^{\perp}_d$, and these adjoints are the identity on $\mathcal{G}_{\leq d}(A)$.
\end{proof}

If $\deg \lambda \leq d$, so that $L(\lambda)$ belongs to $\mathcal{G}_{\leq d}(A)$, then in particular
\begin{equation*}
Q_d (\lambda) \dopgleich {}^{\perp} H_d(P(\lambda))
\end{equation*}
is the projective cover of $L(\lambda)$ in $\mc{G}_{\le d}(A)$ and
\begin{equation*}
J_d (\lambda) \dopgleich H^{\perp}_d(I(\lambda))
\end{equation*}
is the injective hull of $L(\lambda)$ in $\mc{G}_{\le d}(A)$. We note that if $\deg \lambda > d$, then ${}^{\perp} H_d(P(\lambda)) = 0$ and $H^{\perp}_d(I(\lambda)) = 0$. Let $\Irr_{\le d} \mc{G}(T)$ be the ideal in $\Irr \mc{G}(T)$ of all modules with degree at most $d$. A standard argument, see for instance \cite[Theorem 3.5]{CPS}, shows:

\begin{prop}
The category $\mc{G}_{\le d}(A)$ is a highest weight category, with standard modules ${}^{\perp} H_d(\Delta(\lambda)) = \Delta(\lambda)$ and costandard modules $H^{\perp}_d(\nabla(\lambda)) = \nabla(\lambda)$. 
\end{prop}

To describe the projective cover and injective hull more explicitly, recall from \cite[Lemma 4.19]{hwtpaper} that there is a standard filtration $P(\lambda) = F_{0} \supset F_{1} \supset \cdots \supset F^l = 0$ of $P(\lambda) \in \mathcal{G}(A)$ with $F_i / F_{i+1} \simeq \Delta(\lambda_i)$ such that $\deg \lambda_i  \leq \deg \lambda_{i+1} $ for all $i$. It follows that there is some $k(d) \in \bbZ$ such that $d< \deg \lambda_i$ if and only if $i \ge k(d)$. Similarly, there is a costandard filtration $\{ 0\} = F_{0} \subset F_{1} \subset \cdots \subset F_m = I(\lambda)$ such that $F_{i} / F_{i-1} \simeq \nabla(\lambda_i)$ with $\deg \lambda_i \leq \deg \lambda_{i+1}$ and there is some $l(d)$ such that $d < \deg \lambda_i$ if and only if $i > l(d)$. With these notations we obtain:

\begin{lem}\label{lem:projcoversub}
Let $\deg \lambda \leq d$. Then $Q_d (\lambda) = P(\lambda) / F_{k(d)}$ and $J_d (\lambda) =  F_{l(d)}$.
\end{lem}

\begin{proof}
For $i < k(d)$ we have $d \ge \deg \lambda_i$, so $\Delta(\lambda_i) \in \mc{G}_{\le d}(A)$. Hence, $P(\lambda) / F_{k(d)} \in \mc{G}_{\le d}(A)$. Since $Q_d(\lambda) = {}^{\perp} H_d(P(\lambda))$ is the largest quotient of $P(\lambda)$ belonging to $\mathcal{G}_{\leq d}(A)$, it follows that $P(\lambda) / F_{k(d)} \in \mc{G}_{\le d}(A)$ is a quotient of $Q_d(\lambda)$. On the other hand, by Brauer reciprocity in $\mc{G}_{\le d}(A)$, we obtain 
$$
[Q_d(\lambda) : \Delta(\mu)] = [\nabla(\mu) : L(\lambda)] = [P(\lambda) / F_{k(d)} : \Delta(\mu)]
$$
for all $\mu$ with $\deg \mu \leq d$. Hence, $Q_d(\lambda) = P(\lambda) / F_{k(d)}$. The proof for injective hulls is analogous.
\end{proof}

\begin{lem} \label{proj_cover_leq_d_lambda}
If $\deg \lambda = d$, then $Q_d(\lambda) = \Delta(\lambda)$ and $J_d(\lambda) = \nabla(\lambda)$. \qed
\end{lem} 

\begin{rem}
Even if $P(\lambda) = I(\lambda)$ in $\mc{G}(A)$, it is not generally the case that $Q_d(\lambda) = J_d(\lambda)$ in $\mc{G}_{\le d}(A)$. For instance, consider $A = K[x,y] / (x^2, y^2)$ with triangular structure as in \cite[Example 3.2]{hwtpaper} and take $d = 0$. We have $T=K$, thus there is only the simple $T$-module $K$. Considering it as a graded $T$-module concentrated in degree $m \in \bbZ$, we denote it simply by $m$. Then Lemma \ref{lem:projcoversub} implies that the projective cover $Q_0(0)$ of $L(0)$ in $\mc{G}_{\le 0}(A)$ equals $\Delta(0)$ and its injective hull $J_0(0)$ is $\nabla(0)$. In this case, $\Delta(0) \neq \nabla(0)$. This is because both the projective cover and injective hull of $L(0)$ need to be ``trimmed'' to fit in $\mc{G}_{\le 0}(A)$. Moving further down, we have $Q_0(m) = J_0(m)$ for all $m < 0$. One can easily produce more involved examples. 
\end{rem}

In order to illustrate the rather atypical behavior of the highest weight categories considered here, we turn briefly to the homological dimension of objects in the various highest weight categories. If $\mc{A}$ is an abelian category with enough projectives then let $\mathrm{p.d.}_{\mc{A}}(M)$ denote the projective dimension of $M \in \mc{A}$. Similarly, $\mathrm{i.d.}_{\mc{A}}(M)$ denotes the injective dimension of $M$ if $\mc{A}$ has enough injectives. 

\begin{lem} 
In $\mc{G}_{\le d}(A)$ we have
\begin{enum_thm}
\item $\mathrm{p.d.}_{\mc{G}_{\le d}(A)}(\Delta(\lambda)) \le d - \deg \lambda$.
\item $\mathrm{i.d.}_{\mc{G}_{\le d}(A)}(\nabla(\lambda)) \le d - \deg \lambda $.
\end{enum_thm}
Hence, $\mathrm{p.d.}_{\mc{G}_{\le d}(A)}(M) < \infty$ for all $M \in \mc{G}_{\le d}^{\Delta}(A)$ and $\mathrm{i.d.}_{\mc{G}_{\le d}(A)}(M) < \infty$ for all $M \in \mc{G}_{\le d}^{\nabla}(A)$.
\end{lem}

\begin{proof}
This is the usual induction argument, using Lemma \ref{proj_cover_leq_d_lambda} 
\end{proof}

In general, the quantities $\mathrm{p.d.}_{\mc{G}_{\le d}(A)}(\nabla(\lambda))$ and $\mathrm{i.d.}_{\mc{G}_{\le d}(A)}(\Delta(\lambda))$ are infinite, as is both the injective and projective dimension of $L(\lambda)$. In the larger category $\mc{G}(A)$, the simples, standards and costandards can all have infinite projective and injective dimension. All of these claims can easily be checked to be true for our favorite example $A = K[x,y] / (x^2,y^2)$.

\subsection{Subquotients}\label{sec:quasihereditary}

The category $\mc{G}(A)$ is a highest weight category with infinitely many simple objects. In this section we consider certain subquotients of $\mc{G}(A)$ that are themselves highest weight categories, but have only finitely many simple objects. Thus, they are equivalent to the module category of some quasi-hereditary algebra. 

If $d \ge 0$, then $\mc{G}_{< 0}(A) := \mc{G}_{\le -1}(A)$ is a Serre subcategory of $\mc{G}_{\le d}(A)$ and we denote by
\begin{equation*}
\mc{G}_{[d]}(A) \dopgleich \mc{G}_{\leq d}(A)/\mc{G}_{<0} 
\end{equation*}
the quotient category. We have 
\begin{equation*}
\Irr \mc{G}_{[d]}(T) = \{ \lambda \in \Irr \mc{G}(T) \ |  \ \deg \lambda \in [0,d] \} \;.
\end{equation*}

\begin{lem}\label{lem:porjinjquot}
Let $0 \le \deg \lambda \le d$. The projective cover of $L(\lambda) \in \mc{G}_{[d]}(A)$ is $Q_d(\lambda)$, and the injective hull of $L(\lambda)$ is $J_d(\lambda)$. 
\end{lem}

\begin{proof}
For $0 \le \deg \lambda, \deg \mu \le d$, it is easily seen that 
$$
\Hom_{\mc{G}_{[d]}(A)}(Q_d(\lambda),L(\mu)) = \Hom_{\mc{G}_{\le d}(A)}(Q_d(\lambda),L(\mu))
$$
and 
$$
\Hom_{\mc{G}_{[d]}(A)}(L(\lambda),J_d(\mu)) = \Hom_{\mc{G}_{\le d}(A)}(L(\lambda),J_d(\mu))
$$ 
Therefore it suffices to show that $Q_d(\lambda)$ is projective and $J_d(\lambda)$ is injective. Since the proofs are similar, we only show that $J_d(\lambda)$ is injective. Let 
$$
\begin{tikzcd}
0 \arrow{r} & K \arrow{d}[swap]{\phi} \arrow{r}  & L \arrow[dotted]{dl}{\psi} \arrow{r} & M \arrow{r} & 0\\
 & J_d(\lambda) & & & 
\end{tikzcd}
$$
be a commutative diagram in $\mc{G}_{[d]}(A)$, with the row being exact. Then we need to show that the morphism $\psi$ exists, preserving commutativity. By \cite[Corollary 15.8]{FaithAlgebraI}, there exists a short exact sequence $0 \rightarrow K' \rightarrow L' \rightarrow M' \rightarrow 0$ in $\mc{G}_{\le d}(A)$ descending to the one above in $\mc{G}_{[d]}(A)$. Since $\phi \in \Hom_{\mc{G}_{[d]}(A)}(K',J_d(\lambda))$, there exist submodules $K'' \subset K'$ and $J' \subset J_d(\lambda)$ such that $K'/K''$ and $J'$ belong to $\mc{G}_{<0}(A)$ and $\phi$ is represented by $\phi'' : K'' \rightarrow J_d(\lambda) / J'$. However, $L(\lambda) = \Soc_A (J_d(\lambda))$ does not belong to $\mc{G}_{< 0}(A)$. Thus, $J' = 0$. Moreover, since $J_d(\lambda)$ is injective, the morphism $\phi''$ is the restriction of a morphism $\phi' : K' \rightarrow J_d(\lambda)$. Thus, again by the injectivity of $J_d(\lambda)$, we get a commutative diagram 
$$
\begin{tikzcd}
0 \arrow{r} & K' \arrow{d}{\phi'} \arrow{r}  & L' \arrow{dl}{\psi'} \arrow{r} & M' \arrow{r} & 0\\
 & J_d(\lambda) & & & 
\end{tikzcd}
$$
in $\mc{G}_{\le d}(A)$. The morphism $\psi'$ descends to the required $\psi$. 
\end{proof}

By restriction, $\Irr \mc{G}_{[d]}(T)$ is a partially ordered set. The following result is standard, see \cite[Theorem 3.5]{CPS}.

\begin{prop}\label{prop:subquotequiv}
The pair $(\mc{G}_{[d]}(A), \Irr \mc{G}_{[d]}(T))$ is a highest weight category with finitely many simple objects. It is Morita equivalent to the module category $\mc{M}(C_d)$ of the quasi-hereditary algebra
$$
C_d \dopgleich \End_{\mc{G}_{\le d}(A)}(Q)^{\op}, \quad \textrm{where} \quad Q \dopgleich \bigoplus_{\lambda \in \Irr \mc{G}_{[d]}(T)} Q_d(\lambda) \;.
$$
\end{prop}

We note another consequence of the general theory of quasi-hereditary algebras. 

\begin{lem}\label{lem:cornering}
For each $0 \le t \le d$, there is an idempotent $e_t \in C_d$ such that $C_t \simeq e_t C_d e_t$. 
\end{lem}

\begin{proof}
The subset $\Irr \mc{G}_{[d-t]}(T) \subset \Irr \mc{G}_{[d]}(T)$ is an ideal. Let $\Omega$ be the complementary coideal. By \cite[Lemma 3.4]{CPS}, there is a hereditary ideal $J \lhd C_d$ such that $\Hom_{C_d}(J,L(\lambda)) \neq 0$ if and only if $\lambda \in \Omega$. By \cite[Statement 6]{QuasiHerediatryDlabRingel}, $J = C_d e_t C_d$ for some idempotent $e_t$. It is shown in the proof of \cite[Corollary 3.7]{CPS} that the quotient category $\mc{G}(\Omega) = \mc{G}_{[d]}(A) \, / \, \mc{G}_{[d-t]}(A)$ is equivalent to $\mc{M}(e_t C_d e_t)$. On the other hand, shifting degree by $d-t$ defines an equivalence
$$
[d-t] : \mc{G}(\Omega) \stackrel{\sim}{\longrightarrow} \mc{G}_{[t]}(A).
$$
Since $\mc{M}(C_t) \simeq \mc{G}_{[t]}(A)$ and both $e_t C_d e_t$ and $C_t$ are basic algebras, we conclude that $C_t \simeq e_t C_d e_t$. 
\end{proof}

\subsection{Proof of Theorem \ref{thm:A0elling}} \label{cover_proof_section1}
To make things manageable, we break the proof of Theorem \ref{thm:A0elling} into a series of smaller results. As previously mentioned, we will prove it in slightly greater generality, assuming that $A$ is only \word{triangular self-injective} (see Definition \ref{def_triangular_selfinj} below) instead of graded symmetric.

For $\lambda \in \Irr \mc{G}(T)$ let $r_{\lambda} \ge 0$ so that $\Supp L(\lambda) =  [\deg \lambda -  r_{\lambda},\deg \lambda]$. Notice that $r_{\lambda[i]} = r_{\lambda}$ for all $i$ and that $r_{\lambda} \le N$ since $L(\lambda) \subset \nabla(\lambda)$. Recall that $Q_d(\lambda) := {}^{\perp}H_d(P(\lambda))$ is the projective cover of $L(\lambda)$ in $\mc{G}_{\le d}(A)$. We will denote by the same symbol the image of $Q_d(\lambda)$ in $\mc{G}_{[d]}(A)$. 

\begin{lem}\label{lem:obviousdecomp}
As graded left $A$-modules, 
$$
A[i] = \bigoplus_{\lambda \in \Irr \mc{G}(T)} P(\lambda) \otimes_K L(\lambda)_i.
$$
In particular, $P(\lambda)$ is a non-zero summand of $A[i]$ if and only if $i \in [\deg \lambda - r_{\lambda}, \deg \lambda]$.  
\end{lem}

\begin{proof}
This follows from the fact that the multiplicity of $P(\lambda)$ in $A[i]$ equals 
$$
\dim \Hom_{\mc{G}(A)}(A[i],L(\lambda)) = \dim \Hom_{\mc{G}(A)}(A,L(\lambda[-i])) = \dim L(\lambda)_i
$$
and that $\Supp L(\lambda) = [\deg \lambda - r_{\lambda}, \deg \lambda]$.
\end{proof}

\begin{lem}\label{lem:basicsubcat} \hfill

\begin{enum_thm}
\item \label{lem:basicsubcat_Hd} ${}^{\perp} H_d(A[i]) = A[i]$ if and only if $i \le \ell$. 
\item If $\deg \lambda + N \le d + r_{\lambda}$ then $Q_d(\lambda) = P(\lambda)$. 
\item If $A$ is self-injective, then $Q_d(\lambda)$ is injective in $\mc{G}_{\le d}(A)$ if and only if $Q_d(\lambda) = P(\lambda)$. 
\end{enum_thm}
\end{lem}

\begin{proof}
Recall that ${}^{\perp} H_d(M) = M / A \cdot M_{> d}$. Thus, ${}^{\perp} H_d(A[i]) = A[i]$ if and only if $A[i]_{> d} = 0$. Since the top degree of $A$ is $A_{N}$, statement \ref{lem:basicsubcat_Hd} follows. By Lemma  \ref{lem:obviousdecomp}, $P(\lambda)$ is a direct summand of $A[\deg \lambda - r_{\lambda}]$. Thus, if $\deg \lambda - r_{\lambda} \le \ell$, then ${}^{\perp} H_d(A[\deg \lambda - r_{\lambda}]) = A[\deg \lambda - r_{\lambda}]$ implies that ${}^{\perp} H_d(P(\lambda)) = P(\lambda)$. Finally, we note that if $Q_d(\lambda)$ is injective in $\mc{G}_{\le d}(A)$ then it has a costandard filtration. This implies that it has both a standard filtration and a costandard filtration when considered as an element in $\mc{G}(A)$. Therefore \cite[Corollary 5.7]{hwtpaper} implies that $Q_d(\lambda)$ is projective as an element of $\mc{G}(A)$. But it is a quotient of $P(\lambda)$. Thus, $Q_d(\lambda) = P(\lambda)$. 
\end{proof}

To make the ingredients of our results clearer, we introduce the following notion.

\begin{defn} \label{def_triangular_selfinj}
We say that $A$ is \word{triangular self-injective} if the graded Nakayama permutation $\nu$ of $A$ preserves degrees, i.e., $\deg L(\lambda) = \deg L(\nu(\lambda))$ for all $\lambda \in \Irr \mc{G}(T)$.
\end{defn}

Graded symmetric algebras are triangular self-injective since the graded Nakayama permutation is trivial by \cite[Lemma 5.6]{hwtpaper}. This is the main example we have in mind.

\begin{lem} \label{A_0_ell_projective}
$A \langle \ell \rangle$ is a projective object in $\mc{G}_{\lbrack d \rbrack}(A)$. If $A$ is triangular self-injective, then it is also injective.
\end{lem}

\begin{proof}
By Lemma \ref{lem:basicsubcat}\ref{lem:basicsubcat_Hd}, ${}^{\perp}H_d(A\langle \ell \rangle) = A\langle \ell \rangle$ is a projective module in $\mc{G}_{\le d}(A)$. To show that its image in $\mc{G}_{[d]}(A)$ is projective, Lemma \ref{lem:porjinjquot} implies that it suffices to show that if $Q_d(\lambda)$ is a non-zero summand of $A[0,\ell]$, then $0 \le \deg \lambda \le d$. But this follows from Lemma \ref{lem:obviousdecomp}. This also implies that the head $L(\lambda)$ of each $Q_d(\lambda)$ appearing in $A\langle \ell \rangle$ satisfies $0 \le \deg \lambda \le d$. If $A$ is triangular self-injective, then $\Soc P(\lambda) = L(\mu)$ for some $\mu$ such that $\deg \mu = \deg \lambda$, and hence $P(\lambda) = I(\mu)$. To show that the image of $A\langle \ell \rangle$ in $\mc{G}_{[d]}(A)$ is injective, it suffices once again by Lemma \ref{lem:porjinjquot} to show that if $J_d(\mu)$ is a non-zero summand of $A[0,\ell]$, then $0 \le \deg \mu \le d$. But, as noted above, if $L(\lambda)$ is the head of such a $J_d(\mu)$, then $0 \le \deg \lambda \le d$. Since $\deg \mu = \deg \lambda$, the result follows.
\end{proof}

\begin{lem}\label{lem:endGdA}
$B_{\ell} = \End_{\mc{G}_{[d]}(A)}(A\langle \ell \rangle)$. 
\end{lem}

\begin{proof}
We note that $B_{\ell} = \End_{\mc{G}(A)}(A\langle \ell \rangle)$, and that $A[i] = \bigoplus_{\lambda \in \Irr \mc{G}(T)} P(\lambda) \otimes_K L(\lambda)_i$. Therefore, we wish to show that 
\begin{align*}
& \End_{\mc{G}(A)} (A\langle \ell \rangle) \\
& = \bigoplus_{i,j \in [0,\ell]} \bigoplus_{\lambda,\mu \in \Irr \mc{G}_{\le d}(T)}  \Hom_{\mc{G}_{\le d}(A)}(Q_d(\lambda),Q_d(\mu)) \otimes_K \Hom_{K}(L(\lambda)_i,L(\mu)_j)  \\
 & = \bigoplus_{i,j \in [0,\ell]} \bigoplus_{\lambda,\mu \in \Irr \mc{G}_{[d]}(T)} \Hom_{\mc{G}_{[d]}(A)}(Q_d(\lambda),Q_d(\mu)) \otimes_K \Hom_{K}(L(\lambda)_i,L(\mu)_j)\;.
\end{align*}
Since the support of $L(\lambda)$ is contained in $[\deg \lambda - N, \deg \lambda]$, it follows that $L(\lambda)_i = 0$ unless $0 \le \deg \lambda \le d$. Therefore, the second equality follows from Proposition \ref{prop:subquotequiv}, which implies that 
$$
\Hom_{\mc{G}_{\le d}(A)}(Q_d(\lambda),Q_d(\mu)) = \Hom_{\mc{G}_{[d]}(A)}(Q_d(\lambda),Q_d(\mu))
$$
for all $\lambda,\mu \in \Irr \mc{G}_{[d]}(T)$ with $0 \le \deg(\lambda), \deg(\mu) \le d$. 

Thus, we concentrate on establishing the first equality. The functor ${}^{\perp} H_d$ defines a canonical map $\Hom_{\mc{G}(A)}(P(\lambda),P(\mu)) \rightarrow \Hom_{\mc{G}_{\le d}(A)}(Q_d(\lambda),Q_d(\mu))$. It suffices to show that the induced map   
\begin{multline}\label{eq:somesubid}
 \Hom_{\mc{G}(A)}(P(\lambda),P(\mu)) \otimes_K \Hom_{K}(L(\lambda)_i,L(\mu)_j) \rightarrow \\ \Hom_{\mc{G}_{\le d}(A)}(Q(\lambda),Q(\mu)) \otimes_K \Hom_{K}(L(\lambda)_i,L(\mu)_j)
\end{multline}
is an isomorphism provided $i,j \in [0,\ell]$. First, we show that the morphism 
\begin{equation}\label{eq:morphi45}
\Hom_{\mc{G}(A)}(P(\lambda),P(\mu)) \rightarrow \Hom_{\mc{G}_{\le d}(A)}(Q(\lambda),Q(\mu))
\end{equation}
is always surjective. Any map $\phi \in \Hom_{\mc{G}_{\le d}(A)}(Q(\lambda),Q(\mu))$ can be lifted to a map $\phi' : P(\lambda) \rightarrow Q(\mu)$. Since $P(\lambda)$ is projective, $\phi'$ is the composite of a morphism $\phi'' : P(\lambda) \rightarrow P(\mu)$ with the quotient $P(\mu) \rightarrow Q(\mu)$. Thus, $\phi''$ maps to $\phi$ under (\ref{eq:morphi45}).

Even for $0 \le \deg(\lambda), \deg(\mu) \le d$, the map 
$$
\Hom_{\mc{G}(A)}(P(\lambda),P(\mu)) \rightarrow \Hom_{\mc{G}_{\le d}(A)}(Q_d(\lambda),Q_d(\mu))
$$
is not in general injective. Since the head $L(\lambda)$ of $P(\lambda)$ is an object of $\mc{G}_{\le d}(A)$, 
$$
 \Hom_{\mc{G}_{\le d}(A)}(Q_d(\lambda),Q_d(\mu)) =  \Hom_{\mc{G}(A)}(P(\lambda),{}^{\perp} H_d(P(\mu)))
$$
and hence a morphism $\phi \in \Hom_{\mc{G}(A)}(P(\lambda),P(\mu))$ is mapped to zero if and only if its image is a submodule of 
$$
M := \Ker (P(\mu) \rightarrow {}^{\perp}H_d(P(\mu))) = A \cdot P(\mu)_{ > d}.
$$
Lemma \ref{lem:projcoversub} shows that $M$ has a standard filtration with subquotients $\Delta(\mu')$ for $\deg \mu'  > d$. Thus, if $\phi$ has non-zero image in $M$, then $[M : L(\lambda)] \neq 0$. But this implies that $\Supp L(\lambda) \subset \Supp \Delta(\mu')$ for some $\mu'$ with $\deg \mu' > d$. Since $[0,\ell] \cap \Supp \Delta(\mu') = \emptyset$, we deduce that $L(\lambda)_i = 0$ for all $i \in [0,\ell]$. Thus, the map (\ref{eq:somesubid}) is injective. 
\end{proof}

\subsection{Proof of Theorem \ref{thm:highestweightcover}} \label{cover_proof_section}
In this section we give the proof of Theorem \ref{thm:highestweightcover}. As before, we will prove this in slightly greater generality by employing an assumption on the socle of $A^\pm$, namely that $\Soc A^\pm = A^\pm_{\pm N}$.  In Section \ref{sec_socle_assumption} we will show that this assumption holds if $A$ is graded symmetric and ambidextrous.

\begin{lem}
For each $\lambda \in \Irr \mc{G}_{[d]}(T)$, $F_d(L(\lambda)) = 0$ if and only if $\Supp L(\lambda) \cap [0,\ell] = \emptyset$.
\end{lem}

\begin{proof}
Let $M$ be the smallest submodule of $A\langle \ell \rangle$ such that $A\langle \ell \rangle /M \in \mc{G}_{< 0}(A)$. Then, by definition of quotient category, the space $\Hom_{\mc{G}_{[d]}(A)}(A\langle \ell \rangle,L(\lambda))$ equals $\Hom_{\mc{G}_{\le d}(A)}(M,L(\lambda))$. Therefore, it suffices to show in this case that $M = A[0,\ell]$. By Lemma \ref{lem:obviousdecomp}, $Q_d(\lambda)$ is a non-zero summand of $A[0,\ell]$ if and only if $L(\lambda)_i \neq 0$ for some $i \in [0,\ell]$. Therefore, it suffices to show that the smallest submodule $R$ of $Q_d(\lambda)$ such that $Q_d(\lambda) / R \in \mc{G}_{< 0}(A)$ is $Q_d(\lambda)$ itself, when $L(\lambda)_i \neq 0$ for some $i \in [0,\ell]$. But if $Q_d(\lambda) / R \neq 0$ then, in particular, $L(\lambda) \in \mc{G}_{< 0}(A)$. This contradicts the fact that $L(\lambda)_i \neq 0$ since $\Supp L(\lambda) \subset \N_{< 0}$ for all simple objects $L(\lambda)$ of $\mc{G}_{< 0}(A)$. 
\end{proof}

Let $\mc{P}_{[d]}(A)$ be the full subcategory of projective objects in $\mc{G}_{[d]}(A)$. Let $G$ be a right adjoint to $F_d$. The dual basis lemma implies that $\epsilon : F_d \circ G \rightarrow 1$ is an equivalence. Recall from section \ref{sec:quasihereditary} that we fixed the basic quasi-hereditary algebra $C_d$ such that $\mc{G}_{[d]}(A) \simeq \mc{M}(C_d)$. We can think of $C_d$ as a minimal projective generator in $\mc{G}_{[d]}(A)$. 

\begin{lem}\label{lem:Ffaithfulproj}
The functor $F_d$ is faithful on $\mc{P}_{[d]}(A)$ if and only if
\begin{equation}\label{eq:homsimplePzero}
\Hom_{\mc{G}_{[d]}(A)}(L(\lambda),Q) = 0
\end{equation}
for all objects $Q$ in $\mc{P}_{[d]}(A)$ and all $\lambda$ such that $F_d(L(\lambda)) = 0$. 
\end{lem}

\begin{proof}
Let $G$ be the right adjoint to $F \dopgleich F_d$. Then for each $Q \in \mc{P}_{[d]}(A)$, we have an exact sequence 
\begin{equation}\label{eq:fullyfaithfulQ}
0 \rightarrow \Ker \eta_Q \rightarrow Q \stackrel{\eta_Q}{\longrightarrow} G \circ F (Q) \rightarrow \Coker \eta_Q \rightarrow 0.
\end{equation}
Since $F \circ G \simeq 1$ and $F$ is exact, $F( \Ker \eta_Q) \simeq F(\Coker \eta_Q) \simeq 0$. Exactness of $F$ also implies that $F(M) = 0$ if and only if $F(L) = 0$ for every simple composition factor of $M$. In particular, $F(M) = 0$ implies that $F(\Soc M) = 0$. Thus, equation (\ref{eq:homsimplePzero}) implies that $\Ker \eta_Q = 0$, and hence $F$ is faithful. Conversely, if there exist $\lambda,\mu$ such that $\Hom_{\mc{G}_{[d]}(A)}(L(\lambda),Q_d(\mu)) \neq 0$ and $F(L(\lambda)) = 0$, then there is a non-zero morphism $\phi : Q_d(\lambda) \twoheadrightarrow L(\lambda) \hookrightarrow Q_d(\mu)$ such that $F(\phi) = 0$, i.e. $F$ is not faithful on $\mc{P}_{[d]}(A)$.   
\end{proof}

\begin{prop}\label{prop:covercriterion}
The following are equivalent:
\begin{enum_thm}
\item \label{prop:covercriterion:cover} The functor $F_d \from \mc{G}_{\lbrack d \rbrack}(A) \to \mc{M}(B_\ell)$ is a highest weight cover. 
\item We have  
$$
\Ext^i_{\mc{G}_{[d]}(A)}(L(\lambda),Q) = 0 \quad \forall \ i \in \{0,1 \} \;,
$$
where $\lambda$ runs over all simple objects such that $F_d(L(\lambda)) = 0$ and $Q$ an object of $\mc{P}_{[d]}(A)$.
\item \label{prop:covercriterion:dcp} The double centralizer property holds: $C_d \simeq \End_{B_{\ell}}(F_d(C_d))$. 
\end{enum_thm}
\end{prop}

\begin{proof}
It is well-known that \ref{prop:covercriterion:cover} is equivalent to \ref{prop:covercriterion:dcp}. Therefore, it suffices to show that $F \dopgleich F_d$ is fully faithful on $\mc{P}_{[d]}(A)$ if and only if
\begin{equation}\label{eq:homsimplePzero2}
\Ext^i_{\mc{G}_{[d]}(A)}(L(\lambda),Q_d(\mu)) = 0 \quad \forall \mu, \ i \in \{0,1 \} \;,
\end{equation}
and all $\lambda$ such that $F(L(\lambda)) = 0$. First we show that (\ref{eq:homsimplePzero}) implies that $F$ is fully faithful. By Lemma \ref{lem:Ffaithfulproj}, we know that $F$ is faithful. Hence it suffices to show that $\Coker \eta_Q = 0$ in the exact sequence (\ref{eq:fullyfaithfulQ}). Let $\Ker F$ be the full subcategory of $\mc{G}_{[d]}(A)$ consisting of objects killed by $F$. Then the long exact sequence of ext-groups and (\ref{eq:homsimplePzero2}) imply that the sequence $0 \rightarrow Q \rightarrow G \circ F (Q) \rightarrow \Coker \eta_Q \rightarrow 0$ splits. Thus, $G \circ F (Q) \simeq Q \oplus \Coker \eta_Q$. But adjunction implies that
$$
\Hom_{\mc{G}_{[d]}}(\Coker \eta_Q,G \circ F (Q)) \simeq \Hom_{\mc{M}(A_0)}(F(\Coker \eta_Q), F (Q)) = 0.
$$
Hence $\Coker \eta_Q = 0$. 

Conversely, assume that $F$ is fully faithful. In particular, by Lemma \ref{lem:Ffaithfulproj}, this implies that $\Hom_{\mc{G}_{[d]}(A)}(L(\lambda),Q_d(\mu)) = 0$ for all $\mu$ and all $\lambda$ such that $F(L(\lambda)) = 0$. Therefore, we just need to show that $\Ext^1_{\mc{G}_{[d]}(A)}(L(\lambda),Q_d(\mu)) = 0$. Let $0 \rightarrow Q_d(\mu) \rightarrow M \rightarrow L(\lambda) \rightarrow 0$ be a short exact sequence in $\mc{G}_{[d]}(A)$. Then we get a commutative diagram
$$
\begin{tikzcd}
0 \arrow{r} & Q(\mu) \arrow{r} \arrow{d}{\wr} & M \arrow{d} \arrow{r} & L(\lambda) \arrow{r} \arrow{d} & 0 \\   
0 \arrow{r} & G \circ F(Q(\mu)) \arrow{r} & G \circ F(M) \arrow{r} & 0 & 
\end{tikzcd}
$$
with exact rows. This implies that $G \circ F(M) \simeq Q_d(\mu)$ and we get a splitting of the embedding $Q_d(\mu) \hookrightarrow M$. Hence $\Ext^1_{\mc{G}_{[d]}(A)}(L(\lambda),Q_d(\mu)) = 0$.
\end{proof} 

In fact, a stronger statement holds for faithfulness. First we note the following key technical lemma. Recall that $A$ is said to be \textit{well-generated} if $A^{\pm}$ is generated by $A^{\pm}_{\pm 1}$. 

\begin{lem}\label{lem:soclAlhdsupport}
Assume that $\Soc A^{\pm} = A^{\pm}_{\pm N}$ and that $A$ is well-generated. Choose $k \in \N$, set $U = (A^{\pm})^{\oplus k}$ and choose $M \subset U$, a graded $A^{\pm}$-submodule.  
\begin{enum_thm}
\item If $i$ is the largest integer with $M_i \neq 0$, then $\Supp (\Soc_{A^-} U / M) \subset (-\infty,i+1]$.
\item If $j$ is the smallest integer with $M_j \neq 0$, then $\Supp (\Soc_{A^+} U / M) \subset [j-1,\infty)$.
\end{enum_thm}
\end{lem}

\begin{proof}
The proof of the two statements is similar, therefore we just show the first statement. The module $M$ is contained in $U_{\le i}$, and the latter is a $A^-$-submodule of $U$. Therefore the map $\eta : U / M  \rightarrow U / U_{\le i}$ is surjective and $\Soc  U / M$ is contained in $\eta^{-1}(\Soc U / U_{\le i})$. This implies that  $\Supp (\Soc U / M) \subset \Supp (\eta^{-1}(\Soc U / U_{\le i}))$. If we write $\pi : U \rightarrow U / U_{\le i}$, then $\pi$ factors through $\eta$, which means that 
$$
\Supp (\eta^{-1}(\Soc U / U_{\le i})) \subset \Supp (\pi^{-1}(\Soc U / U_{\le i})).
$$
Since $\pi$ is a graded morphism, and $U$ is a graded free $A^-$-module,
\begin{align*}
\Supp (\pi^{-1}(\Soc U / U_{\le i})) & \subset \Supp (\Soc U / U_{\le i}) \cup (-\infty,i] \\&= \Supp (\Soc A^- / A^-_{\le i}) \cup (-\infty,i] \;. 
\end{align*}
Hence, it suffices to show that $\Soc A^- / A^-_{\le i} = A^-_{\le i+1} / A^-_{\le i}$, and thus 
\[
\Supp \Soc A^- / A^-_{\le i} = \{i+1 \} \;.
\] 
The proof of this claim will be by induction on $i$. If $i = -N -1$, then the statement is equivalent to the fact that $\Soc A^- = A^-_{-N}$. Therefore, we may assume that it is true for $i - 1$. Let $a \in A_{i+2}^-$ be non-zero. We need to find some non-zero element $b \in \rad A^- = A^-_{\le -1}$ such that $b a \neq 0$ in $A^- / A^-_{\le i}$. Consider the space $A_{-1}^- \cdot a$. This is contained in $A^-_{i-1}$. Therefore $ A_{-1}^- \cdot a = 0$ in $A^- / A_{\ge i}^-$ if and only if it is zero in $A^-$. Since $A^{-}$ is generated by  $A^-_{-1}$, the radical of  $A^-$ is generated by $A^-_{-1}$ as an ideal. Therefore $A^-_{-1} \cdot a = 0$ implies that $(\rad A^-) \cdot a = 0$, which implies that $a \in \Soc A^- = A^-_{-N}$. But this contradicts the fact that $a \in A^-_{i-2}$. 
\end{proof}

\begin{prop}\label{prop:minusonefaithful}
Assume that $A$ is well-generated and $\Soc A^{\pm} = A^{\pm}_{\pm N}$. If $\Supp L(\lambda) \cap [0,\ell] = \emptyset$ then $\Hom_{\mc{G}_{[d]}(A)}(L(\lambda),\Delta(\mu)) = 0$ for all standard modules $\Delta(\mu)$ in $\mc{G}_{[d]}(A)$. 
\end{prop}

\begin{proof}
We begin by noting that 
$$
\Hom_{\mc{G}_{[d]}(A)}(L(\lambda),\Delta(\mu)) = \lim_{\stackrel{M \subset \Delta(\mu),}{M \in \mc{G}_{< 0}(A)}} \Hom_{\mc{G}_{\le d}(A)} (L(\lambda),\Delta(\mu) / M).
$$
If $M,M' \subset \Delta(\mu)$ are submodules that belong to $\mc{G}_{< 0}(A)$, then so too does their sum $M + M'$. Therefore there exists a unique largest submodule $M$ belonging to $\mc{G}_{< 0}$ and the space $\Hom_{\mc{G}_{[d]}(A)}(L(\lambda),\Delta(\mu)) $ equals $\Hom_{\mc{G}_{\le d}(A)} (L(\lambda),\Delta(\mu) / M)$. Since $[M : L(\rho)] \neq 0$ implies that $\deg(\rho) < 0$, $M_i = 0$ unless $i < 0$ and hence $(\Delta(\mu) / M)_i = \Delta(\mu)_i$ for all $i \ge 0$. The group $\Hom_{\mc{G}_{\le d}(A)} (L(\lambda),\Delta(\mu) / M)$ is non-zero if and only if $L(\lambda)$ appears in the socle of $\Delta(\mu) / M$. Thus, it suffices to show that $L(\lambda) \subset \Soc \Delta(\mu) / M$ implies that $\Supp L(\lambda) \cap [0,\ell] \neq \emptyset$. Notice first that if $\Supp L(\lambda) \subset \N_{< 0}$, then $L(\lambda)$ is an object of $\mc{G}_{< 0}$. This would contradict the maximality of $M$. Thus, $\Supp L(\lambda) \cap \N \neq \emptyset$, and it suffice to show that $\Supp L(\lambda) \cap (- \infty, \ell] \neq \emptyset$. Since 
$$
\Soc_{A^{-}} L(\lambda) \subset \Soc_{A^{-}} (\Soc_A \Delta(\mu) / M) \subset   \Soc_{A^{-}} \Delta(\mu) / M,
$$
it suffices to show that $\Supp (\Soc_{A^{-}} \Delta(\mu) / M) \subset (-\infty, \ell]$. There are two cases to consider here. First, assume that $M \neq 0$. Since $\Supp \Delta(\mu) \subset [\deg \mu - N , \deg \mu]$ and $\Supp M \subset \N_{< 0}$, this implies that $\deg \mu < N$ and the result follows from Lemma \ref{lem:soclAlhdsupport}. Secondly, if $M = 0$, then $\Supp (\Soc_{A^{-}} \Delta(\mu) / M) = \Supp (\Soc_{A^{-}} \Delta(\mu) ) = \{ \deg \mu - N \}$. Since $\deg \mu \le d$, $\deg \mu - N \le \ell$, as required. 
\end{proof} 

\begin{rem}
Notice that Proposition \ref{prop:minusonefaithful} implies that if $F_d(L(\lambda)) = 0$ then $\Hom_{\mc{G}_{[d]}(A)}(L(\lambda),M)$ is zero for all $M \in \mc{G}_{[d]}^{\Delta}(A)$. In particular, $F_d$ is a $(-1)$-faithful cover, in the sense of Rouquier \cite{RouquierQSchur}. 
\end{rem}

\begin{lem}\label{lem:prestageprop}
For each $0 \le \deg \mu \le d$, the projective indecomposable $Q_d(\mu)$ is a direct summand of ${}^{\perp}H_d(A[\deg \mu])$ in $\mc{G}_{[d]}(A)$, with complement $V$ belonging to $\mc{G}_{[d]}^{\Delta}(A)$.
\end{lem}

\begin{proof}
First, as noted in Lemma \ref{lem:obviousdecomp},  
$$
\dim \Hom_{\mc{G}(A)}(A[\deg \mu], L(\mu)) = \dim L(\mu)_{\deg \mu} = \dim \mu  \neq 0.
$$
Therefore, $P(\mu)$ is a non-zero summand of $A[\deg \mu]$. Hence $Q_d(\mu) = {}^{\perp}H_d(P(\mu))$ is a non-zero direct summand of ${}^{\perp}H_d(A[\deg \mu ])$. We can choose the complement $V'$ of $P(\mu)$ in $A[\deg \mu]$ to be projective. Then $V = {}^{\perp}H_d(V')$ is a complement to $Q_d(\mu)$ in ${}^{\perp}H_d(A[\deg \mu])$. Again, as noted in Lemma \ref{lem:obviousdecomp}, if $P(\lambda)$ occurs as a non-zero summand of $A[\deg \mu]$, then $\deg \mu \in [\deg \lambda - r_{\lambda}, \deg \lambda ]$. In particular, $\deg \lambda \in [\deg \mu, \deg \mu + N]$. Thus, it suffices to show that ${}^{\perp}H_d(P(\lambda))$ belongs to $\mc{G}_{[d]}^{\Delta}$ for any projective indecomposable in $\mc{G}(A)$ with $\deg \lambda \ge 0$ (recall that $\mu$ is assumed to satisfy $\deg \mu \ge 0$). If $\deg \lambda  > d$ then Lemma \ref{lem:projcoversub} says that ${}^{\perp}H_d(P(\lambda)) = 0$. Similarly, if $0 \le \deg \lambda \le d$, then by Lemma \ref{lem:projcoverquotient}, $Q_d(\lambda) = {}^{\perp}H_d(P(\lambda))$ is the projective cover of $L(\lambda)$, and hence admits a standard filtration.
\end{proof}

\begin{lem}\label{lem:observeminusi}
Assume that $A$ is well-generated. Then, for each $i \ge 1$, 
$$
A / A \cdot A_{i} = A \otimes_{B^{+}}  B^{+} /  B^{+}_{\ge i}
$$
as left $A$-modules. 
\end{lem}

\begin{proof}
The statement holds if and only if $A \cdot A_i = A \cdot B^{+}_{\ge i}$. This equality follows from the fact that $A^{+}$ is generated by $A^{+}_{1}$. 
\end{proof}

\begin{lem}
Assume that $A$ is well-generated and $\Soc A^{\pm} = A^{\pm}_{\pm N}$. For each $0 \le \deg \lambda \le d$, there is some $k > 0$ and a short exact sequence 
\begin{equation}\label{eq:Qembedses}
0 \rightarrow Q_d(\lambda) \rightarrow A\langle \ell\rangle^{\oplus k} \rightarrow M \rightarrow 0
\end{equation}
in $\mc{G}_{[d]}(A)$ such that $M \in \mc{G}_{[d]}^{\Delta}(A)$. 
\end{lem}

\begin{proof}
By Lemma \ref{lem:prestageprop}, it suffices to show that ${}^{\perp}H_d(A[\deg \lambda])$ embeds in $A[\ell]^{\oplus k}$ for some $k$ with quotient $M$ in $\mc{G}_{[d]}^{\Delta}(A)$. Let $i = \deg \lambda$. First, note that Lemma \ref{lem:observeminusi} implies 
$$
{}^{\perp}H_d(A[i]) = \left( A \otimes_{B^+} B^+ / B^+_{> d - i} \right)[i].
$$
Therefore it suffices to show that $(B^+ / B^+_{> d - i})[i]$ embeds in $(B^+\langle \ell\rangle)^{\oplus k}$, with quotient $M'$ say, because then the quotient $M$ will equal $A \otimes_{B^+} M'$, which belongs to $\mc{G}_{\le d}(A) \cap \mc{G}^{\Delta}(A) =\mc{G}_{\le d}^{\Delta}(A)$, and hence its image in $\mc{G}_{[d]}(A)$ will belong to $\mc{G}^{\Delta}_{[d]}(A)$. By Lemma \ref{lem:soclAlhdsupport}, the socle of $(B^+ / B^+_{> d - i})[i]$ is contained in $(B^+_{\ge d - i} / B^+_{> d - i})[i]$. In fact, it follows that the socle equals $(B^+_{\ge d - i} / B^+_{> d - i})[i]$. Notice that $\Supp (B^+_{\ge d - i} / B^+_{> d - i})[i] = \{ d \}$. This implies that there is a graded embedding 
$$
\Soc_{B^+}((B^+ / B^+_{> d - i})[i]) = (B^+_{\ge d - i} / B^+_{> d - i})[i]  \hookrightarrow (B^+\langle \ell\rangle)^{\oplus k}
$$
for $k \gg 0$. Since the graded shifts of $B^+$ give a set of injective co-generators in $\mc{G}(B^+)$, this extends to an embedding of $(B^+ / B^+_{> d - i})[i]$ into $(B^+\langle \ell\rangle)^{\oplus k}$, as required. 
\end{proof}

The following proposition completes the proof of Theorem \ref{thm:highestweightcover}. 

\begin{lem}
Suppose that $A$ is well-generated, and that $\Soc A^\pm = A^\pm_{\pm N}$. Then the functor $F_d \from \mc{G}_{\lbrack d \rbrack}(A) \to \mc{M}(B_\ell)$ is a highest weight cover.
\end{lem}

\begin{proof}
By Proposition \ref{prop:covercriterion}, it suffices to show that $\Ext^i_{\mc{G}_{[d]}(A)}(L(\lambda),Q_d(\mu)) = 0$ for all $0 \le \deg \lambda,\deg \mu \le d$ such that $\Supp L(\lambda) \cap [0,\ell] = \emptyset$, and all $i \in \{ 0, 1 \}$. We have shown in Proposition \ref{prop:minusonefaithful} that $\Hom_{\mc{G}_{[d]}(A)}(L(\lambda),M) = 0$ for all $M \in \mc{G}_{[d]}^{\Delta}(A)$. Fixing a short exact sequence as in (\ref{eq:Qembedses}), and applying $\Hom_{\mc{G}_{[d]}(A)}(L(\lambda), - )$, we get 
\begin{multline*}
0 \rightarrow \Ext^1_{\mc{G}_{[d]}(A)}(L(\lambda),Q_d(\mu)) \rightarrow   \Ext^1_{\mc{G}_{[d]}(A)}(L(\lambda),A\langle \ell\rangle^{\oplus k}) \\  \rightarrow  \Ext^1_{\mc{G}_{[d]}(A)}(L(\lambda),M) \rightarrow \cdots 
\end{multline*}
By Lemma \ref{A_0_ell_projective}, the middle Ext-group vanishes. Therefore, we deduce that the group $\Ext^1_{\mc{G}_{[d]}(A)}(L(\lambda),Q_d(\mu))$ equals zero too. 
\end{proof}

\begin{rem}
When $i < \ell$, $A[0,i]$ is still a well-defined projective-injective in $\mc{P}_{[d]}(A)$, with endomorphism ring $B_i$. However, the functor $M \mapsto \Hom_{\mc{G}_{[d]}(A)}(A\langle i \rangle,M)$ is not faithful on $\mc{P}_{[d]}(A)$ and hence cannot be a highest weight cover of $B_i$. To see this, take any $\lambda \in \Irr \mc{G}(T)$ with $\deg \lambda \le d$ and $r_{\lambda} = \ell$. By Lemma \ref{A_0_ell_projective}, $A\langle \ell\rangle$ is a projective object in $\mc{G}_{[d]}(A)$. Lemma \ref{lem:obviousdecomp}, together with the fact that $A$ is triangular self-injective, implies that  
$$
\Hom_{\mc{G}_{[d]}(A)}(L(\lambda),A[\ell]) = \Hom_{\mc{G}(A)}(L(\lambda)[-\ell],A) \neq 0,
$$
since $L(\lambda[-\ell])$ is in the socle of $A$. On the other hand, $\Supp L(\lambda) = [\ell, \deg \lambda]$, which implies that $\Hom_{\mc{G}_{[d]}(A)}(A\langle 
i\rangle,L(\lambda)) = 0$. Hence, Lemma \ref{lem:Ffaithfulproj} implies that this functor is not faithful on $\mc{P}_{[d]}(A)$.  
\end{rem}

\subsection{On the socle assumption} \label{sec_socle_assumption}
We show that the socle assumption $\Soc A^\pm = A^\pm_{\pm N}$ we used for the highest weight cover is satisfied if $A$ is graded symmetric and ambidextrous.

\begin{lem}\label{lem:Alhdsoc}
Suppose that $A$ is a positively graded $K$-algebra with $A_0 = K$. If $A$ is $d$-Frobenius, then $d=\max \Supp(A)$ and 
\begin{equation*}
A_{d} = \Soc ( _AA ) = \Soc ( _AA_A ) = \Soc ( A_A ) \;.
\end{equation*}
This space is one-dimensional, so a $d$-Frobenius form $\Phi$ on $A$ is up to scalar the projection onto~$A_d$. 
\end{lem}

\begin{proof}
Let $N \dopgleich \max \Supp(A)$. The fact that $A$ is Frobenius implies that $\Soc ( _AA) = \Soc ( A_A)$, see \cite[Theorem 58.12]{CR}. It clear that the Jacobson radical $\Rad(A)$ of $A$ is equal to the unique maximal graded ideal of $A$, which is $\bigoplus_{i \in \bbN_{>0}} A_i$. From this, we see that $A/\Rad(A) \simeq A_0 = K$. Again since $A$ is Frobenius, it follows from \cite[Theorem 16.14]{Lam-Lectures-Modules-Rings-99} that $\Soc ( _AA) \cong A/\Rad(A)$ as left $A$-modules. In particular, the socle of $A$ is one-dimensional. By degree reasons and the definition of $N$, the radical of $A$ acts trivially on $A_N$. Hence, $A_N$ is contained in the socle of $A$ by \cite[Lemma 58.3]{CR}. This immediately implies that $A_N$ is one-dimensional and equal to the socle. Furthermore, this shows that it is a minimal two-sided ideal of $A$, thus contained in $\Soc ( _AA_A)$. Clearly, $\Soc ( _AA_A) \subs \Soc ( A_A)$, and this shows that indeed $A_N = \Soc ( _AA_A) $. Finally, $A_N$ cannot be contained in the kernel of $\Phi$ by definition. It is then clear that $\Ker \Phi = \bigoplus_{i \neq N} A_i$ and that $\Phi$ is projection onto $A_N$. In particular, $d=N$.
\end{proof}  

Of course, there is a similar version of the lemma for $A$ negatively graded.  Now, let $N^\pm$ be the maximum of $\Supp A^+$, resp. the minimum of $\Supp A^-$. Recall that we defined $N = N^+$.

\begin{lem} \label{A_graded_frob_N}
If $A$ is graded Frobenius, then $N^- = - N^+$.
\end{lem}

\begin{proof}
Let $\Phi \from A \to K$ be the Frobenius form. By definition, $A_{N^-}^- \neq 0$, so there is $0 \neq a \in A_{N^-}^-$. Since $A$ is graded Frobenius and $Aa$ is a non-zero left ideal of $A$, we have $\Phi( (Aa)_0 ) \neq 0$, so $(Aa)_0 \neq 0$. But $(Aa)_0 = A_{-N^-} a$, so $A_{-N^-} \neq 0$, implying that $-N^- \leq N^+$. Similarly, we see that $-N^- \geq N^+$.  
\end{proof}

\begin{lem}
If $A$ is graded symmetric and ambidextrous then:
\begin{enum_thm}
\item $B^{\pm}$ is $\pm N$-Frobenius and $\Soc(B^\pm) = B^\pm_{\pm N}$.
\item $A^{\pm}$ is $\pm N$-Frobenius and $\Soc(A^\pm) = A^\pm_{\pm N}$.
\end{enum_thm} 
\end{lem}

\begin{proof}
By Lemma \ref{A_graded_frob_N} we have $N^\pm = \pm N$. Once we established the Frobenius property, the claim about the socle follows from Lemma \ref{lem:Alhdsoc}. Let $\Phi \from A \to K$ be a symmetrizing trace on $A$. Fix $a \in A^+_N$ non-zero and define $\eta : B^- \rightarrow K$ by $\eta(b) = \Phi (b a)$. Notice that $\eta(b) = 0$ for all homogeneous $b$ with $\deg b \neq -N$. Let $I \subset B^-$ be a left ideal. We must show that $\eta(I) \neq 0$. By \cite[Lemma 5.4]{hwtpaper}, we may assume that $I$ is graded. The triangular decomposition implies that $I a \neq 0$. Therefore $A I a$ is a non-zero graded left ideal of $A$ and hence $\Phi(A I a ) \neq 0$. We claim that $\eta(I) = \Phi(A I a)$. It suffices to show that if $c \in A$ and $b \in I$ are homogeneous, then there exists some element $b' \in I$ such that $\Phi(c b a) = \Phi(b' a)$. First, we note that $\Phi(c b a) = \Phi(b a c)$. Now, since $A$ is assumed to be ambidextrous, $c = \sum_{i \in \Z} c_i^+ \otimes t_i \otimes c_i^-$, where $c_i^+ \in A^+_i$, $t_i \in T$ and $c_i^- \in A_{d-i}^-$, where $\deg c = d$. For degree reasons, $a \in \Soc A^+$, which implies that $a c = c_0^+ \otimes t_0 \otimes c_0^-$, where $c_0^+ = 1$. Thus, $\Phi(c b a ) = \Phi (b a t_0 c_0^-)$. Taking $b' = t_0 c_0^- b$, we have $\Phi(c b a ) = \Phi (b a t_0 c_0^-) = \Phi(b' a)$. Similarly we see that $B^+$ is Frobenius. It now suffices to show that if $B^+$ is $N$-Frobenius then $A^+$ is also $N$-Frobenius. The fact that $B^+ / B^+_{> 0} \simeq T$ implies that $\Soc B^+ = B^+_N$ is a free left $T$-module of rank one. Therefore, the triangular decomposition implies that $\Soc A^+ = A_N^+$ is one-dimensional. Since $A^+$ is a positively graded, local ring, this implies that it is $N$-Frobenius.
\end{proof}

\subsection{Basic sets}

We note that the notion of canonical basic sets makes sense in our setting. These where first introduced in \cite{GeckRouquierBasic}, and a more general definition of basic sets was given in \cite{GordonGriffChlo}. Recall that $\Irr \mc{G}_{[d]}(T) =  \Irr \mc{M}(T) \times [0,d]$ is the poset parameterizing the simple objects in the highest weight category $\mc{G}_{[d]}(A)$. Let $B \subset \Irr \mc{G}_{[d]}(T)$ to be the set of all $\lambda$ such that $D_{\ell} (\lambda) := F(L(\lambda)) \neq 0$. Set $S_{\ell}(\lambda) := F(\Delta(\lambda))$, for all $\lambda \in \Irr \mc{G}_{[d]}(T)$. For each $\lambda$ in $B$, define  
$$
\deg' \lambda = \min \{ \deg \mu \ |\ \mu \in \Irr \mc{G}_{[d]}(T), \ [ S_{\ell}(\lambda) : D_{\ell}(\mu)] \neq 0 \}.
$$  

\begin{prop} \hfill

\begin{enum_thm}
\item The set $\{ D_{\ell}(\lambda) \ | \ \lambda \in B \}$ is a complete set of pairwise non-isomorphic simple objects in $\mc{M}(B_{\ell})$. 
\item For all $\lambda \in B$, $\deg \lambda = \deg' \lambda$ and $[S_{\ell}(\lambda) : D_{\ell}(\lambda)] = 1$. 
\item If $[S_{\ell}(\lambda) : D_{\ell}(\mu)] \neq 0$ and $\mu \neq \lambda$ then $\mu < \lambda$. 
\end{enum_thm}
\end{prop}

\begin{proof}
The first fact is a consequence of the fact that Theorem \ref{thm:highestweightcover} implies that there is an equivalence of categories $\mc{G}_{[d]} (A) / \Ker F \stackrel{\sim}{\longrightarrow} \mc{M}(B_{\ell})$. Then statements (b) and (c) both follow from the fact that 
$$
D_{\ell}(\lambda) = F(L(\lambda)) = \bigoplus_{i = 0}^{\ell} L(\lambda)_i, \quad S_{\ell}(\mu) = F(\Delta(\mu)) = \bigoplus_{i = 0}^{\ell} \Delta(\mu)_i,
$$
and the fact that $[\Delta(\lambda) : L(\mu)] \neq 0$ in $\mc{G}_{[d]}(A)$ implies that $\mu < \lambda$.  
\end{proof}

\subsection{Examples} \label{cover_examples}

\begin{ex}\label{ex:Ndual}
Fix an integer $n \ge 1$ and let $A = K[x,y] / (x^{n},y^{n})$. This is a $\Z$-graded Gorenstein algebra, where we put $\deg (x) = -1$ and $\deg (y) = 1$. The projective indecomposables in $\mc{G}(A)$ are $P(i) = A[i]$. Let $E$ be the non-unital algebra
$$
E = \bigoplus_{i,j \in \Z} \Hom_{\mc{G}(A)}(P(i),P(j)).
$$
Then $\mc{G}(A)$ is equivalent to the category of finite dimensional $E$-modules. Let $Q_{\infty}$ be the quiver with vertices $\{ e_i \ |  \ i \in \Z \}$ and arrows $x_i : e_i \rightarrow e_{i-1}$ and $y_i : e_i \rightarrow e_{i+1}$, then $x$ corresponds to the lowering arrows $x_i$ and $y$ corresponds to the raising arrows $y_i$. Let $I$ be the admissible ideal of $k Q_{\infty}$ generated by $x_{i+1} \circ y_i = y_{i-1} \circ x_i$, $y_{i+n-1} \circ \cdots \circ y_i = 0$ and $x_{i-n+1} \circ \cdots \circ x_i = 0$  for all $i$, then $E \simeq k Q_{\infty} / I$ and hence $\mc{G}(A) \simeq \mc{M}(K Q_{\infty} / I)$. 
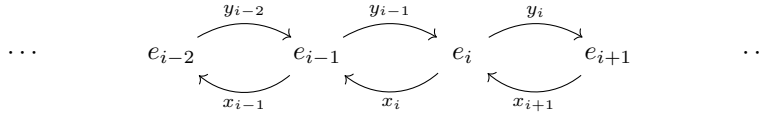
$$
\begin{tikzcd}
\cdots \ \ & \makebox[\widthof{XXX}]{$e_{i-2}$} \arrow[bend left=30]{r}{y_{i-2}}  & \makebox[\widthof{XXX}]{$e_{i-1}$} \arrow[bend left=40]{l}{x_{i-1}} \arrow[bend left=30]{r}{y_{i-1}} & \makebox[\widthof{XXX}]{$e_{i}$} \arrow[bend left=40]{l}{x_{i}} \arrow[bend left=30]{r}{y_{i}} & \makebox[\widthof{XXX}]{$e_{i+1}$} \arrow[bend left=40]{l}{x_{i+1}} & \ \ \cdots
\end{tikzcd}
$$
Choose an integer $d \ge 0$. For $i \ge - d$, let $u(i,d) = \min \{ i + d + 1, n \}$ and let 
$$
Q(i) = K[x,y] / (x^{n}, y^{u(i,d)} ) .
$$
Then $Q(i)$ is the projective cover of $L(i)$ in $\mc{G}_{\le d}(A)$ and the latter category is equivalent to finite dimensional representations of 
$$
E_d = \bigoplus_{i,j \ge - d} \Hom_{\mc{G}(A)}(Q(i),Q(j)). 
$$
This time, we take $Q_d$ to be the quiver with vertices $\{ e_i \ |  \ i \ge -d \}$ and arrows $x_i : e_i \rightarrow e_{i-1}$ and $y_i : e_i \rightarrow e_{i+1}$. Again $x$ corresponds to the lowering arrows $x_i$ and $y$ corresponds to the raising arrows $y_i$. Let $I_d$ be the admissible ideal of $k Q_d$ generated by 
\begin{align*}
0 & = y_{j+n-1} \circ \cdots \circ y_j & \forall \ j \ge -d \\
0 &  = x_{j-n+1} \circ \cdots \circ x_j  & \forall \ j > n - d \\
x_{i+1} \circ y_i & = y_{i-1} \circ x_i, & \forall \ i > -d \\
0 & = y_{-d} \circ x_{-d + 1} & 
\end{align*}
Then a direct check shows that $E \simeq K Q_{d} / I_d$ and hence $\mc{G}_{\le d}(A) \simeq \mc{M}(K Q_{d} / I_d)$.

Finally, let $Q_{[d]}$ be the quiver with $d+1$ vertices $\{e_{-d}, \dots, e_0 \}$ and arrows $x_i : e_i \rightarrow e_{i-1}$ for $-d < i \le 0$ and $y_j : e_j \rightarrow e_{j+1}$ for $-d \le j < 0$. We impose the admissible relations 
\begin{align*}
x_{i+1} \circ y_i & = y_{i-1} \circ x_i, & \textrm{for all} \ -d < i < 0 \\
0 & = y_{-d} \circ x_{-d + 1} & 
\end{align*}
together with 
$$
y_{i+n-1} \circ \cdots \circ y_i = x_{j-n+1} \circ \cdots \circ x_j =0
$$
for all $-d \le i \le -n$ and $n-d + 1 \le j \le 0$, if $d \ge n$. Then the quasi-hereditary algebra $C_{d}$ is isomorphic to $K Q_{[d]} / I_{[d]}$. 

\begin{rem}
	Let $K$ be an arbitrary algebraically closed field and $\mc{S}_d$ denote the $q$-Schur $K$-algebra associated to the cyclic group $\Z / d \Z$ at the special parameter $(a_1, \ds, a_{d}) = (0,\ds, 0)$; see \cite[\S 8.6]{McGertyKZ}. Then 
	$$
	\mc{S}_d = \End_{K[T] / (T-1)^{d}} \left( \bigoplus_{i = 1}^{d} K[T] / (T-1)^{i} \right). 
	$$
	It follows from the above computation that the quasi-hereditary algebra $C_{d}$ from example \ref{ex:Ndual} is isomorphic to $\mc{S}_d$.
\end{rem}
\end{ex} 

\begin{ex}\label{ex:x2}
The special case of example \ref{ex:Ndual}, where $n = 2$, occurs frequently. Here we describe some of those occurrences. Thus, take $K = \C$ and let $A = \C[x,y] / (x^2,y^2)$.  
\begin{itemize}
\item[(a)] Let $W = \s_3$ be the symmetric group on three letters and $\mbf{c} \neq 0$. We consider the block $B$ of the restricted rational Cherednik algebra $\overline{\H}_{\mbf{c}}(W)$ labelled by the reflection representation for $\s_3$. Then it is easy to see that $B$ is graded Morita equivalent to $A$. 
\item[(b)] The article  \cite{K-Dia-alg-IV} introduces a certain algebra $K_r^{\infty}$, and it is shown that the category of finite-dimensional $K_r^{\infty}$-modules is equivalent to  the principal block of $\mf{gl}(r|r)$, the super general linear group of type $(r|r)$, of atypicality $1$. The explicit presentation of $\mc{G}(A)$ given in example \ref{ex:Ndual}, together with the calculation on \cite[page 7]{K-Dia-alg-IV}, shows that there is an equivalence of highest weight categories
$$
\mc{G}(A) \simeq \mc{M}(K_1^{\infty}). 
$$
\item[(d)] The category $\mc{G}(A)$ also appears as example (vi) in \cite[page 291]{QuasiHerediatryDlabRingel}.  
\item[(c)] The subquotients $\mc{G}_{[d]}(A)$ are equivalent to the highest weight categories of perverse sheaves, with coefficients in $\C$, on $\mathbb{P}^d$, locally constant with respect to the standard stratification
\end{itemize}
\end{ex}

\begin{ex}\label{ex:sl2example}
	Let $K$ be a field of characteristic $p > 2$ and consider the restricted enveloping algebra $A = \overline{U}(\mf{sl}_2)$ of $\mf{sl}_2$, i.e. $A$ is the quotient of $U(\mf{sl}_2)$ by the relations $E^p = F^p = 0$ and $H^p = H$. Identifying $K[H] / (H^p - H)$ with $K^{\oplus p}$, $A$ has a triangular decomposition $A = A^{-} \otimes T \otimes A^{+}$ with $A^{-} = K[E] / (E^p)$, $T = k \mathbb{F}_p$ and $A^{+} = K[F] / (F^p)$, where we assign $\deg(E) = 1$, $\deg(H) = 0$ and $\deg(F) = -1$. Let $\Omega = 4 FE + (H+1)^2$ be the Casimir. Then the computation in \cite[Remark 3]{TangeCentre} shows that 
	$$
	\overline{U}(\mf{sl}_2)_0  = \frac{T [\Omega]}{\langle \Omega(\Omega^{(p-1) / 2) } - 1) \rangle} \simeq  K^{\oplus p} \oplus (K[t]/ (t^2))^{\oplus \frac{p(p-1)}{2}}.
	$$ 
	The ungraded simple modules for $\overline{U}(\mf{sl}_2)$ are $\{ L(i) \ | \ i \in \mathbb{F}_p \}$, where $L(i)$ has highest weight $i$. If we consider the highest weight cover $\mc{G}(\overline{U}(\mf{sl}_2))_{[p]}$ of $U_0(\mf{sl}_2)_0$ then a direct calculation shows that $\mc{G}(\overline{U}(\mf{sl}_2))_{[p]}$ has $p$ blocks equivalent to $\mc{M}(K)$ (whose simple modules are the $p$ graded lifts of $L(1)$) and $\frac{p(p-1)}{2}$ blocks (whose simple modules, after forgetting the grading, are $L(i)$ and $L(2-i)$ for $i \neq 1$) that are all equivalent to finite dimensional modules over the algebra $K Q / I$, where $Q$ is the quiver 
	$$
	\begin{tikzcd}
	& e_1 \arrow[bend left]{r}{\alpha}  & e_2 \arrow[bend left]{l}{\beta}
	\end{tikzcd}
	$$
	and $I$ is generated by the relation $\beta \alpha = 0$. Notice that the same quiver with relations over $\C$ is Morita equivalent to the principal block of category $\mc{O}$ for $\mf{sl}_2$. 
\end{ex}

\subsection{$0$-faithfulness} \label{0_faithful}
Ideally, one would like to show that the functor $F : \mc{G}_{[d]}(A) \rightarrow \mc{M}(B_{\ell})$ is a $0$-faithful cover. Proposition \ref{prop:minusonefaithful} already shows that $F$ is a $(-1)$-faithful cover. However, the following example shows that $F$ is not a $0$-faithful cover in general. 

\begin{ex}
We take $A = K[x,y] / (x^2,y^2)$, as in \cite[Example 3.2]{hwtpaper}. Then $N = 1$ and we consider the case $d = N$. When $K$ has characteristic zero, $\mc{G}_{[1]}(A)$ is equivalent to the principal block of category $\mc{O}$ for $\mf{sl}_2$. As a quiver with relations, the algebra $C_1$ is constructed as follows. Take $Q$ to be the quiver with vertices $e_{0}$ and $e_{1}$, and arrows $a : e_0 \rightarrow e_{1}$, $b : e_{1} \rightarrow e_0$ and fundamental relation $ba = 0$, and hence $\dim C_1 = 5$. We have $\End_{C_{1}}(Q_1(1)) = k[\epsilon] / (\epsilon^2) = A_0^{\op}$, and $F = \Hom_{C_{1}}(Q_1(1) , - ) : \mc{M}(C_1) \rightarrow \mc{M}(k[\epsilon] / (\epsilon^2))$. We note that 
$$
Q_1(0) = k \{ e_0, a \}, \quad Q_1(1) = k \{ e_{1}, b,ab \}, \quad \Delta(0) = Q_1(0), \quad \Delta(1) = k \{ e_{1} \} = L(1). 
$$
We have minimal resolutions
$$
0 \rightarrow Q_1(0) \rightarrow Q_1(1) \rightarrow Q_1(0) \rightarrow L(0) \rightarrow 0,
$$
$$
0 \rightarrow Q_1(0) \rightarrow Q_1(1) \rightarrow L(1) \rightarrow 0.
$$
This implies that $\Ext_{C_1}^i(L(j),\Delta(k)) = 0$ for all $i \in \{ 0, 1,2 \}$ and $j,k \in \{ 0, 1\}$, except for the groups  
\begin{multline*}
 \Hom_{C_{1}}(L(1),\Delta(0)) = \Hom_{C_{1}}(L(1),\Delta(1)) = \Ext^1_{C_{1}}(L(0),\Delta(1))   \\
 = \Ext^1_{C_{1}}(L(1),\Delta(0))  = \Ext^2_{C_{1}}(L(0),\Delta(0))  = K. 
\end{multline*}
Since $F(L(0)) = 0$ and $F(L(1)) \neq 0$, it follows that $F$ is only a $(-1)$-faithful cover. However, the groups $\Ext_{C_1}^i(L(j),\Delta(k))$ are zero, for all $i \in \{ 0, 1,2 \}$ and $j,k \in \{ 0, 1\}$, except for the groups 
$$
\Hom_{C_{1}}(L(-1),Q_1(1))  = \Ext^2_{C_{1}}(L(0),Q_1(0)) = K. 
$$
This implies that $F$ is a highest weight cover. It is also easy to check directly that 
$$
\Hom_{C_{1}}(Q_1(i),Q_1(j)) \rightarrow \Hom_{K[\epsilon] /(\epsilon^2)}(F(Q_1(i)),F(Q_1(j))) 
$$
is an isomorphisms for all $i,j \in \{ 0, 1 \}$.  

More generally, if we consider $C_{d}$, with $d \ge 1$ and set $\ell = d - 1$, then the algebra $C_d$ was described in the introduction, and 
$$
B_{\ell}^{\op} = \left( \begin{array}{ccccc}
\frac{K[\epsilon]}{(\epsilon^2)} & K & 0 & \cdots & 0 \\
K & \frac{K[\epsilon]}{(\epsilon^2)} & \ddots & & \vdots \\
0 & \ddots & \ddots  & \ddots  & 0  \\
\vdots & & \ddots & \ddots & K \\
0 & \cdots & 0 & K & \frac{K[\epsilon]}{(\epsilon^2)} 
\end{array} \right). 
$$
 \end{ex}

\begin{ex}
Similarly, one can take $A = K[x,y] / (x^3,y^3)$, so that $N = 2$. Again, if we take $d = N$, then 
\[
\begin{array}{c}
Q_2(0) =  K[x,y] / (x^3,y^3) \;, \quad Q_2(1)  =  K[x,y] / (x^3,y^2)[1] \;,\vspace{6pt}\\
Q_2(2)  = K[x,y] / (x^3,y)[2] 
\end{array}
\]
in $\mc{G}_{[2]}(A)$. For each Hom-space $\Hom_{\mc{G}_{[2]}(A)}(M,N)$ a basis is given in the table below:
$$
\begin{array}{c|ccc}
M \setminus N & Q_2(2)          & Q_2(1)         & Q_2(0) \\
\hline
Q_2(2) & \{ e_{2} \} & \{ y_{2} \} & \{ y_{1} y_{2} \} \\
Q_2(1) & \{ x_{1} \} & \{ e_{1}, y_{2} x_{1} \} & \{ y_{1}, y_{1} y_{2} x_{1} \} \\
Q_2(0) & \{ x_{1} x_0 \} & \{ x_0, y_{2} x_{1} x_0 \} & \{ e_0, y_{1} x_{0}, y_{1} y_{2} x_{1} x_{0} \}
\end{array}
$$
A basis of each projective is given as follows 
$$
Q_2(0) = K \{ e_0, y_{1} x_0 , y_{1} y_{2} x_{1} x_0 , x_0, y_{2} x_{1} x_0, x_{1} x_0 \},
$$
$$
Q_2(1) = K \{ y_{1}, y_{1} y_{2} x_{1}, e_{1}, y_{2} x_{1} , x_1 \}, \quad Q(2) = K \{ e_{2}, y_{2} , y_{1} y_{2} \}.
$$
Therefore we get partial projective resolutions
$$
0 \rightarrow Q_2(1) \stackrel{x_0}{\longrightarrow} Q_2(0) \rightarrow L(0) \rightarrow 0,
$$
$$
0 \rightarrow Q_2(2) \stackrel{x_{1}}{\longrightarrow} Q_2(1) \stackrel{y_{2}}{\longrightarrow} Q_2(2) \rightarrow L(2) \rightarrow 0,
$$
$$
\cdots \rightarrow Q_2(2) \oplus Q_2(1) \stackrel{D}{\longrightarrow} Q_2(2) \oplus Q_2(0) \stackrel{(x_{1},y_{1})}{\longrightarrow} Q_2(1) \rightarrow L(1) \rightarrow 0,
$$
where
$$
D = \left( \begin{array}{cc}
0 & x_{1} x_0 \\
y_{2} & - x_{0} 
\end{array} \right). 
$$
From this information, it is easy to deduce that the algebra $C_2$ is the quotient of the quiver $Q$, with vertices $e_{0},e_{1}$ and $e_{2}$ and arrows $y_{2} : e_{2} \rightarrow e_{1}$, $y_{1} : e_{1} \rightarrow e_0$, $x_0 : e_0 \rightarrow e_{1}$ and $x_{1} : e_{1} \rightarrow e_{2}$, by the relations
$$
x_{1} y_{2} = 0, \quad y_{2} x_{1} = x_0 y_{1}.
$$
Our convention is that $y_{1} y_{2}$ is the non-zero path from $e_{2}$ to $e_{0}$. Thus, $\dim C_2 = 14$. As in the previous example, one can check using the above partial resolutions that 
$$
\Ext^i_{C_2}(L(j),Q_2(k)) = 0
$$
for $i \in \{ 0, 1 \}$, $j \in \{ 1, 2 \}$ and $k \in \{ 0, 1, 2 \}$.  
\end{ex}

\appendix

\section{Cellularity via tilting objects} \label{cellular_appendix}

Let $\mathbf{U}_q$ be the (Lusztig type) quantum enveloping algebra of a complex semisimple Lie algebra with $q$ a root of unity\footnote{For technical reasons, it is assumed that $q$ is of odd order which is additionally assumed to be prime to $3$ for type $G_2$.} in a field $\bbK$ of arbitrary characteristic. Andersen, Stroppel, and Tubbenhauer have shown in \cite{AST} that if $T$ is a duality-stable tilting object in the category of finite-dimensional modules of type $1$ for $\mathbf{U}_q$, then the endomorphism algebra $\End_{\mathbf{U}_q}(T)$ is cellular in the sense of Graham and Lehrer \cite{Graham-Lehrer-Cellular}. It is clear from the proofs in \cite{AST} and also from the comments \cite[Remark 1.1 and 5.1.7]{AST} that this result actually holds for certain other examples of highest weight categories. There is, however, a subtle limitation, as pointed out in \textit{loc. cit.}, namely the reliance on the notion of \textit{weight spaces}. This is in general—and in particular in our context—not available in an arbitrary highest weight category. We will present here a generalization of the result in \cite{AST} not relying on the notion of weight spaces. The key ingredient is a modification of a specific construction in \cite{AST}, see (\ref{cellular_appendix_key_modification}). We also introduce an appropriate axiomatization of indecomposable tilting objects that works in more general highest weight categories, see Assumption \ref{tilting_indec_assumption}. We drop the assumption on the existence of a duality on the category. In this more general setting, we show that one still gets a \word{standard datum}, in the sense of Du–Rui \cite{DuRui} (see Definition \ref{standard_datum_def}), on the endomorphism algebra. \\

We will specify our setup by the three assumptions \ref{tilting_hwc_assumption}, \ref{tilting_ext_vanish}, and \ref{tilting_indec_assumption} below. In Remark \ref{finite_hwc_has_everything} we argue that all these assumptions hold in the \textit{split quasi-hereditary} case, i.e., in the case where we have \textit{finitely} many simple objects and these are all \textit{absolutely simple}. Note that our object of interest, the graded module category $\mathcal{G}(A)$ of an algebra with a triangular decomposition, does not belong to this class but still satisfies all these assumptions if $A$ is self-injective.

\begin{ass}[Highest weight category] \label{tilting_hwc_assumption}
We assume that $K$ is a field and that $\mathcal{C}$ is a $K$-linear abelian finite-length category whose (possibly \textit{infinite}) set of isomorphism classes of simple objects is indexed by $\Lambda$, and which is with respect to an interval-finite partial order $\leq$ on $\Lambda$ both a highest weight category with standard objects $\Delta(\lambda)$ and a highest weight category with costandard objects $\nabla(\lambda)$. 
\end{ass}

We note that $\mathcal{C}$ is $\Hom$-finite and $\Ext^1$-finite by \cite[Lemma 3.2(c)]{CPS}. We denote by $L(\lambda)$ the (simple) head of $\Delta(\lambda)$, which is also isomorphic to the socle of $\nabla(\lambda)$.

\begin{ass}[Ext-vanishing] \label{tilting_ext_assumption}
We assume that for any $\lambda,\mu \in \Lambda$ and any $0 \leq i \leq 2$ we have
\[
\Ext_\mathcal{C}^i(\Delta(\lambda), \nabla(\mu)) = \left\lbrace \begin{array}{ll} K & \tn{if } i=0 \tn{ and } \lambda = \mu \;, \\ 0 & \tn{else}  \;.\end{array} \right.
\]
\end{ass}

We denote by $\mathcal{C}^\Delta$, respectively $\mathcal{C}^\nabla$, the full subcategories of objects admitting a standard, respectively a costandard, filtration, defined exactly as in \cite[\S4.2]{hwtpaper} in our special setting. As in \cite[\S4.2]{hwtpaper} the Ext-vanishing assumption implies that for $M \in \mathcal{C}^\Delta$ and $N \in \mathcal{C}^\nabla$ we have
\begin{equation} \label{hwc_standard_costandard_mult}
\lbrack M:\Delta(\lambda) \rbrack = \dim_K \Hom_{\mathcal{C}}(M,\nabla(\lambda))   \quad \tn{and} \quad \lbrack N:\nabla(\lambda) \rbrack = \dim_K \Hom_{\mathcal{C}}(\Delta(\lambda),N)
\end{equation}
for the multiplicities of standard, respectively costandard, objects in $M$, respectively in $N$. We denote by $\mathcal{C}^t \dopgleich \mathcal{C}^\Delta \cap \mathcal{C}^\nabla$ the full subcategory of \word{tilting objects}, i.e., of objects admitting both a standard and a costandard filtration.

\begin{lem} \label{tilting_ext_vanish}
The following holds for $M \in \mathcal{C}$:
\begin{enum_thm}
\item $M \in \mathcal{C}^\Delta$ if and only if $\Ext_{\mathcal{C}}^1(M,\nabla(\lambda)) = 0$ for all $\lambda \in \Lambda$.
\item $M \in \mathcal{C}^\nabla$ if and only if $\Ext_{\mathcal{C}}^1(\Delta(\lambda),M) = 0$ for all $\lambda \in \Lambda$.
\end{enum_thm}
\end{lem}

\begin{proof}
We just show the first assertion of the lemma, the second is proven similarly. 
Assumption \ref{tilting_ext_assumption} immediately implies that if $M \in \mathcal{C}^\Delta$, then $\Ext_{\mathcal{C}}^1(M,\nabla(\lambda)) = 0$ for all $\lambda \in \Lambda$. The converse can be proven by exactly the same arguments as in \cite[Proposition 3.5]{AST-Notes} but we will include the details here to show where Assumption \ref{tilting_ext_assumption} is used. Let $\mathcal{X}$ be the class of all objects $M \in \mathcal{C}$ satisfying $\Ext_{\mathcal{C}}^1(M,\nabla(\lambda)) = 0$ for all $\lambda \in \Lambda$. Let $M \in \mathcal{X}$. Since $M$ has only finitely many composition factors, we can find $\lambda \in \Lambda$ minimal with the property that $\Hom_{\mathcal{C}}(M,L(\lambda)) \neq 0$. Let $\pi:\Delta(\lambda) \twoheadrightarrow L(\lambda)$ be the projection. Note that due to the highest weight category structure we have $\mu < \lambda$ for all composition factors $L(\mu)$ of $\Ker \pi$. We claim that $\Ext_{\mathcal{C}}^1(M,\Ker \pi) = 0$. Suppose that $\Ext_{\mathcal{C}}^1(M,\Ker \pi) \neq 0$. Then there must be a composition factor $L(\mu)$ of $\Ker \pi$ such that $\Ext_{\mathcal{C}}^1(M,L(\mu)) \neq 0$. From the exact sequence $0 \rarr L(\mu) \rarr \nabla(\mu) \rarr \nabla(\mu)/L(\mu) \rarr 0$ we obtain the exact sequence
\[
\Hom_{\mathcal{C}}(M,\nabla(\mu)/L(\mu)) \rarr \Ext_{\mathcal{C}}^1(M,L(\mu)) \rarr \Ext_{\mathcal{C}}^1(M,\nabla(\mu)) = 0 \;.
\]
Since $\Ext_{\mathcal{C}}^1(M,L(\mu)) \neq 0$, we must have $\Hom_{\mathcal{C}}(M,\nabla(\mu)/L(\mu)) \neq 0$, so there is a non-zero morphism $\varphi:M \rarr \nabla(\mu)/L(\mu)$. We can now find a constituent $L(\nu)$ of $\Im \varphi$ and a non-zero morphism $\Im \varphi \rarr L(\nu)$. Again by the highest weight category structure we have $\nu < \mu$. By composition we obtain a non-zero morphism $M \rarr L(\nu)$. Since $\nu < \mu < \lambda$, this contradicts the minimality of $\lambda$. Hence, we have shown that $\Ext_{\mathcal{C}}^1(M,\Ker \pi) = 0$. If we now choose a non-zero $\varphi \in \Hom_{\mathcal{C}}(M,L(\lambda))$, then there is $\hat{\varphi} \in \Hom_{\mathcal{C}}(M,\Delta(\lambda))$ with $\pi \circ \hat{\varphi} = \varphi$. Since $\pi$ is essential and $\varphi$ is surjective, also $\hat{\varphi}$ is surjective. From the exact sequence $0 \rarr \Ker \hat{\varphi} \rarr M \rarr \Delta(\lambda) \rarr 0$ we obtain
\[
0 = \Ext_{\mathcal{C}}^1(M,\nabla(\mu)) \leftarrow \Ext_{\mathcal{C}}^1(\Ker \hat{\varphi}, \nabla(\mu)) \leftarrow \Ext_{\mathcal{C}}^2(\Delta(\lambda),\nabla(\mu)) = 0 \;,
\]
so $\Ker \hat{\varphi} \in \mathcal{X}$. We can now argue by induction on the length of $M \in \mathcal{X}$ that all objects in $\mathcal{X}$ admit a standard filtration. Namely, let $M \in \mathcal{X}$ be of minimal length. Then, since $\Ker \hat{\varphi} \in \mathcal{X}$, we must have $\Ker \hat{\varphi} = 0$, so $M \simeq \Delta(\lambda)$, showing that $M$ admits a standard filtration. If $M$ has arbitrary length, we know that $M/\Ker \hat{\varphi} \simeq \Delta(\lambda)$ and by induction that $\Ker \hat{\varphi}$ admits a standard filtration, so $M$ also admits a standard filtration.
\end{proof}

\begin{lem} \label{tilting_krull_schmidt}
The following holds:
\begin{enum_thm}
\item \label{tilting_krull_schmidt:ext} If $T \in \mathcal{C}$, then $T \in \mathcal{C}^t$ if and only if $\Ext_{\mathcal{C}}^1(T,\nabla(\lambda)) = 0 = \Ext_{\mathcal{C}}^1(\Delta(\lambda),T)$ for all $\lambda \in \Lambda$.
\item \label{tilting_krull_schmidt:closed} $\mathcal{C}^t$ is closed under direct sums and under direct summands in $\mathcal{C}$.
\item $\mathcal{C}^t$ is a Krull–Schmidt category.
\end{enum_thm}
\end{lem}

\begin{proof}
The first assertion is just Lemma \ref{tilting_ext_vanish}. Part \ref{tilting_krull_schmidt:ext} clearly implies part \ref{tilting_krull_schmidt:closed} since $\Ext_\mathcal{C}^1$ is a bi-additive functor. Since $\mathcal{C}$ is abelian of finite length, it is Krull–Schmidt by \cite[Theorem 5.5]{Krause-KS}. Hence, $\mathcal{C}^t$ is also Krull–Schmidt by part \ref{tilting_krull_schmidt:closed}. 
\end{proof}

Since $\mathcal{C}^t$ is Krull–Schmidt, every tilting object is a finite direct sum  of indecomposable tilting objects, the decomposition being unique up to permutation and isomorphism of the summands. We will now assume an additional property about these indecomposable tilting objects. For this we use the notion of a \word{highest weight} as introduced in \cite[Definition 4.6]{hwtpaper} (the definition clearly works for arbitrary highest weight categories). 

\begin{ass}[Tilting objects] \label{tilting_indec_assumption}
We assume:
\begin{enum_thm}
\item For any $\lambda \in \Lambda$ there is an indecomposable tilting object $T(\lambda) \in \mathcal{C}^t$ such that:
\begin{enumerate}
\item $T(\lambda)$ has highest weight $\lambda$,
\item there is a monomorphism $\Delta(\lambda) \hookrightarrow T(\lambda)$,
\item there is an epimorphism $T(\lambda) \twoheadrightarrow \nabla(\lambda)$.
\end{enumerate}
\item The map $\lambda \mapsto T(\lambda)$ is a bijection between $\Lambda$ and the set of isomorphism classes of indecomposable tilting objects of $\mathcal{C}$.
\end{enum_thm}
\end{ass}

\begin{rem} \label{finite_hwc_has_everything}
Suppose that $\mathcal{C}$ is a $K$-linear abelian finite-length category with a \textit{finite} set $\Lambda$ of isomorphism classes of simple objects which is a highest weight category with costandard objects with respect to some partial order on $\Lambda$ as defined in \cite{CPS} such that furthermore all simple objects are \word{absolutely simple}, i.e., $\End_{\mathcal{C}}(S) = K$ for all simple objects $S$ of $\mathcal{C}$. Then all our assumptions are indeed satisfied. First of all, it is shown in \cite[\S1A]{CPS-Duality-in-highest-weight-89} that $\mathcal{C}$ is already a highest weight category with standard objects for the same partial ordering on $\Lambda$, hence our first assumption about $\mathcal{C}$ holds. In the proof of \cite[Theorem 3.11]{CPS} it is shown that the Ext-vanishing property holds (for this we need the assumption that all simple objects are absolutely simple).  Finally, it is a result by Ringel \cite[Proposition 2]{Ringel-Filtrations} that our assumption about tilting objects holds.
\end{rem}

\begin{lem} \label{tilting_standard_homs}
The following holds:
\begin{enum_thm}
\item If $\Hom_\mathcal{C}(\Delta(\lambda),T(\mu)) \neq 0$, then $\lambda \leq \mu$. Moreover, $\Hom_\mathcal{C}(\Delta(\lambda),T(\lambda)) \simeq K$, so the embedding $\Delta(\lambda) \hookrightarrow T(\lambda)$ is unique up to scalars.
\item \label{tilting_standard_homs:nabla} If $\Hom_\mathcal{C}(T(\mu),\nabla(\lambda)) \neq 0$, then $\lambda \leq \mu$. Moreover, $\Hom_\mathcal{C}(T(\lambda),\nabla(\lambda)) \simeq K$, so the projection $T(\lambda) \twoheadrightarrow \Delta(\lambda)$ is unique up to scalars.
\item The composition $\Delta(\lambda) \hookrightarrow T(\lambda) \twoheadrightarrow \nabla(\lambda)$ is non-zero.
\end{enum_thm}
\end{lem}

\begin{proof}
Let $\varphi:\Delta(\lambda) \rarr T(\mu)$ be a non-zero morphism. This induces an isomorphism $\Delta(\lambda)/\Ker \varphi \simeq \Im \varphi$. Since $\varphi$ is non-zero, we have $\Ker \varphi \subsetneq \Delta(\lambda)$, so $\Ker \varphi \subs \Rad \Delta(\lambda)$, implying that $\lbrack \Delta(\lambda)/\Ker \varphi : L(\lambda) \rbrack \neq 0$. Hence, $\lbrack T(\mu):L(\lambda) \rbrack \neq 0$, so $\lambda \leq \mu$ since $T(\mu)$ has highest weight $\mu$. Choose an embedding $i:\Delta(\lambda) \hookrightarrow T(\lambda)$ and let $q:T(\lambda) \twoheadrightarrow T(\lambda)/\Delta(\lambda)$ be the quotient map, where we identify $\Delta(\lambda)$ with $\Im i$. We thus have an exact sequence
\[
\begin{tikzcd}[column sep=small]
0 \arrow{r} & \Delta(\lambda) \arrow{r}{i} & T(\lambda) \arrow{r}{q} & T(\lambda)/\Delta(\lambda) \arrow{r} & 0 \;.
\end{tikzcd}
\]
Applying $\Hom_{\mathcal{C}}(\Delta(\lambda),-)$ yields an exact sequence
\[
\begin{tikzcd}[column sep=tiny, row sep=tiny]
0 \arrow{r} & \Hom_{\mathcal{C}}(\Delta(\lambda),\Delta(\lambda)) \arrow{r} & \Hom_{\mathcal{C}}(\Delta(\lambda),T(\lambda)) \arrow{r} & \Hom_{\mathcal{C}}(\Delta(\lambda),T(\lambda)/\Delta(\lambda))  \\
\phantom{0} \arrow{r} &  \Ext_{\mathcal{C}}^1(\Delta(\lambda),\Delta(\lambda)) \;.
\end{tikzcd}
\]
From \cite[Corollary 4.16]{hwtpaper}, which is true for any highest weight category, we conclude that $\Ext_{\mathcal{C}}^1(\Delta(\lambda),\Delta(\lambda)) = 0$, so $\Hom_{\mathcal{C}}(\Delta(\lambda),T(\lambda)/\Delta(\lambda))$ is the image of $\Hom_{\mathcal{C}}(\Delta(\lambda),T(\lambda))$ under composition with $q$. But it is clear that this image is equal to zero, so the above sequence yields
\[
\Hom_{\mathcal{C}}(\Delta(\lambda),T(\lambda)) \simeq \Hom_{\mathcal{C}}(\Delta(\lambda),\Delta(\lambda)) \simeq K \;,
\]
where we use \cite[Corollary 4.17]{hwtpaper} (which is again true in any highest weight category in which all simple objects are absolutely simple). Part \ref{tilting_standard_homs:nabla} is proven similarly. Finally, suppose that $i:\Delta(\lambda) \hookrightarrow T(\lambda)$ is an embedding and $\pi:T(\lambda) \twoheadrightarrow \nabla(\lambda)$ is a projection such that the composition is zero. Then $\Im i \subs \Ker \pi$, so we get a non-zero epimorphism $T(\lambda)/\Delta(\lambda) \twoheadrightarrow \nabla(\lambda)$. Since $L(\lambda)$ is a constituent of $\Delta(\lambda)$ and of $\nabla(\lambda)$, we deduce that $\lbrack T(\lambda):L(\lambda) \rbrack \geq 2$. This is not possible since $T(\lambda)$ has highest weight $\lambda$, so $\lbrack T(\lambda):L(\lambda) \rbrack = 1$.
\end{proof}

Now, we dive into the theory of Andersen, Stroppel, and Tubbenhauer \cite{AST}. Let us choose for any $\lambda \in \Lambda$ a non-zero morphism 
\begin{equation} \label{c_lambda_def}
c^\lambda: \Delta(\lambda) \rarr \nabla(\lambda)\;.
\end{equation}
By the Ext-vanishing assumption, this is unique up to scalar. By Lemma \ref{tilting_standard_homs} we can choose an embedding $i^\lambda:\Delta(\lambda) \hookrightarrow T(\lambda)$ and a  projection $\pi^\lambda:T(\lambda) \twoheadrightarrow \nabla(\lambda)$ so that we get a factorization
\begin{equation*}
c^\lambda = \pi^\lambda \circ i^\lambda \;.
\end{equation*}

\begin{lem} \label{tilting_lift_lemma}
Let $\lambda \in \Lambda$. The following holds:
\begin{enum_thm}
\item Let $M \in \mathcal{C}^\Delta$. Then any morphism $f:M \rarr \nabla(\lambda)$ factors through a morphism $\hat{f}:M \rarr T(\lambda)$, i.e., the diagram
\[
\begin{tikzcd}
M \arrow{r}{\hat{f}} \arrow{dr}[swap]{f}  & T(\lambda)\arrow[twoheadrightarrow]{d}{\pi^\lambda} \\
& \nabla(\lambda)
\end{tikzcd}
\]
commutes.
\item Let $N \in \mathcal{C}^\nabla$. Then any morphism $g:\Delta(\lambda) \rarr N$ extends to a morphism $\hat{g}:T(\lambda) \rarr N$, i.e., the diagram
\[
\begin{tikzcd}
T(\lambda) \arrow{r}{\hat{g}} & N \\
\Delta(\lambda) \arrow[hookrightarrow]{u}{i^\lambda} \arrow{ur}[swap]{g}
\end{tikzcd}
\]
commutes.
\end{enum_thm}
\end{lem}

\begin{proof}
We just consider the first statement, the second is dual. Since $\nabla(\lambda)$ occurs at the top of a costandard filtration of $T(\lambda)$, it is clear that $\Ker \pi^\lambda$ has a costandard filtration and so it follows from Lemma \ref{tilting_ext_vanish} that $\Ext_\mathcal{C}^1(M, \Ker \pi^\lambda) = 0$ since $M$ has a standard filtration. Applying $\Hom_\mathcal{C}(M,-)$ to the exact sequence $0 \rarr \Ker \pi^\lambda \rarr T(\lambda) \overset{\pi^\lambda}{\rarr}\nabla(\lambda) \rarr 0$ proves the claim.
\end{proof}

In the following we let $M \in \mathcal{C}^\Delta$ and $N \in \mathcal{C}^\nabla$. For a $K$-basis $F_M^\lambda$ of $\Hom_\mathcal{C}(M,\nabla(\lambda))$ and a choice of lift $\hat{f}$ as in Lemma \ref{tilting_lift_lemma} for every $f \in F_M^\lambda$ we set 
\begin{equation*}
\hat{F}_M^\lambda \dopgleich \lbrace \hat{f} \mid f \in F_M^\lambda \rbrace \subs \Hom_\mathcal{C}(M,T(\lambda)) \;.
\end{equation*}
Similarly, for a $K$-basis $G_N^\lambda$ of $\Hom_\mathcal{C}(\Delta(\lambda),N)$  and a choice of lifts $\hat{g}$ we set 
\begin{equation*}
\hat{G}_N^\lambda \dopgleich \lbrace \hat{g} \mid g \in G_N^\lambda \rbrace \subs \Hom_\mathcal{C}(T(\lambda),N)\;.
\end{equation*}
Note that the lifts, and thus the subsets $\hat{F}_M^\lambda$ and $\hat{G}_N^\lambda$, are not unique. For any such choice we define the subset
\begin{equation*}
\hat{G}_N^\lambda \hat{F}_M^\lambda \dopgleich \lbrace \hat{g} \circ \hat{f} \mid \hat{g} \in \hat{G}_N^\lambda, \hat{f} \in \hat{F}_M^\lambda \rbrace \subs \Hom_\mathcal{C}(M,N) \;.
\end{equation*}
The elements of this subset can be illustrated by the commutative diagram
\begin{equation*}
\begin{tikzcd}
& \Delta(\lambda) \arrow[hookrightarrow]{d}[swap]{i^\lambda} \arrow{dr}{g} \\
M \arrow{r}{\hat{f}} \arrow{dr}[swap]{f} & T(\lambda) \arrow{r}[swap]{\hat{g}} \arrow[twoheadrightarrow]{d}{\pi^\lambda} & N \\
& \nabla(\lambda)
\end{tikzcd}
\end{equation*}
Let
\begin{equation*}
\hat{G}_N \hat{F}_M \dopgleich \bigcup_{\lambda \in \Lambda} \hat{G}_N^\lambda \hat{F}_M^\lambda \;.
\end{equation*}
Our aim, following \cite{AST}, is to show:

\begin{thm} \label{cellular_appendix_basis_independent}
Independently of all choices the set $\hat{G}_N \hat{F}_M$ is a $K$-vector space basis of $\Hom_{\mathcal{C}}(M,N)$.
\end{thm}

The proof of the above theorem, which is \cite[Theorem 4.1]{AST}, needs some preparation and some modifications to make it work in our setting. The key point is the independence of the choice of bases and lifts. This is where \cite{AST} begin to use \textit{weight spaces}. Their idea is to use the filtration of the Hom-spaces given by restrictions of morphisms to weight spaces, with \cite[Lemma 4.5]{AST} and \cite[Lemma 4.6]{AST} being the central ingredients. We will give an alternative proof of this fact, not relying on the notion of weight spaces. This is based on the following minor modification of the constructions in \cite{AST}. For $\lambda \in \Lambda$ and $\varphi \in \Hom_{\mathcal{C}}(M,N)$, instead of letting $\varphi_\lambda$ be the restriction of $\varphi$ to the $\lambda$-weight space as in \cite{AST}, we define 
\begin{equation} \label{cellular_appendix_key_modification}
\varphi_\lambda \dopgleich \lbrack \Im \varphi:L(\lambda) \rbrack  \in \bbN \;.
\end{equation}
The following set can then be defined similarly as in \cite{AST}:
\begin{equation*}
\Hom_{\mathcal{C}}(M,N)^{\leq \lambda} \dopgleich \lbrace \varphi \in \Hom_\mathcal{C}(M,N) \mid \varphi_\mu = 0 \tn{ unless } \mu \leq \lambda \rbrace \;.
\end{equation*}

\begin{lem} \label{hom_weight_space_is_space}
The set $\Hom_{\mc{C}}(M,N)^{\leq \lambda}$ is a sub vector space of $\Hom_{\mc{C}}(M,N)$.
\end{lem}

\begin{proof} 
Let $\phi,\psi \in \Hom_{\mc{C}}(M,N)^{\leq \lambda}$. Since $\Im \phi, \Im \psi, \Im(\phi+\psi) \subs \Im \phi+ \Im \psi$, we have
\[
(\phi+ \psi)_\mu = \lbrack \Im(\phi+\psi):L(\mu) \rbrack \leq \lbrack \Im \phi : L(\mu) \rbrack + \lbrack \Im \psi : L(\mu) \rbrack = \phi_\mu + \psi_\mu \;.
\]
If $\mu \not\leq \lambda$, then $\phi_\mu = 0 = \psi_\mu$ since $\phi,\psi \in \Hom_\mc{C}(M,N)^{\leq \lambda}$. Hence $(\phi+\psi)_\mu = 0$, implying that $\phi+\psi \in \Hom_{\mc{C}}(M,N)^{\leq \lambda}$.
\end{proof}

\begin{lem} \label{f_plus_g_weight}
If $f,g \in \Hom_\mc{C}(M,N)$ and $f_\lambda = 0$, then $(f+g)_\lambda = g_\lambda$.
\end{lem}

\begin{proof}
Let $\ol{f+g}:M \rarr N/\Im(f)$ and $\ol{g}:M \rarr N/\Im(f)$ be the map induced by $f+g$ and $g$, respectively. Clearly, $\ol{f+g}=\ol{g}$. Since $0=f_\lambda = \lbrack \Im(f):L(\lambda) \rbrack$, we have $(f+g)_\lambda = (\ol{f+g})_\lambda = \ol{g}_\lambda = g_\lambda$.
\end{proof}

The following lemma contains the key ingredients for the proof of Theorem \ref{cellular_appendix_basis_independent}. The three statements are essentially  \cite[Proposition 4.4]{AST}, \cite[Lemma 4.5]{AST}, and \cite[Lemma 4.6]{AST}. The most crucial is the second one which is a sufficient version of \cite[Lemma 4.5]{AST} that we can prove in our setting. 

\begin{lem} \label{cellular_appendix_ingredients}
Let $F_M^\lambda$ and $\hat{F}_M^\lambda$ be arbitrary for all $\lambda \in \Lambda$. There is a choice of $G_N^\lambda$ and $\hat{G}_N^\lambda$ for all $\lambda \in \Lambda$ satisfying the following properties:
\begin{enum_thm}
\item \label{cellular_appendix_ingredients:basis} $\hat{G}_N\hat{F}_M$ is a $K$-basis of $\Hom_\mc{C}(M,N)$.
\item \label{cellular_appendix_ingredients:lemma45} If $\phi$ is a non-zero element of the $K$-span $\langle \hat{G}_N^\lambda \hat{F}_M^\lambda \rangle_K$ of $\hat{G}_N^\lambda \hat{F}_M^\lambda$, then $\phi_\lambda$ is non-zero.
\item \label{cellular_appendix_ingredients:filtered_basis}The set $(\hat{G}_N\hat{F}_M)^{\leq \lambda} \dopgleich \bigcup_{\mu \leq \lambda} \hat{G}_N^\mu \hat{F}_M^\mu$ is a $K$-basis of $\Hom_\mc{C}(M,N)^{\leq \lambda}$.
\end{enum_thm}
\end{lem}

\begin{proof} 
We will show parts \ref{cellular_appendix_ingredients:basis} and \ref{cellular_appendix_ingredients:lemma45} by induction on the length of a costandard filtration of $N$. So, assume that $N =\nabla(\mu)$. We choose bases and lifts as in \ref{cellular_appendix_ingredients:basis}. First, we argue that $\hat{G}_N^\lambda \hat{F}_M^\lambda = \emptyset$ if $\lambda \neq \mu$. By definition, an element $\phi$ of $\hat{G}_N^\lambda \hat{F}_M^\lambda$  factorizes as $\phi = \hat{g} \hat{f}$ for some $\hat{f} \in \hat{F}_M^\lambda$ and $\hat{g} \in \hat{G}_N^\lambda$. Moreover, $\hat{g}$ is a lift of a morphism $g \in \Hom_\mc{C}(\Delta(\lambda),\nabla(\mu))$. Because of Assumption \ref{tilting_ext_assumption} we have $\Hom_\mc{C}(\Delta(\lambda),\nabla(\mu)) = 0$ whenever $\lambda \neq \mu$. Hence, if $\lambda \neq \mu$, then $G_N^\lambda = \emptyset$, thus $\hat{G}_N^\lambda = \emptyset$ and $\hat{G}_N^\lambda \hat{F}_M^\lambda = \emptyset$. Consequently, if $\phi \in \langle \hat{G}_N^\lambda \hat{F}_M^\lambda \rangle_K$ is non-zero, we must have $\lambda = \mu$. In this case, $\phi$ is a non-zero morphism $M \rarr \nabla(\lambda)=N$. The image is a non-zero submodule of $\nabla(\lambda)$, thus contains $\Soc \nabla(\lambda) = L(\lambda)$, hence $\phi_\lambda \neq 0$. 

Now, let $N$ be arbitrary with costandard filtration $0 = N_0 \subset N_1 \subset \dots \subset N_{k-1} \subset N_k = N$. Let $\mu \in \Lambda$ with $N_k/N_{k-1} \simeq \nabla(\mu)$. Applying $\Hom_\mc{C}(M,-)$ to the exact sequence
\[
0 \rarr N_{k-1} \rarr N \overset{\pi}{\rarr} \nabla(\mu) \rarr 0 
\]
yields the exact sequence
\begin{equation} \label{cellularity_hom_exact_sequence}
0 \rarr \Hom_\mc{C}(M,N_{k-1}) \overset{\mrm{inc}}{\rarr} \Hom_\mc{C}(M,N) \rarr \Hom_\mc{C}(M,\nabla(\mu)) \rarr 0 \;,
\end{equation}
the exactness on the right hand side following from the fact that $M \in \mc{C}^\Delta$, $N_{k-1} \in \mc{C}^\nabla$, and Assumption \ref{tilting_ext_assumption}. By the induction assumption, we can choose a lift $\hat{G}_{N_{k-1}}$ such that the basis $\hat{G}_{N_{k-1}} \hat{F}_{M}$ of $\Hom_\mc{C}(M,N_{k-1})$ satisfies both (a) and (b).  We will argue as in the second half of the proof of \cite[Proposition 4.4]{AST} that we can choose a suitable lift $\hat{G}_N$ such that the basis $\hat{G}_N\hat{F}_M = \bigcup_{\lambda \in \Lambda} \hat{G}_N^\lambda \hat{F}_M^\lambda$ of $\Hom_\mc{C}(M,N)$ satisfies \ref{cellular_appendix_ingredients:lemma45}. The choice of lift of $\hat{G}_N^\lambda$ will depend on whether $\lambda \neq \mu$ or $\lambda = \mu$. If $\lambda \neq \mu$, we set $G_N^\lambda \dopgleich \mrm{inc}\left(G_{N_{k-1}}^\lambda\right)$ and $\hat{G}_N^\lambda \dopgleich \mrm{inc}\left (\hat{G}_{N_{k-1}}^\lambda\right)$. If $\lambda = \mu$, then $\mrm{inc}\left(G_{N_{k-1}}^\lambda\right)$ is still linearly independent but since $\lbrack N:\nabla(\lambda) \rbrack = \lbrack N_{k-1}:\nabla(\lambda) \rbrack + 1$, we have $\dim_K \Hom_\mc{C}(\Delta(\lambda),N) = \dim_K \Hom_\mc{C}(\Delta(\lambda),N_{k-1})+1$, so we need one more basis element. As argued in the proof of \cite[Proposition 4.4]{AST} we can choose $g^\lambda:\Delta(\lambda) \rarr N$ such that $\pi \circ g^\lambda = c^\lambda$, where $c^\lambda$ is as in (\ref{c_lambda_def}). Then the set $\mrm{inc}(G_{N_{k-1}}^\lambda) \cup \lbrace g^\lambda \rbrace$ is a basis of $\Hom_\mc{C}(\Delta(\lambda),N)$. Let $\hat{g}^\lambda:T(\lambda) \rarr N$ be any lift and set $\hat{G}_N^\lambda \dopgleich \mrm{inc}(\hat{G}_{N_{k-1}}^\lambda) \cup \lbrace \hat{g}^\lambda \rbrace$. With these definitions of $\hat{G}_N^\lambda$ it follows from the proof of \cite[Proposition 4.4]{AST} that $\hat{G}_N\hat{F}_M$ is a basis of $\Hom_\mc{C}(M,N)$, so \ref{cellular_appendix_ingredients:basis} holds.

We argue that this basis satisfies \ref{cellular_appendix_ingredients:lemma45}. Let $\phi \in \langle \hat{G}_N^\lambda \hat{F}_M^\lambda \rangle_K$ be non-zero. First, assume that $\lambda \neq \mu$. Then, by definition of $\hat{G}_N^\lambda\hat{F}_M^\lambda$, there is a non-zero $\tilde{\phi} \in \langle \hat{G}_{N_{k-1}}^\lambda \hat{F}_M^\lambda \rangle_K$ with $\phi = \mrm{inc}(\tilde{\phi})$. By the induction hypothesis, we have $\tilde{\phi}_\lambda \neq 0$. Since $N_{k-1} \rarr N$ is an embedding, it follows that $\phi_\lambda \geq \tilde{\phi}_\lambda$, and $\phi_\lambda \neq 0$ as claimed. Now, assume that $\lambda = \mu$. By definition of $\hat{G}_N^\lambda$ we can write $\phi = \mrm{inc}(\wt{\phi}) + \phi'$, where $\wt{\phi} \in \langle \hat{G}_{N_{k-1}}^\lambda F_M^\lambda \rangle_K$ and $\phi' \in \langle \hat{g}^\lambda F_M^\lambda \rangle_K$. First, assume that $\phi' \neq 0$. Because of the exact sequence (\ref{cellularity_hom_exact_sequence}) we have $\pi \circ \phi' \neq 0$. This is a morphism $M \rarr \nabla(\lambda)$, and since it is non-zero, we have $(\pi \circ \phi')_\lambda \neq 0$. Since $\pi \circ \mrm{inc}(\wt{\phi}) = 0$, we have $\pi \circ \phi = \pi \circ \phi'$. Thus $(\pi \circ \phi)_\lambda \neq 0$, implying that $\phi_\lambda \neq 0$ too. On the other hand, if $\phi' = 0$, we must have $\wt{\phi} \neq 0$ and the same argument, as above, shows that $\phi_\lambda \neq 0$.

Let us now choose bases and lifts satisfying \ref{cellular_appendix_ingredients:basis} and \ref{cellular_appendix_ingredients:lemma45}. As in \cite{AST} we write $F_M^\lambda = \lbrace f_j^\lambda \mid j \in \ms{J}^\lambda \rbrace$ and $G_N^\lambda = \lbrace g_i^\lambda \mid i \in \ms{I}^\lambda \rbrace$. Let $c_{ij}^\lambda \dopgleich \hat{g}_i^\lambda \hat{f}_j^\lambda \in \hat{G}_N^\lambda \hat{F}_M^\lambda$. If $(c_{ij}^\mu)_\nu \neq 0$, then by definition $\lbrack \Im c_{ij}^\mu : L(\nu) \rbrack \neq 0$. This implies in particular that $\lbrack \hat{g}_i^\mu : L(\nu) \rbrack \neq 0$. Recall that $\hat{g}_i^\mu$ is a morphism $T(\mu) \rarr N$. Since $T(\mu)$ has highest weight $\mu$, we conclude that $\nu \leq \mu$. In other words, $(c_{ij}^\mu)_\nu = 0$ unless $\nu \leq \mu$. Moreover, as $0 \neq c_{ij}^\mu \in \hat{G}_N^\mu \hat{F}_M^\mu$, we know from \ref{cellular_appendix_ingredients:lemma45} that $(c_{ij}^\mu)_\mu \neq 0$. Hence, $c_{ij}^\mu \in \Hom_\mc{C}(M,N)^{\leq \lambda}$ if and only if $\mu \leq \lambda$. In particular, $\langle (\hat{G}_N\hat{F}_M)^{\leq \lambda} \rangle_K \subs \Hom_\mc{C}(M,N)^{\leq \lambda}$. By assumption $(\hat{G}_N\hat{F}_M)^{\leq\lambda}$ is linearly independent, so we just need to show that it spans $\Hom_\mc{C}(M,N)^{\leq \lambda}$. Let $\phi \in \Hom_{\mc{C}}(M,N)^{\leq \lambda}$ be non-zero. Since $\hat{G}_N\hat{F}_M$ is a basis of $\Hom_\mc{C}(M,N)$, we can write
\[
\phi = \sum_{ \substack{\mu \in \Lambda \\ i \in \ms{I}^\mu , \ j \in \ms{J}^\mu }} a_{ij}^\mu c_{ij}^\mu 
\]
with certain $a_{ij}^\mu \in K$, not all zero. Choose $\mu \in \Lambda$ maximal with the property that $a_{i j}^\mu \neq 0$ for some $i \in \ms{I}^\mu$ and $j \in \ms{J}^\mu$. Let 
\[
\phi^\mu \dopgleich \sum_{i\in \ms{I}^\mu,j\in \ms{J}^\mu} a_{ij}^\mu c_{ij}^\mu \tn{ and } \phi^{\neq \mu} \dopgleich \sum_{ \substack{\nu \neq \mu \\ i \in \ms{I}^\nu, j \in \ms{J}^\nu}} a_{ij}^\nu c_{ij}^\nu \;,
\]
so $\phi = \phi^\mu + \phi^{\neq \mu}$. Note that $\phi^\mu \in \langle \hat{G}_N^\mu \hat{F}_M^\mu \rangle_K$ is non-zero. Hence, we know from \ref{cellular_appendix_ingredients:lemma45} that $(\phi^\mu)_\mu \neq 0$. Moreover, the maximality of $\mu$ and the arguments above show that $(\phi^{\neq \mu})_\mu = 0$. Hence, $\phi_\mu = (\phi^\mu + \phi^{\neq \mu})_\mu = (\phi^\mu)_\mu \neq 0$ by Lemma \ref{f_plus_g_weight}. By definition of $\Hom_\mc{C}(M,N)^{\leq \lambda}$, this implies $\mu \leq \lambda$. Hence, by what we have said above, we have $c_{ij}^\mu \in \Hom_{\mc{C}}(M,N)^{\leq \lambda}$ for all $i,j$. In total, this shows that $\phi \in \langle (\hat{G}_N\hat{F}_M)^{\leq \lambda} \rangle_K$. 
\end{proof}

\begin{proof}[Proof of Theorem \ref{cellular_appendix_basis_independent}]
Using Proposition \ref{cellular_appendix_ingredients}\ref{cellular_appendix_ingredients:filtered_basis}, which is \cite[Lemma 4.6]{AST}, the statement can now be proven by the exact same arguments as in \cite[Lemma 4.7]{AST} and \cite[Lemma 4.8]{AST}.
\end{proof}

For a tilting object $T \in \mathcal{C}$ let us define
\begin{equation*}
\mathcal{P}_T \dopgleich \lbrace \lambda \in \Lambda \mid G_T^\lambda F_T^\lambda \neq \emptyset \rbrace = \lbrace \lambda \in \Lambda \mid \lbrack T:\Delta(\lambda) \rbrack \neq 0 \tn{ and } \lbrack T:\nabla(\lambda) \rbrack \neq 0 \rbrace \;.
\end{equation*}
We equip this set with the partial order $\leq$ from $\Lambda$. Set
\begin{equation*}
E_T \dopgleich \End_{\mc{C}}(T) \;.
\end{equation*}
The exact same arguments as given in the proof of \cite[Theorem 4.11]{AST} now show the following main theorem (for the definitions of a \word{standard datum} and a \word{cellular datum} see \S\ref{cellular_structure}).

\begin{thm} \label{end_tilting_is_cellular}
For any tilting object $T \in \mathcal{C}$ the $K$-algebra $E_T$ admits a standard datum over the poset $\mathcal{P}_T$ with standard basis $\hat{G}_T\hat{F}_T$. If moreover $\mathcal{C}$ is equipped with a duality $\bbD$ such that $\bbD(T) \simeq T$, this standard datum together with the anti-involution on $E_T$ induced by $\bbD$ is cellular, so $E_T$ is a cellular algebra. \qed
\end{thm}

Generalizing the cell modules defined by Graham–Lehrer \cite[Definition 2.1]{Graham-Lehrer-Cellular}, Du–Rui \cite[Definition 2.1.2]{DuRui} define (cellular) standard and costandard modules attached to any $\lambda \in \mathcal{P}_T$. Though related, these should not be confused with the standard and costandard modules in the highest weight category $\mc{C}$. The (cellular) \word{standard module} $\Delta_T(\lambda)$ attached to $\lambda \in \mc{P}_T$ is a left $E_T$-module defined using the the coefficients for left multiplication appearing in the definition of a standard datum. Similarly, the (cellular) \word{costandard module} $\nabla_T(\lambda)$ is a left $E_T$-module which is the dual of an analogous right module defined using right multiplication. In the terminology of Graham–Lehrer the left $E_T$-module $\Delta_T(\lambda)$ is the left cell module attached to $\lambda$ and the dual of $\nabla_T(\lambda)$, a right $E_T$-module, is the right cell module (or dual cell module) attached to $\lambda$. 

The arguments in \cite[\S4.1]{AST} show that we have 
\begin{equation*}
\Delta_T(\lambda) \simeq \Hom_{\mathcal{C}}(\Delta(\lambda),T) \quad \tn{and} \quad \nabla_T(\lambda)^* \simeq \Hom_{\mathcal{C}}(T,\nabla(\lambda)) \;.
\end{equation*}
This is due to the particular form of the standard basis and shows that these modules do not depend on the choice of the bases $F^\lambda$ and $G^\lambda$. Generalizing the construction of Graham–Lehrer, Du–Rui \cite[2.3]{DuRui} define for any $\lambda \in \mathcal{P}_T$ a bilinear pairing $\beta_\lambda$ between $\Delta_T(\lambda)$ and $\nabla_T(\lambda)^*$. It then follows from \cite[Theorem 2.4.1]{DuRui} that the subset
\begin{equation*}
\ul{\mathcal{P}}_{T} \dopgleich \lbrace \lambda \in \mathcal{P}_T \mid \beta_\lambda \neq 0 \rbrace \subs \mathcal{P}_T
\end{equation*}
classifies the simple $E_T$-modules: the cellular standard module $\Delta_T(\lambda)$ has simple head $L_T(\lambda)$ if and only if $\beta_\lambda \neq 0$, and $\lambda \mapsto L_T(\lambda)$ is a bijection between $\ul{\mathcal{P}}_{T}$ and the set of isomorphism classes of simple left $E_T$-modules. 

As shown in \cite[Definition 1.2.1]{DuRui}, the opposite algebra $E_T^\op$ is naturally equipped with a standard datum, the \word{opposite standard datum},  which has the same indexing poset $\mc{P}_T$ and flipped basis parts. The standard and costandard modules for this standard datum are then given by
\begin{equation} \label{opposite_cellular_standard_modules}
\Delta_T^\op(\lambda) = \nabla_T(\lambda)^* = \Hom_{\mc{C}}(T,\nabla(\lambda)) \;\ \tn{and} \;\ \nabla_T^\op(\lambda) = \Delta_T(\lambda)^* = \Hom_{\mc{C}}(\Delta(\lambda),T)^* \;,
\end{equation}
respectively. \\

The arguments by Andersen–Stroppel–Tubbenhauer \cite[\S4, Theorem 4.12, Theorem 4.13]{AST} can be used word-for-word to prove the following two theorems:

\begin{thm} \label{cell_appendix_dim_simples}
If $\lambda \in \ul{\mathcal{P}}_{T}$, then $\dim L_T(\lambda) = \lbrack T:T(\lambda) \rbrack$, the multiplicity of $T(\lambda)$ in $T$.
\end{thm}

\begin{thm} \label{cell_appendix_semisimple}
The algebra $E_T$ is semisimple if and only if $T \in \mathcal{C}$ is semisimple.  
\end{thm}

\section{A relative Morita theorem} \label{morita_appendix}

We state and prove a relative Morita theorem, since it is used several times in the article, and is not quite the usual setting one finds in the literature. 

Let $X$ be a set. We say that $E$ is an $X$-algebra if it is a $K$-algebra (not necessarily with unit) and there are idempotents $\{ e_{\lambda} \ | \ \lambda \in X \}$ in $E$ such that 
$$
E = \bigoplus_{\lambda,\mu \in X} e_{\lambda} E e_{\mu}. 
$$
We say that an $E$-module $M$ is \textit{unital} if it admits a weight decomposition $M = \bigoplus_{\lambda \in X} e_{\lambda} M$. Let $\Lmod{E}$ denote the category of finite-dimensional unital left $E$-modules. 

Let $\mc{A}$ be a $K$-linear Artinian category with finite dimensional hom spaces and enough projectives. Choose $X \subset \Irr \mc{A}$ and let $\mc{B}$ be the Serre subcategory of $\mc{A}$ generated by $\Irr \mc{A} \smallsetminus X$. Let $\pi : \mc{A} \rightarrow \mc{A} / \mc{B}$ be the quotient map, and $P(\lambda)$ the projective cover of $\lambda \in \Irr \mc{A}$ in $\mc{A}$. The following is a consequence of \cite[Theorem 3.5]{CPS}, but we include a proof for completeness.

\begin{lem}\label{lem:projcoverquotient}
For each $\lambda \in X$, the object $Q(\lambda) := \pi(P(\lambda))$ is the projective cover of $\pi(\lambda)$ in $\mc{A} / \mc{B}$. In particular, $\mc{A} / \mc{B}$ has enough projectives. 
\end{lem}

\begin{proof}
We begin by showing that $\pi(P(\lambda))$ is projective in $\mc{A} / \mc{B}$. Let $0 \rightarrow U \rightarrow V \rightarrow W \rightarrow 0$ be a short exact sequence in $\mc{A} / \mc{B}$. By \cite[Corollary 15.8]{FaithAlgebraI}, there exists a short exact sequence $0 \rightarrow U' \rightarrow V' \rightarrow W' \rightarrow 0$ in $\mc{A}$ whose image under the quotient functor is $0 \rightarrow U \rightarrow V \rightarrow W \rightarrow 0$. We get a short exact sequence 
$$
0 \rightarrow \Hom_{\mc{A}}(\pi(P(\lambda)),U) \rightarrow \Hom_{\mc{A}}(\pi(P(\lambda)),V') \rightarrow \Hom_{\mc{A}}(\pi(P(\lambda)),W') \rightarrow 0
$$
and it suffices to show that the quotient functor identifies the $\Hom_{\mc{A}}(\pi(P(\lambda)),U')$ with $\Hom_{\mc{A} / \mc{B}}(\pi(P(\lambda)),U)$. Since the head $\pi(\lambda)$ of $\pi(P(\lambda))$ is non-zero in $\mc{A} / \mc{B}$, 
$$
\Hom_{\mc{A} / \mc{B}}(\pi(P(\lambda)),U) = \lim_{U'' \in \mc{B}} \Hom_{\mc{A}}(\pi(P(\lambda)),U'/ U'').
$$
Since $\pi(P(\lambda))$ is projective, the map $\Hom_{\mc{A}}(\pi(P(\lambda)),U') \rightarrow \Hom_{\mc{A} / \mc{B}}(\pi(P(\lambda)),U)$ is surjective. If it were not injective then there would exist a non-zero morphism $\phi$ in $\Hom_{\mc{A}}(\pi(P(\lambda)),U')$ whose image is in $ \mc{B}$. But $\lambda$ does not belong to $ \mc{B}$, so this cannot happen. 
\end{proof}

We drop $\pi$ from the notation and write $\lambda$ for the image of $\lambda$ in $\mc{A} / \mc{B}$. Let 
\begin{align*}
E & = \left( \bigoplus_{\lambda,\mu \in X} \Hom_{\mc{A}/ \mc{B}}(Q(\lambda),Q(\mu)) \right)^{\op},\\
 & =  \left( \bigoplus_{\lambda,\mu \in X} \Hom_{\mc{A}}(P(\lambda),P(\mu)) \right)^{\op}.
\end{align*}
Then $E$ is an $X$-algebra with $e_{\lambda}$ the identity in $\Hom_{\mc{A}/ \mc{B}}(Q(\lambda),Q(\lambda))$. 

\begin{lem}\label{lem:counitiso}
For each $\lambda \in X$, the canonical map 
$$
\bigoplus_{\mu \in X} Q(\mu) \otimes_E \left( \bigoplus_{\rho \in X} \Hom_{\mc{A} / \mc{B}}(Q(\rho), Q(\lambda)) \right) \rightarrow Q(\lambda)
$$
is an isomorphism in $\mc{A} / \mc{B}$. 
\end{lem}

\begin{proof}
Let $X_1 = \{ \rho \in X \ | \ [Q(\lambda): \rho] \neq 0 \}$ and let $X_2 = \{ \mu \in X \ | \ [Q(\rho) : \mu ] \neq 0 \textrm{ for some } \rho \in X_1 \}$. Then both $X_1$ and $X_2$ are finite sets and 
\begin{align*}
Q &:= \bigoplus_{\mu \in X} Q(\mu) \otimes_E \left( \bigoplus_{\rho \in X} \Hom_{\mc{A} / \mc{B}} (Q(\rho),Q(\lambda)) \right) \\ &= \bigoplus_{\mu \in X_2} Q(\mu) \otimes_E \left( \bigoplus_{\rho \in X_1} \Hom_{\mc{A} / \mc{B}} (Q(\rho),Q(\lambda)) \right). 
\end{align*}
By Yoneda's lemma, the morphism will be an isomorphism if and only if the map $\Hom_{\mc{A} / \mc{B}} (Q,M) \rightarrow \Hom_{\mc{A}/ \mc{B}}(Q(\lambda),M)$ is an isomorphism for all $M$. By induction on the length of $M$, we may assume that $M$ is simple. Since $Q(\mu)$ is the projective cover of $\pi(\mu)$ in $\mc{A}/ \mc{B}$, the statement is clear in this case.
\end{proof}

\begin{thm}\label{thm:relativeMorita}
The functor $F : \mc{A} / \mc{B} \rightarrow \Lmod{E}$, 
$$
F(M) = \bigoplus_{\lambda \in X} \Hom_{\mc{A} / \mc{B}}(Q(\lambda),M)
$$
is an equivalence, with quasi-inverse
$$
N \mapsto \bigoplus_{\lambda \in X} Q(\lambda) \otimes_E N \simeq \bigoplus_{\lambda \in X} \Hom_E(Q(\lambda) ,  N). 
$$
\end{thm} 

\begin{proof}
To prove that $F$ is an equivalence with quasi-inverse $G := \bigoplus_{\lambda \in X} Q(\lambda) \otimes_E -$, it suffices to show that $F$ is essentially surjective and fully faithful. Since $G$ is left adjoint to $F$, we have a unit $\eta : 1 \rightarrow F \circ G$ and counit $\epsilon : G \circ F \rightarrow 1$. We first show that $\epsilon$ is an isomorphism. By Lemma \ref{lem:counitiso}, it is an isomorphism on all projective modules in $\mc{A} / \mc{B}$. Next fix a presentation $\bigoplus_{i} Q(\lambda_i) \rightarrow \bigoplus_{j} Q(\mu_j) \rightarrow M \rightarrow 0$, where both sums are finite. Then we have a commutative diagram 
$$
\begin{tikzcd}
G \circ F \left( \bigoplus_{i} Q(\lambda_i) \right)  \arrow{d}{\wr}  \arrow{r} &  G \circ F \left( \bigoplus_{j} Q(\mu_j) \right)  \arrow{d}{\wr}  \arrow{r} & G \circ F (M)  \arrow{r}  \arrow{d}{\epsilon} & 0 \\
 \bigoplus_{i} Q(\lambda_i)  \arrow{r} & \bigoplus_{j} Q(\mu_j)  \arrow{r} & M  \arrow{r} & 0
\end{tikzcd}
$$
which implies that $\epsilon : G \circ F (M) \rightarrow M$ is an isomorphism. We can use this to prove that $F$ is fully faithful. Let $M_1,M_2 \in \mc{A} / \mc{B}$. Then we wish to show that $F_{1,2} : \Hom_{\mc{A}/ \mc{B}}(M_1,M_2) \rightarrow \Hom_{E}(F(M_1),F(M_2))$ is an isomorphism. But the composite 
\begin{align*}
\Hom_{\mc{A}/ \mc{B}}(M_1,M_2) & \stackrel{F_{1,2}}{\longrightarrow} \Hom_{E}(F(M_1),F(M_2)) \rightarrow \Hom_{\mc{A}/ \mc{B}}(G \circ F(M_1),M_2) \\& \stackrel{\epsilon}{\longrightarrow} \Hom_{\mc{A}/ \mc{B}}(M_1,M_2)
\end{align*}
is the identity on $\Hom_{\mc{A}/ \mc{B}}(M_1,M_2)$, and all arrows except the first are isomorphisms. Thus, $F_{1,2}$ is an isomorphism. 

Finally, we just need to show that $F$ is essentially surjective. First we note that for a fixed $\mu \in X$, there are only finitely many $\lambda \in X$ such that $e_{\lambda} E e_{\mu} \neq 0$. This follows from the fact that $e_{\lambda} E e_{\mu} = \Hom_{\mc{A} / \mc{B}}(Q(\lambda),Q(\mu))$ and $Q(\mu)$ has finite length. Therefore $E e_{\mu}$ is finite dimensional and each $K \in \Lmod{E}$ admits a presentation $\bigoplus_i E e_{\lambda_i} \rightarrow \bigoplus_j E e_{\mu_j} \rightarrow K \rightarrow 0$, with \textit{both} sums being finite. Since tensoring is right exact, we have an exact sequence 
$$
\bigoplus_i Q(\lambda_i) \rightarrow \bigoplus_j Q(\mu_j) \rightarrow \bigoplus_{\lambda} Q(\lambda) \otimes_E K \rightarrow 0 
$$
and applying the exact functor $F$, we get a commutative diagram 
$$
\begin{tikzcd}
\bigoplus_i E e_{\lambda_i}  \arrow{d}{\wr}  \arrow{r} &  \bigoplus_j E e_{\mu_j}  \arrow{d}{\wr} \ar[r] & N  \arrow{r}  \arrow{d}{\eta} & 0 \\
G \circ F \left( \bigoplus_i E e_{\lambda_i} \right)  \arrow{r} & G \circ F\left( \bigoplus_j E e_{\mu_j} \right)  \arrow{r} & G \circ F (N)  \arrow{r} & 0
\end{tikzcd}
$$
with exact rows. This implies that $\eta : N \rightarrow F \circ G (N)$ is an isomorphism and that $F$ is essentially surjective. 
\end{proof}

Let 
\[
R = \left( \bigoplus_{\mu,\rho \in \Irr \mc{A}} \Hom_{\mc{A}}(P(\mu),P(\rho)) \right)^{\op} \;,
\]
an $\Irr \mc{A}$-algebra such that the category of finite-dimensional unital $R$-modules is equivalent to $\mc{A}$. 

\begin{lem}\label{lem:adjointiso}
For each $M \in \mc{A}$, the canonical morphism 
$$
\bigoplus_{\stackrel{\lambda \in X}{\mu \in \Irr \mc{A}}} \Hom_{\mc{A}}(P(\lambda),P(\mu)) \otimes_R M \rightarrow \pi(M)
$$
is an isomorphism of $E$-modules. 
\end{lem}

\begin{proof}
It suffices to check that the morphism is well-defined. But this follows, as in the proof of Proposition \ref{prop:adjointstopi}, from the fact that there are finitely many $\mu_1, \ds, \mu_{\ell}$ such that $\Hom_{\mc{A}}(P(\rho),M) =  0$ for all $\rho \in \Irr \mc{A} \smallsetminus \{ \mu_1, \ds, \mu_{\ell} \}$ and finitely many $\lambda_1, \ds, \lambda_k$ such that $\Hom_{\mc{A}}(P(\lambda),P(\mu_i)) =  0$ for all $\lambda \in X \smallsetminus \{ \lambda_1, \ds, \lambda_k \}$.
\end{proof}

\begin{prop}\label{prop:adjointstopi}
The quotient functor $\pi$ admits both a left adjoint $\pi^*$ and a right adjoint $\pi^!$, given by 
$$
\pi^*(N) = \bigoplus_{\lambda \in X} P(\lambda) \otimes_E N, \quad \pi^!(N) = \bigoplus_{\stackrel{\lambda \in X}{\mu \in \Irr \mc{A}}} \Hom_{E}(\Hom_{\mc{A}}(P(\lambda),P(\mu)),N). 
$$
\end{prop}

\begin{proof}
By Theorem \ref{thm:relativeMorita}, we can identify $\mc{A} / \mc{B}$ with $\Lmod{E}$ such that 
$$
\pi(M) = \bigoplus_{\lambda \in X} \Hom_{\mc{A}}(P(\lambda),M).
$$
First we note that for $M \in \mc{A}$ and $N \in \Lmod{E}$, there are finitely many $\lambda_1, \ds, \lambda_k$ such that 
$$
e_{\mu} \cdot N = 0, \quad \textrm{and} \quad \Hom_{\mc{A}}(P(\mu),M) =  0 \quad \forall \ \mu \in X \smallsetminus \{ \lambda_1, \ds, \lambda_k \}. 
$$
Then
$$
\pi^*(N) = \bigoplus_{i = 1}^k  P(\lambda_i) \otimes_E N, \quad \pi(M) = \bigoplus_{i = 1}^k \Hom_{\mc{A}}(P(\lambda_i),M),
$$ 
and hence 
\begin{align*}
\Hom_{\mc{A}}(\pi^*(N),M) & = \Hom_{\mc{A}}\left( \bigoplus_{i = 1}^k  P(\lambda_i) \otimes_E N, M \right) \\
 & = \Hom_{E}\left( N, \bigoplus_{i = 1}^k \Hom_{\mc{A}}(P(\lambda_i),M) \right) \\
 & = \Hom_E(N,\pi(M)). 
\end{align*}
Similarly, there are only finitely many $\mu_1, \ds, \mu_{\ell} \in \Irr \mc{A}$ such that 
$$
\bigoplus_{i = 1}^k \bigoplus_{\mu \in \Irr \mc{A}} \Hom_{\mc{A}}(P(\lambda_i),P(\mu)) = \bigoplus_{i =1}^k \bigoplus_{j = 1}^{\ell} \Hom_{\mc{A}}(P(\lambda_i),P(\mu_j)).
$$
Therefore, then 
\begin{align*}
\Hom_{\mc{A}}(M, \pi^!(N)) & = \Hom_{\mc{A}}(M, \bigoplus_{i,j} \Hom_E(\Hom_{\mc{A}}(P(\lambda_i),P(\mu_j)), N)) \\
 & = \Hom_{E}(\bigoplus_{i,j} \Hom_{\mc{A}}(P(\lambda_i),P(\mu_j)) \otimes_R M, N) \\
 & = \Hom_E(\pi(M),N),
\end{align*}
where we have used Lemma \ref{lem:adjointiso} in the final equality.
\end{proof}

%\bibliography{biblo}{}

{\small

\bibliographystyle{abbrv}

\def\cprime{$'$} \def\cprime{$'$} \def\cprime{$'$} \def\cprime{$'$}
\def\cprime{$'$} \def\cprime{$'$} \def\cprime{$'$} \def\cprime{$'$}
\def\cprime{$'$} \def\cprime{$'$} \def\cprime{$'$} \def\cprime{$'$}
\def\cprime{$'$} \def\cprime{$'$} \def\cprime{$'$} \def\cprime{$'$}
\def\cprime{$'$} \def\cprime{$'$} \def\cprime{$'$} \def\cprime{$'$}
\def\cprime{$'$} \def\cprime{$'$} \def\cprime{$'$} \def\cprime{$'$}
\def\cprime{$'$} \def\cprime{$'$}

}
\end{document}

%% file: gwyns_definitions.tex
\newcommand{\bdm}{\begin{displaymath}}
\newcommand{\edm}{\end{displaymath}}
\newcommand{\bthm}{\begin{thm}}
\newcommand{\ethm}{\end{thm}}
\newcommand{\blem}{\begin{lem}}
\newcommand{\elem}{\end{lem}}
\newcommand{\bcor}{\begin{cor}}
\newcommand{\ecor}{\end{cor}}
\newcommand{\beq}{\begin{equation}\label}
\newcommand{\eeq}{\end{equation}}
\newcommand{\bprop}{\begin{prop}}
\newcommand{\eprop}{\end{prop}}
\newcommand{\bdefn}{\begin{defn}}
\newcommand{\edefn}{\end{defn}}

\newcommand{\Z}{\mathbb{Z}}
\newcommand{\N}{\mathbb{N}}

\let\C\undefined
\newcommand{\C}{\mathbb{C}}

\newcommand{\rad}{\mbox{\textrm{rad}}\,}

\let\mc\mathcal
\let\mf\mathfrak

\DeclareMathOperator{\Ind}{Ind}

\renewcommand{\mod}{\textrm{mod}}

\renewcommand{\H}{\mathsf{H}}
\newcommand{\s}{\mathfrak{S}}

\DeclareMathOperator{\Supp}{\mathrm{Supp}}
\newcommand{\Lmod}[1]{#1\text{-}{\mathsf{mod}}}

\newcommand{\op}{\mathrm{op}}

\newcommand{\idot}{{\:\raisebox{2pt}{\text{\circle*{1.5}}}}}

\newcommand{\ds}{\dots}